\documentclass[11pt, oneside]{amsart}   	
\usepackage[
backend=bibtex,
style=numeric,
backref=true,
]{biblatex}
\usepackage[margin = 0.75in]{geometry}                		
\usepackage{graphicx}				

\DeclareFontEncoding{FMS}{}{}
\DeclareFontSubstitution{FMS}{futm}{m}{n}
\DeclareFontEncoding{FMX}{}{}
\DeclareFontSubstitution{FMX}{futm}{m}{n}
\DeclareSymbolFont{fouriersymbols}{FMS}{futm}{m}{n}
\DeclareSymbolFont{fourierlargesymbols}{FMX}{futm}{m}{n}
\DeclareMathDelimiter{\VERT}{\mathord}{fouriersymbols}{152}{fourierlargesymbols}{147}

\usepackage[english]{babel}
\usepackage{amsthm}
\usepackage{bbm}	
\usepackage{mathtools}

\usepackage[
colorlinks=true]{hyperref} 
\usepackage{ifthen}
\usepackage{tikz}
\usetikzlibrary{decorations.pathreplacing}
\usepackage{tikz-cd}

\allowdisplaybreaks

\usepackage{blindtext}
\usepackage{scrextend}
\addtokomafont{labelinglabel}{\sffamily}

\usepackage{amssymb}
\usepackage{amsmath}
\newcommand{\R}{\mathbb{R}}
\newcommand{\ep}{\varepsilon}
\newcommand{\K}{\mathcal{K}}

\newcommand{\C}{\mathbb{C}}
\newcommand{\E}{\mathbb{E}}

\newcommand{\Z}{\mathbb{Z}}
\newcommand{\N}{\mathbb{N}}
\newcommand{\real}{\operatorname{Re}}

\newcommand{\norm}[1]{\left\lVert#1\right\rVert}

\newcommand{\floor}[1]{\lfloor#1\rfloor}

\newtheorem{theorem}{Theorem}[section]
\newtheorem{corollary}[theorem]{Corollary}
\newtheorem{lemma}[theorem]{Lemma}

\theoremstyle{definition}

\newtheorem{remark}[theorem]{Remark}

\title{More Stability and Convergence results for higher-order Wiener-Wintner systems}
\author{Jacob Folks}							

\begin{document}

\begin{abstract}
    \textit{Higher-order Wiener-Wintner averages} were constructed by Assani, Folks, and Moore \cite{assani2024higherorderwienerwintnersystems} to quantitatively control multiple recurrence averages. Systems in which these averages converge at a polynomial rate for a sufficiently large subset are termed \textit{higher-order Wiener-Wintner systems of power type}, in which properties like pointwise convergence of multiple recurrence averages and multiple return times averages has been shown. 

    We establish that these higher-order Wiener-Wintner averages satisfy a type of sublinearity, and that they bound conditional expectations and products, which transfers to improved stability results of higher-order Wiener-Wintner systems under sums, factors, and products. We also establish more general convergence results for such systems, which include a polynomial return times theorem and convergence of the multilinear one-side ergodic Hilbert transform with polynomial phase.
\end{abstract}
\maketitle

\section{Introduction}

\subsection{Background}
For a dynamical system $(X, \mathcal{F}, \mu, T)$ and $J\in \N$, the \textit{$J$th order multiple recurrence averages} have the form
\begin{align}\label{F:mult rec}
\frac{1}{N} \sum_{n=1}^N \prod_{j=1}^J f_j \circ T^{jn}
\end{align}
for $f_1, \dots, f_J \in L^\infty(\mu)$. These averages first appeared in Furstenberg's proof of Szemerédi's theorem \cite{Fur77} and their convergence properties have been long since studied. Notably, norm convergence in $L^2$ was established independently by Host and Kra \cite{hostkra} and Ziegler \cite{ziegler07}. Pointwise convergence results have been established for some specific cases, including some weakly mixing systems by Assani \cite{assani98} and distal systems by Huang, Shao, and Ye \cite{Ye}. 

In his landmark work \cite{Bourgain}, Bourgain established pointwise converges for the $J=2$, or double recurrence case
\begin{align}\label{F:double rec}
\frac{1}{N}\sum_{n=1}^N f_1\circ T^n \cdot f_2 \circ T^{2n} \, .
\end{align}
He does so, in part, by tying the above average double recurrence average to the weighted averages
\begin{align}\label{F:WW av}
\frac{1}{N}\sum_{n=1}^N e^{2\pi i n t} f\circ T^n \, .
\end{align}
The classical \textit{Wiener-Wintner Theorem} \cite{WW} states that for every $f\in L^1(\mu)$, there exists a set $X_f$ of full measure over which the above converges for all $t\in \R$. Moreover, the \textit{Uniform Wiener-Wintner theorem} (also proved by Bourgain in \cite{Bourgain}) states that uniform convergence of these averages over $t$ characterizes the Kronecker factor $\mathcal{K}$ (the closed span of the eigenfunctions), in that
$$f\in L^2(\K)^\perp \quad \iff \quad \lim_N \sup_t \left| \frac{1}{N}\sum_{n=1}^N e^{2\pi i n t} f(T^n x) \right|= 0 \text{ for almost all }x \, .$$
For the double recurrence averages (\ref{F:double rec}), the case in which either $f_1$ or $f_2$ is in $L^2(\mathcal{K})$ reduces to the classical Wiener-Wintner theorem. For the remaining case, Bourgain shows that the norm behavior of (\ref{F:double rec}) is controlled by the norm of the uniform Wiener-Wintner averages of either $f_1$ or $f_2$, and hence converge to zero for $f_1$ or $f_2$ in  $L^2(\mathcal{K})^\perp$. Tools from harmonic analysis are then used to pass from norm to pointwise convergence, which comprise the bulk of Bourgain's paper.

This portion of the argument can be greatly simplified, as shown by Assani \cite{assani}, under the added condition that the uniform Wiener-Winter averages converge polynomially in norm, or for some $\alpha>0$ and a dense set of $f\in L^2(\mathcal{K})^\perp$ there exists constants $C_f$ such that
$$\left\Vert\sup_t \left| \frac{1}{N}\sum_{n=1}^N e^{2\pi i n t} f \circ T^n \right|\right\Vert_2 \leq \frac{C_f}{N^\alpha}$$
holds for all $N$. Such a system is called a \textit{Wiener-Winter system of power type $\alpha > 0$} (we note that for the dense set, the powers $\alpha$ do not need to be uniform for the convergence to extend by density). Moreover, examples and stability conditions of such systems are established, and further convergence results such as the one-sided ergodic Hilbert transform are shown \cite{assani04}. 

In \cite{assani2024higherorderwienerwintnersystems}, Assani, Folks, and Moore inductively lift this line of reasoning to higher-order multiple recurrence averages. The norm convergence arguments of Host/Kra and Ziegler establish the characteristic factors $\mathcal{Z}_{J-1}$ of the multiple recurrence averages (\ref{F:mult rec}) (known as \textit{Host-Kra-Ziegler factors}) which have the structure of a $J$-step nilsystem inside which pointwise convergence of multiple recurrence averages has been shown by Leibman \cite{leibman05a}. Assani, Folks, and Moore \cite{assani2024higherorderwienerwintnersystems} construct a \textit{$J$-th order Wiener-Wintner average} (here denoted $W_N^J(f)$) whose pointwise limit characterizes $\mathcal{Z}_{J-1}$ and controls the multiple recurrence averages in norm analogously to the classical Wiener-Wintner average in the Bourgain argument. With the added condition that these higher-order Wiener-Wintner averages converge to zero polynomially for a dense subset of $L^2(\mathcal{Z}_{J-1})^\perp$, we obtain pointwise convergence of multiple recurrence, and also a multiterm returns times theorem (a survey of such return times theorems in given by Assani and Presser \cite{AssaniPresser+2014+19+58}). Such a system is called a \textit{$J$-th order Wiener-Wintner dynamical system}, and examples are given which dip outside those covered by the previous convergence results for multiple recurrence averages \cite{assani2024higherorderwienerwintnersystems}
(The construction of these averages and the exact bounds they satisfy are discussed in detail in the following section \S \ref{S:background}). This example arises as the result of a stability result, in that the product of a higher-order Wiener-Wintner system with a K automorphism is a higher order Wiener-Wintner system.

These higher order Wiener-Wintner averages introduce technical complications not present in the first-order case. Notably, the averages are no longer sublinear, and do not immediately form a vector space. This problem is overcome, in certain cases, by Assani, Folks, and Moore  \cite{assani2024higherorderwienerwintnersystems} 
by proving a stronger condition, termed ``multilinearity concerns". While this condition is straightforward to verify in specific examples, it makes more general stability results more challenging. Notably, while it is easy to verify that higher-order Wiener-Wintner systems are preserved under isomorphisms, it does not obviously follow from the definitions that higher-order Wiener-Wintner systems are preserved even under factors.

\subsection{Goals and Overview} In this paper, we wish to establish stability properties of the higher-order Wiener-Wintner averages under sums, factors, and products, and establish convergence results for the one-sided multilinear ergodic Hilbert transform for higher-order Wiener-Wintner functions. We do so in part by establishing an equivalence between the decay rates of the higher order Wiener-Winter average $W_N^J(f)$ and the \textit{uniform multiple recurrence average}, denoted by
\begin{align*}
M_N^J(f) := \sup_{\overset{g_j\in L^\infty(\mu)}{\max_{j=1, \dots, J} \Vert g_j \Vert_\infty \leq 1}}\left\Vert \frac{1}{N}\sum_{n=1}^N \prod_{j=1}^{J} g_j \circ T^{jn} \cdot f\circ T^{(J+1)n}\right \Vert_2 \, .
\end{align*}
The ``Bourgain bound" from \cite{assani2024higherorderwienerwintnersystems} establishes half of this equivalence, and the other is shown here. Passing back and forth between $W_N^J(f)$ and $M_N^J(f)$, we will be able to show bounds that establish the desired stability properties. We note that since most methods used are primarily combinatorial, the constants obtained are ``universal", in that they never pick up any dependence on the dynamical system itself, as seen in the Bourgain bound.

We detail the exact construction of the higher-order Wiener-Wintner averages and the bounds that they satisfy on multiple recurrence in section \S \ref{S:BB and av}, also providing some updated notation. We also provide precise definitions for Higher-order Wiener-Wintner functions and systems in \S \ref{S:Systems}. We introduce a special asymptotic notation $f \precsim g$. to make rigorous our desired notion of equivalence between decay rates.

We begin \S \ref{S:Stability} by establishing the equivalence between the averages $W_N^J(f)$ and $M_N^J(f)$ through what is termed a ``reverse Bourgain bound" (Theorem \ref{T:reverse BB, simplest case}). This lets us immediately conclude a sublinearity property in section \S \ref{S:sublinear}. We show in \S \ref{S:Off-diag} that the multiple recurrence averages $M_N^J(f)$ further bound the ``off-diagonal" Wiener-Wintner averages, as described by the Multilinearity concerns of \cite{assani2024higherorderwienerwintnersystems}. By our previous equivalence, this shows that $W_N^J(f)$ controls its own off-diagonal terms, and the multilinearity concerns are always satisfied for a weaker power. In section \S \ref{S:conditioning}, we show that the higher-order Wiener-Wintner averages of a conditional expectation form one of these off-diagonal averages, and are hence controlled by $W_N^J(f)$. We are immediately able to conclude that a factor of a higher-order Wiener-Wintner system is itself a higher-order Wiener-Wintner system.

We introduce a corresponding notion of ``higher-order weak Wiener-Wintner averages" in \S \ref{S:WWW}, which we show bounds the ``strong" average at the cost of order (Theorem \ref{T:converse bound for weak WW averages}). Since these weak averages interact well with products, we raise this argument inductively and are able to bound $W_N^J(f \otimes g)$ by the individual averages, at the cost of raising the order. This lets us extend the $J$-th order Wiener-Wintner stability result under products with a K automorphism \cite{assani2024higherorderwienerwintnersystems} to stability under products with any weakly mixing $J+1$-th order Wiener-Wintner system. We also show that the equivalence between $W_N^J(f)$ and $M_N^J(f)$ is also inherited by many alternative constructions of the Wiener-Wintner averages. We are able to use this in \S \ref{S:Alt const} to answer a natural question about the construction of $W_N^J(f)$, showing that most other constructions are equivalent.

In \S \ref{S:Conv results}, we show more convergence results. We start by showing that the averages $W_N^{k+J-1}(f)$ control a polynomial Wiener-Wintner multiple recurrence average of the form
\begin{align}\label{F:ex of polynomial norm}
\left\Vert \sup_{t_1, \dots, t_k} \left|\frac{1}{N}\sum_{n=1}^N e^{2\pi i p_{k, t}(n)} \prod_{j=1}^J f_j \circ T^{a_j n} \right| \right\Vert_2
\end{align}
where $p_{k, t} = t_1n + t_2n^2 + \dots + t_kn^k$, naturally generalizing the arguments from \cite{assani2024higherorderwienerwintnersystems} and \cite{AssaniMoore}. This immediately lets us conclude a polynomial return times theorem for multiple recurrence
\begin{align*}
\frac{1}{N}\sum_{n=1}^N g(S^{P(n)} y) \prod_{j=1}^J f_{j}(T^{a_j n} x)
\end{align*}
as shown in \S \ref{S:Polynomial bounds}. Following this, we use polynomial decay to deduce in \S\ref{S:Hilb} that all convergence results thus far established for higher-order WW systems have corresponding ergodic Hilbert transform analogues; for example, polynomial control on (\ref{F:ex of polynomial norm}) leads to almost everywhere convergence of the averages
$$\sum_{n=1}^\infty \frac{e^{2\pi i \left(t_1 n + t_2 n^2 + \dots + t_k n^k\right)}}{n^\sigma} \, \prod_{j=1}^{J}f_j(T^{a_j n} x)$$
for some $\sigma < 1$ to a continuous function in $t_1, \dots, t_k$, and that the multilinear multiterm return times theorem of the form 
$$\sum_{n=1}^\infty \frac{\prod_{k=1}^K g_{k} (S^{kn} y) \prod_{j=1}^J f_{j} (T^{jn} x)}{n^\sigma}$$
converges almost everywhere. Such results do not readily extend by density, but they illustrate that higher-order Wiener-Wintner systems are a class in which convergence holds on a dense subset of $L^2(\mathcal{Z}_{k+J-1})^\perp$. Moreover, since there always exists a dense set of higher-order Wiener-Wintner functions of any order in the orthogonal complement of the Pinsker algebra of any system \cite{assani2024higherorderwienerwintnersystems}, this established that these return times Hilbert transform results always hold on a dense subset of $L^2(\mathcal{P})^\perp$.

Finally, the appendix \ref{A:general case} contains a full proof of the reverse Bourgain bound established in \S\ref{s:RBB}.

\subsection{Acknowledgments} This article is part of the author's PhD thesis, done under Prof. Idris Assani.

\section{Notation and Preliminaries}\label{S:background}

\subsection{Notation and conventions}\label{S:notation} Throughout this paper, quadruples $(X, \mathcal{F}, \mu, T)$ will denote \textit{(measure-preserving) dynamical systems} or simply \textit{systems} consisting of a probability space $(X, \mathcal{F}, \mu)$ and measurable map $T:X \to T$ satisfying $\mu(T^{-1}A) = \mu(A)$ for all $A\in \mathcal{F}$. Such a system is \textit{ergodic}
if for all $A\in \mathcal{F}$, we have $T^{-1}A = A$ implies that $\mu(A) = 0$ or $1$. Moreover, for any $p\in [1, \infty]$, the $L^p$ norm will be denoted $\Vert \cdot \Vert_p$ or $\Vert \cdot \Vert_{L^p(\mu)}$ if we wish to indicate the particular measure for clarification. If $\mathcal{B}\subset \mathcal{F}$ is a sub $\sigma$-algebra, we may use $L^p(\mathcal{B})$ to denote the collection of all $\mathcal{B}$-measureable functions in $L^p(\mu)$. Powers of any map $T^n$ will denote composition of the map $T$ with itself $n$ times.

For $N\in \N$, we let $[N] = \{1, \dots, N\}$. For any $k\in \N$, we denote $V_k$ to be the set $\{0, 1\}^k$. For any $\eta\in V_k$, we let $|\eta|$ count the number of $1$'s appearing. Likewise, for any $h\in \N^k$, we may let $\eta \cdot h$ denote the usual componentwise dot product $\sum_{n=1}^k h_k \eta_k$. We may let $c:\C \to \C$ denote complex conjugation, such that $c^n z$ is equal to $z$ for even $n$ and $\overline{z}$ for odd $n$, and we denote the real component of $z$ by $\Re(z)$. We also make use of the floor function $\floor{ \cdot }$ to be the greatest integer below a given input. Cardinality of a set $A$ is denoted by $\# A $.

Following standard notation, we use subscripts to denote the dependence of certain constants. However, in many cases a constant $C$ will depend on a finite collection of quantities $a_1, \dots, a_k$ of variable length. In such cases, we use an indexless subscript $C_{a}$ to denote that the constant $C$ depends on \textit{all} of the quantities $a_1, \dots, a_k$. Hence, this quantity will implicitly depend on $k$ as the amount of $a_i$'s, but such dependence is not explicitly denoted. In the case that the constant $C$ depends on $k$ apart from the terms $a_1, \dots, a_k$, we denote this constant as $C_{a, k}$. We remark that almost all examples of interest fall into this latter case, in which the constant $C_{a, k}$ depends explicitly on both $a_1, \dots, a_k$ and $k$.

Finally, we will abbreviate ``Wiener-Wintner" as ``WW".

\subsection{Bourgain bounds and Higher-order Wiener-Wintner (WW) averages}\label{S:BB and av}
In his proof of pointwise convergence for double recurrence \cite{Bourgain}, Bourgain implicity uses a bound of the following form on any bounded functions $f_1$ and $f_2$:
\begin{align}\label{bourg}
\left\Vert \frac{1}{N}\sum_{n=1}^N f_1 \circ T^n \cdot f_2 \circ T^{2n} \right\Vert_2 \leq C\left\Vert \sup_{t}\left| \frac{1}{N}\sum_{n=1}^N e^{2\pi i n t} f_1\circ T^n \right| \right\Vert_1^{2/3}
\end{align}
holds for some constant $C$ and sufficiently large $N$. This bound is made explicit, in the above form, by Assani \cite{assani}. As previously stated, this bound lets us immediately conclude norm convergence for double recurrence, and Bourgain uses approaches from harmonic analysis to establish pointwise convergence.

In \cite{assani2024higherorderwienerwintnersystems}, this bound is inductively extended to higher-order recurrence. To begin, for $f\in L^\infty(\mu)$, we denote the \textit{first-order WW average of $f$} to be
\begin{align*}
W^1_N(f) &= \left\Vert \sup_{t}\left| \frac{1}{N}\sum_{n=1}^N e^{2\pi i n t} f\circ T^n \right| \right\Vert_2^{2/3} \, .
\end{align*}
By using the ideas from the proof of (\ref{bourg}), this bound is modified to hold for all $N$ and remove dependence on $f_1$ and $f_2$. Specifically, it is shown that for any $a_1, a_2\in \mathbb{Z}$, distinct and nonzero, there exists a constant $C_{a, 1}$ such that for all $N\in \mathbb{N}$ and $f_1$ and $f_2$ with $\max_{j=1, 2} \Vert f_j \Vert_{\infty} \leq 1$, we have 
\begin{align*}
\left\Vert \frac{1}{N}\sum_{n=1}^N f_1 \circ T^{a_1n} \cdot f_2 \circ T^{a_2n} \right\Vert_2 \leq C_{a, 1} \left( \frac{1}{N} +\left\Vert\sup_t \left|\frac{1}{N}\sum_{n=1}^N e^{2\pi i n t} f\circ T^n \right| \right\Vert_1^{2/3}\right) \, .
\end{align*}
The choice of the subscript $1$ for $C_{a, 1}$ is done to fit the later convention given by the estimate (\ref{L:BB}) below. We remark, as indicated in the previous section \ref{S:notation}, that this constant $C_{a, 1}$ depends on both $a_1$ and $a_2$. An exact value of this is worked out in the appendix of \cite{assani2024higherorderwienerwintnersystems}, and it grows polynomially in $|a_1|$ and $|a_2|$. Of particular interest, this constant does not depend on the dynamical system itself.    

Since we are working on a probability space, we can bound $L^1$ norms by $L^2$, and write the above as
\begin{align*}
\left\Vert \frac{1}{N}\sum_{n=1}^N f_1 \circ T^{a_1n} \cdot f_2 \circ T^{a_2n} \right\Vert_2 \leq C_{a, 1} \left( \frac{1}{N} + W^1_N(f_1)\right)
\end{align*}
Because of the asymmetry in norms ($L^2$ on the left and $L^1$ on the right), we could have balanced the estimate above in another way: by lowering the left hand average to $L^1$. Throughout this paper, unless otherwise noted, we take $W_N^1$ to be defined as above with $L^2$ norms and phrase all bounds in terms of $L^2$. However, in most estimates the $2$-norms can be replaced with $1$-norms at no extra cost.

For $f\in L^\infty(\mu)$ and $k\in \mathbb{N}$, we inductively define the \textit{$k$-th order WW average of $f$} to satisfy
\begin{align}\label{F:inductive construction of WW averages}
W^{k}_N(f) = \frac{1}{\floor{\sqrt{N}}}\sum_{h=1}^{\floor{\sqrt{N}}} W^{k-1}_N \left( f \cdot \overline{f \circ T^h} \right) \, .
\end{align}
It follows inductively 
\begin{align*}
W^{k}_N(f) &= \frac{1}{\floor{\sqrt{N}}^{k-1}}\sum_{h\in [\floor{\sqrt{N}}]^{k-1}} W^{1}_N \left( \prod_{\eta\in V_{k-1}} c^{|\eta|} f\circ T^{\eta \cdot h} \right) \\
&=\frac{1}{\floor{\sqrt{N}}^{k-1}}\sum_{h\in [\floor{\sqrt{N}}]^{k-1}}  \left\Vert \sup_t \left| \frac{1}{N} \sum_{n=1}^N e^{2\pi i n t}\left[\prod_{\eta\in V_{k-1}} c^{|\eta|} f\circ T^{\eta \cdot h} \right]\circ T^n\right| \right\Vert_2^{2/3}
\end{align*}
using some of the previously established notation and conventions.

It is shown that these averages satisfy a \textit{$k$-th order Bourgain bound}: that for any integers $a_1, \dots, a_{k+1}$, all distinct and nonzero, there exists a constant $C_{a, k}>0$ and $N_{a, k}\in \mathbb{N}$ so that for all 
$f_1, \dots, f_{k+1} \in L^\infty(\mu)$ with  $\max_{j=1, \dots, k+1} \Vert f_j \Vert_{\infty} \leq 1$, we have
\begin{align}\label{L:BB}
\left\Vert \frac{1}{N}\sum_{n=1}^N \prod_{j=1}^{k+1} f_j \circ T^{a_j n} \right \Vert_2 \leq C_{a, k} \left( \frac{1}{N^{1/2^k}} + \left[W_N^k (f_1)\right]^{1/2^{k-1}} \right)
\end{align}
for all $N \geq N_{a, k}$. An analogous bound holds if the 2-norms are replaced with 1-norms. As remarked previously, this constant $C_{a, k}$ only depends on all of $a_1, \dots, a_k$, and $k$ itself. Based on the argumentation of \cite{assani2024higherorderwienerwintnersystems}, we see that this constant depends exponentially in $k$, and for a fixed $k$ grows polynomially in each $a_1, \dots, a_k$.

As in the lower order case, it is shown \cite{assani2024higherorderwienerwintnersystems} that the averages $W_N^k$ distinguish in limit the correct characteristic factor for multiple recurrence. The \textit{Host-Kra-Ziegler factors}, denoted $\mathcal{Z}_{k}$, can be defined using the following inductive \textit{Gowers-Host-Kra seminorm} construction: 
\begin{align}\label{F:HKZ seminorm construction}
\begin{split}
 \VERT f \VERT_2^4 &= \lim_H \frac{1}{H} \sum_{h=1}^H \left| \int f \cdot \overline{f\circ T^h  } \, d\mu \right|^2 \\
\VERT f \VERT_3^8 &= \lim_H \frac{1}{H} \sum_{h=1}^H \left\VERT f \cdot \overline{f \circ T^h }\right\VERT_2^4 \\
\VERT f \VERT_{k}^{2^k} &= \lim_H \frac{1}{H} \sum_{h=1}^H \left\VERT f\cdot \overline{f\circ T^h}\right\VERT_{k-1}^{2^{k-1}} 
\end{split}
\end{align}
such that $f\in L^2(\mathcal{Z}_{k})^\perp$ if and only if $\VERT f \VERT_{k+1} = 0$. In \cite{assani}, it is shown that there exists constants $C_k$ such that
\begin{align}\label{E:bound WW by HKZ seminorm}
\limsup_N W_{N}^k(f) \leq C_k \VERT f \VERT_{k+1}^{2/3}
\end{align}
and that a bounded function $f$ has $f\in L^2(\mathcal{Z}_k)^\perp$ if and only if the pointwise Wiener-Wintner average ($W_N^k(f)$ with the integrals removed and $h$ sums extended to $N$) converges to zero for almost all $x\in X$.

\begin{remark}\label{R:is ergodicity needed}
We note that this seminorm construction only works in the case that $(X, \mathcal{F}, \mu, T)$ is ergodic. However, the proof of the Bourgain bound (\ref{L:BB}) is primarily combinatorial, and does not require ergodicity. Hence, in the non-ergodic case, the averages $W_N^k(f)$ are still defined and still bound multiple recurrence in norm, but the limiting behavior of $W_N^k(f)$ cannot be analyzed by the seminorms above. 

Throughout this article, we remark that almost every bound does not require ergodicity. It is only those that directly reference the Gowers-Host-Kra seminorms, such as (\ref{E:bound WW by HKZ seminorm}), that may fail to hold if we no longer require the system to be ergodic. Such results are only interested in the limiting behavior of the WW averages, and for a fixed value of $N$ all of the bounds shown in  \S \ref{S:Stability} and \S \ref{S:Conv results} hold regardless of whether the system is ergodic.
\end{remark}

\subsection{Higher-order Wiener-Wintner (WW) functions and systems}\label{S:Systems}

On the space of functions $\R^\N$, for $\alpha > 0$ denote the following subset by
\begin{align*}
\text{Poly}(\N, \alpha) &= \{r:\N \to \R \, :\, \exists \, C>0 \text{ such that }r(N) \leq N^{-\alpha} \text{ for all }N\} \\
&= \{r:\N \to \R \, :\, r = O(N^{-\alpha}) \} \, .
\end{align*}
We denote the subset of \textit{functions of polynomial decay} as
\begin{align*}
\text{Poly}(\N) &= \{r:\N \to \R \, :\, \exists \,\alpha > 0, C>0 \text{ such that }r(N) \leq N^{-\alpha} \text{ for all }N\} \\
&= \{r:\N \to \R \, :\, \exists \,\alpha > 0\text{ such that }r = O(N^{-\alpha}) \} \\
&= \bigcup_{\alpha > 0} \text{Poly}(\N, \alpha) \, .
\end{align*}

Classically, good decay rates may provide summability which in certain cases can extend norm convergence to pointwise convergence. This was used by Assani \cite{assani} to prove pointwise convergence for double recurrence such that one function had polynomial decay on its WW average, termed a \textit{Wiener-Wintner function of power type}.

Following \cite{assani2024higherorderwienerwintnersystems}, we say that $f\in L^\infty(\mu)$ is a \textit{$k$-th order WW function of power type $\alpha$} if $W_N^k(f) \in \text{Poly} (\N, \alpha)$, and $f\in L^\infty(\mu)$ is a \textit{$k$-th order WW function of power type} if $W_N^k(f) \in \text{Poly}(\N)$. For such functions, pointwise convergence of multiple recurrence averages follows by the same summability argument. Hence we say that an ergodic dynamical system $(X, \mathcal{F}, \mu, T)$ is a \textit{$k$-th order WW system of power type $\alpha$} if there exists a dense subset of $k$-th order WW functions of power type $\alpha$ inside $L^2(\mathcal{Z}_k)^\perp$. Similarly, we define a \textit{$k$-th order WW system of power type} such that the WW functions of power type are dense in $L^2(\mathcal{Z}_k)^\perp$. In such systems, pointwise convergence of multiple recurrence and multiple return times convergence theorems are established by a classical summation argument.

Following the remark \ref{R:is ergodicity needed}, we note that the definition of WW \textit{functions} does not require the system to be ergodic. The definition of WW \textit{systems}, however, is based on the seminorm construction of Host-Kra-Ziegler factors, and does require ergodicity. This is of particular importance for some later results concerning product systems; while such systems may not be ergodic, it is still well phrased to consider WW functions.

\subsection{Special asymptotic notation}\label{S:special asymp}
Since many results throughout this paper focus on the transfer of decay rates between certain averages, we make use of the following notation for simplification. For functions $f, g:\N \to \R^{\geq 0}$, we say that $f \precsim g$ if there exists constants $0 < \alpha, \beta, \gamma \leq 1$ and $C >0$ and $N_0\in \mathbb{N}$, and a nondecreasing function $\phi: \N \to \N$ satisfying $\phi(N) \geq N^\beta$ such that
\begin{align}\label{E:def of precsim notation}
f(N) \leq C\left(\frac{1}{N^\alpha} + g\Big( \phi(N)\Big)^\gamma\right)
\end{align}
for all $N > N_0$. Following standard convention, we also use $f(N) \precsim_{a, b, \dots} g(N)$ to denote that the constants depend on variables $a, b, \dots$, and we use $f \approx g$ to denote that $f \precsim g$ and $g \precsim f$.

While similar in purpose to the standard asymptotic notation $f\lesssim g$, the constants $\alpha, \beta$ and $\gamma$ in $f\precsim g$ generalize the scope considerably and create a weaker notion of comparison. However, in the case that $g$ tends to zero, it follows that any $f$ with $f \precsim g$ will also tends to zero at ``approximately" the same rate as $g$, up to scaling by powers on the inside and outside of $g$ and a polynomial remainder term. This is exactly the relationship that arose in the Bourgain bound (\ref{L:BB}), which can be written in this notation as
$$\left\Vert \frac{1}{N}\sum_{n=1}^N \prod_{j=1}^{k+1} f_j \circ T^{a_j n} \right \Vert_2 \precsim_{a, k} W_N^k (f_1)\, .$$
Notably, we see that the flexibility afforded by the inclusion of the function $\phi$ in the formula (\ref{E:def of precsim notation}) is not needed to write the Bourgain bound, as $\phi(N) = N$ suffices in the above statement. Later bounds will have the same form, but with polynomial scaling on the inside, such as the following:
$$f(N) \leq C\left(\frac{1}{N^\alpha} + g\left( \floor{N^\beta}\right)^\gamma\right)$$
The notation $f\precsim g$ is defined by (\ref{E:def of precsim notation}) rather than the above because the above relation is not transitive, as the nested floor functions do not necessarily simplify. Throughout all following results, the function $\phi$ can be thought of as behaving like $\floor{N^\beta}$ for some $0 < \beta \leq 1$; the inclusion of the general function $\phi$ in the definition of $f\precsim g$ is specifically to ensure that the relation is transitive.

Moreover, if $f(N) = O(N^{-\alpha})$ for some $\alpha>0$, it follows that $f \precsim 0$. Since all positive functions satisfy $0 \precsim f$, we may denote functions with polynomial decay as $f \approx 0$. Hence, $f\in L^\infty (\mu)$ is a Wiener-Wintner function of power type if and only if $W_N^k (f) \approx 0$.

While this notation greatly simplifies the kinds of asymptotic comparison that will arise throughout the paper, it is complicated in practice and lacks some standard desirable properties. Specifically, the relation $f \precsim g$ is not additive, as the functions $\phi$ may misalign the comparison. However, we note in the case that $f \precsim 0$ and $g \precsim 0$, it does follow that $f + g \precsim 0$. Since we will be almost exclusively interested in the transfer of polynomial decay of functions $f \precsim 0$, the relation $f \precsim g$ behaves well enough in most every case that we use it. Because of its technical limitations, we remark that we primarily use this notation to simplify the statement of results rather than to simplify proofs, and in all cases we will include the precise statements of the estimates proved. However, using this notation informally can cleanly explain the structure of many of the following arguments.

\subsection{Tools}\label{S:tools}
Here, we collect some essential estimates. Primarily, we make extensive use of the classical
\textit{Van der Corput inequality} (cf. \cite{KN74}), under the following statement:
		
		\begin{lemma}[Van der Corput's estimate]\label{vdc-lem}
			Let $\{v_n\}$ be a sequence of complex numbers. For any integers $1 \leq H \leq N$, we have
			\begin{align*}
				& \left| \frac{1}{N} \sum_{n=0}^{N-1} v_n \right|^2
				\leq \frac{N+H}{N^2(H+1)} \sum_{n=0}^{N-1} |v_n|^2 + \frac{2(N+H)}{N^2(H+1)^2} \sum_{h=1}^H(H+1-h) \real \left(\sum_{n=0}^{N-h-1} \overline{v_{n+h}} v_n \right).
			\end{align*}
		\end{lemma}
		
In the case that $v_n = u_ne^{2\pi int}$ for some sequence of complex numbers $\{u_n\}$, the rotating weigh $e^{2\pi i n t}$ is lost under the application of the Van der Corput identity. Hence, we may take a supremum to achieve the following variation: For every $N \in \N$ and $1 \leq H \leq N-1$, we have
		\begin{equation}\label{vdc-1}
			\sup_t \left| \frac{1}{N} \sum_{n=0}^{N-1} u_ne^{2\pi int} \right|^2
			\leq \frac{2}{N(H+1)} \sum_{n=0}^{N-1} |u_n|^2 + \frac{4}{H+1} \sum_{h=1}^H \left| \frac{1}{N}\sum_{n=0}^{N-h-1} \overline{u_{n+h}} u_n \right|.
		\end{equation}

A straightforward but useful computation (which mirrors the $H = N$ case of the Van der Corput inequality in the context of dynamical systems) is presented here as a lemma:
\begin{lemma}\label{L:vdc for systems}
Let $(X, \mathcal{F}, \mu, T)$ be a dynamical system, and $f\in L^\infty(\mu)$ have $\Vert f\Vert_{\infty} \leq 1$. Then for all $N$, we have
\begin{align*}
\int \left| \frac{1}{N} \sum_{n=1}^N f\circ T^{n} \right|^2 d\mu &\leq \frac{2}{N} \sum_{n = 0}^{N-1} \left(\frac{N-n}{N}\right)\Re \left[\int f \cdot \overline{f\circ T^{n}} \, d\mu \right] \, .
\end{align*}
\end{lemma}

\begin{proof}
Factoring out, we observe
\begin{align*}
\int \left| \frac{1}{N} \sum_{n=1}^N f\circ T^{n} \right|^2 \, d\mu &=  \int  \frac{1}{N^2} \sum_{n, m=1}^N f\circ T^n \cdot \overline{f\circ T^m} \, d\mu \\
&=  \frac{1}{N^2} \sum_{n, m=1}^N \int f \cdot \overline{f\circ T^{m-n}} \, d\mu \, .
\end{align*}
Notice that the term $m-n$ ranges between $1-N$ and $N-1$. Moreover, by measure translation the terms at $m-n$ and $n-m$ are conjugate to each other. Grouping these terms together, we obtain twice of the real component and may restrict the sum to values such that $m - n \geq 0$:
\begin{align*}
&\leq  \frac{2}{N^2} \sum_{m \geq n = 1}^N  \Re \left[\int f\cdot \overline{f\circ T^{m-n}} \, d\mu \right]\, .
\end{align*}
Notice that the factor of 2 does not appear on the terms such that $m-n =0$. However, these terms are real and positive, and can be bounded by 2 anyway, making the above indeed an inequality.

Now, the value $m-n$ ranges from $0$ to $N-1$, and we observe that for each $0 \leq j \leq N-1$, there are $N-j$ pairs of $m$ and $n$ in the desired range such that $m-n = j$. Hence, the $m-n = 1$ summand is being added $N-1$ times, the $m-n = 2$ summand is being added $N-2$ times, and so on. Combining the $n, m$ sum into one sum over $m-n$, and relabeling this variable as $n$, we obtain
\begin{align*}
&=  \frac{2}{N^2} \sum_{n = 0}^{N-1} \left(N-n\right)\Re \left[\int f \cdot \overline{f\circ T^{n}} \, d\mu \right] \, .
\end{align*}
Moving one power of $1/N$ inside the sum yields the desired inequality.
\end{proof}

We also use Hölder's inequality, particularly in the following form on averages:
\begin{lemma}[Hölder's inequality on averages]
Let $\{a_n\}_{n=1}^N$ be a finite sequence of real, nonnegative numbers. The function $A: \R \to \R$ such that
$$A_N(p) := \left(\frac{1}{N}\sum_{n=1}^N a_n^p\right)^{1/p}$$
is increasing in $p$: i.e. for $p \leq q$ we have
$$\left(\frac{1}{N}\sum_{n=1}^N a_n^p\right)^{1/p}\leq \left(\frac{1}{N}\sum_{n=1}^N a_n^q\right)^{1/q}.$$
In particular, for $0 < r < 1$, we have the following:
$$ \frac{1}{N} \sum_{n=1}^N a_n^r \leq \left( \frac{1}{N} \sum_{n=1}^N a_n\right)^r\, .$$
\end{lemma}

We also use the following maximal inequality to extend some results via density (cf. \cite[Theorem 1.8]{AssaniWW} for instance):

\begin{lemma}[Maximal inequality]\label{L:maxIneq}
Let $(X, \mathcal{F}, \mu, T)$ be a measure-preserving system, and $p \in (1, \infty)$. For every real-valued function $f \in L^p(\mu)$, we have
\[ \norm{  \sup_N \frac{1}{N} \sum_{n=1}^N f \circ T^n}_p \leq \frac{p}{p-1} \norm{f}_p \, . \]
\end{lemma}

\section{Stability under sums, factors, products}\label{S:Stability}
\subsection{A reverse Bourgain bound}\label{s:RBB}

Recall that the constants in the Bourgain bound only depend on the exponents and order of multiple recurrence. For example, fixing the exponents as $1, 2, 3, \dots$, it follows that for any $k$ there exists $C_k$ and $N_k$ such that for any $f\in L^\infty$ bounded by 1 and all functions $g_1, \dots, g_k$ with $\max_{j=1, \dots, k} \Vert g_j \Vert_\infty \leq 1$ we have
\begin{align*}
\left\Vert \frac{1}{N}\sum_{n=1}^N \prod_{j=1}^{k} g_j \circ T^{jn} \cdot f\circ T^{(k+1)n}\right \Vert_2 \leq C_{k} \left( \frac{1}{N^{1/2^k}} + \left[W_N^k (f)\right]^{1/2^{k-1}} \right)
\end{align*}
for all $N \geq N_k$. Hence, we can take a supremum over all such functions $g_j$ to see that
\begin{align*}
\sup_{\overset{g_j\in L^\infty(\mu)}{\max_{j=1, \dots, k} \Vert g_j \Vert_\infty \leq 1}}\left\Vert \frac{1}{N}\sum_{n=1}^N \prod_{j=1}^{k} g_j \circ T^{jn} \cdot f\circ T^{(k+1)n}\right \Vert_2 \leq C_{k} \left( \frac{1}{N^{1/2^k}} + \left[W_N^k (f)\right]^{1/2^{k-1}} \right)
\end{align*}
holds for sufficiently large $N$, depending only on $k$.

Let us denote the above supremum over $k+1$ multiple recurrence averages as $M^k_N(f)$, so that we may rewrite the previous bound as
\begin{align}\label{E:BB for M, W}
M_N^k(f) \leq C_{k} \left( \frac{1}{N^{1/2^k}} + \left[W_N^k (f)\right]^{1/2^{k-1}} \right)
\end{align}
for sufficiently large $N$ depending on $k$, or
$$M_N^k(f) \precsim_k W_N^k(f)$$
using the special asymptotic notation. We shall refer to $M_N^k$ as the \textit{$k$-th order Uniform multiple recurrence average}. Hence, it follows from this Bourgain bound (\ref{E:BB for M, W}) that for any $f\in L^2(\mathcal{Z}_k)^\perp$, the uniform multiple recurrence averages $M_N^k(f)$ converge to zero, and the rate of convergence can be quantitatively controlled by the WW averages $W_N^k(f)$; that is, a ``good" decay rate on the averages $W_N^k(f)$ will transfer to the averages $M_N^k(f)$. Specifically, we see that if $W_N^k(f) \in \text{Poly}(\N)$, then $M_N^k(f) \in \text{Poly}(\N)$.

In this section, we seek to reverse this transfer, and establish that good decay rates on $M_N^k(f)$ can pass back to $W_N^k(f)$. We do so by obtaining a ``reverse" Bourgain bound, in which the WW averages are controlled by the uniform multiple recurrence averages:
\begin{theorem}[Reverse Bourgain bound]\label{T:reverse BB, simplest case}
For each $k\in \N$, there exists a constant $C'_k$, such that for any $f\in L^\infty(\mu)$ with $\Vert f\Vert_\infty \leq 1$ we have 
\begin{align}\label{E:RBB for W, M}
W_N^k(f) \leq C'_k \left(\frac{1}{N^{1/24}} + \left[M_{\floor{N^{1/4}}}^k(f)\right]^{1/6} \right)
\end{align}
for all $N$. Hence, we have
$$W_N^k(f) \precsim_{k} M_N^k(f)$$
\end{theorem}
As an immediate consequence, we see that if $M_N^k(f) \in \text{Poly}(\N)$, then $W_N^k(f) \in \text{Poly}(\N)$. Together with the Bourgain bound (\ref{E:BB for M, W}), we see that $M_N^k$ and $W_N^k$ must have ``essentially" the same decay rate, in that $W_N^k(f) \approx_k M_N^k(f)$ for all bounded $f$.

Here, we prove the case $k=1$:
$$W_N^1(f) \leq C'_1\left( \frac{1}{N^{1/24}} + \left[ M^1_{\floor{N^{1/4}}} (f)\right]^{1/6}\right)$$
This estimate corresponds to double recurrence (as seen from the definition of $M^k_N(f)$), and its proof contains most of the relevant observations and ideas. In fact, we improve on the $k=1$ case here and show
$$W_N^1(f) \leq C'_1\left( \frac{1}{N^{1/6}} + \left[ M^1_{\floor{\sqrt{N}}} (f)\right]^{1/6}\right)$$
holds for all $N$. 
A much more general case is worked out in the appendix \ref{A:general case}, from which this result and more throughout this section follow immediately. The $k>1$ case has many more technical details arising from the summation switch, which are handled in detail but geometrically force a slightly worse estimate than what is presented here.
\begin{proof}[Proof of Theorem \ref{T:reverse BB, simplest case}, with $k=1$]
Applying the Van der Corput lemma pointwise, we bound the first piece trivially to yield a remainder of the order $1/\floor{\sqrt{N}}$. Extending the $n$ sum from $N-h$ to $N$ picks up another remainder of the order $1/\floor{\sqrt{N}}$. Hence:
\begin{align*}
W^1_N(f)^3 &\leq \norm{\sup_t \left|\frac{1}{N} \sum_{n=1}^N e^{2\pi i n t} f\circ T^n \right|}^2_2 \\
&\leq \frac{6}{\floor{\sqrt{N}}} +  \frac{4}{\floor{\sqrt{N}}}\sum_{h=1}^{\floor{\sqrt{N}}} \int \left| \frac{1}{N} \sum_{n=1}^N f \circ T^n \cdot \overline{f \circ T^{n + h}} \right| \, d\mu \\
&\leq \frac{6}{\floor{\sqrt{N}}} +  4\left( \frac{1}{\floor{\sqrt{N}}}\sum_{h=1}^{\floor{\sqrt{N}}} \int \left| \frac{1}{N} \sum_{n=1}^N f \circ T^n \cdot \overline{f \circ T^{n + h}} \right|^2\, d\mu \right)^{1/2}
\end{align*}
On the integral, we apply the lemma \ref{L:vdc for systems}:
\begin{align*}
&\leq \frac{6}{\floor{\sqrt{N}}} +  4\left( \frac{1}{\floor{\sqrt{N}}}\sum_{h=1}^{\floor{\sqrt{N}}} \frac{2}{N} \sum_{n = 0}^{N-1} \left(\frac{N-n}{N}\right)\Re \left[\int f \cdot \overline{f\circ T^{h}} \cdot \overline{f \circ T^n} \cdot f\circ T^{n+h} \, d\mu \right] \right)^{1/2} \, .
\end{align*}
Notice that for any sequence $a_n$, by interchanging sums we have
$$\frac{1}{N}\sum_{n=0}^{N-1} \left( \frac{N-n}{N} \right) a_n = \frac{1}{N}\sum_{q=0}^{N-1} \frac{1}{N}\sum_{n=0}^q a_n\, .$$
Applying this, we continue to see
\begin{align*}
&= \frac{6}{\floor{\sqrt{N}}} +  4\sqrt{2}\left( \frac{1}{\floor{\sqrt{N}}}\sum_{h=1}^{\floor{\sqrt{N}}} \frac{1}{N} \sum_{q=0}^{N-1} \frac{1}{N} \sum_{n = 0}^{q} \Re \left[\int f \cdot \overline{f\circ T^{h}} \cdot \overline{f \circ T^n} \cdot f\circ T^{n+h} \, d\mu \right] \right)^{1/2} \\
&\leq \frac{6}{\floor{\sqrt{N}}} + 4\sqrt{2} \left(  \frac{1}{N} \sum_{q=0}^{N-1} \left|\frac{1}{N \cdot \floor{\sqrt{N}}}  \sum_{n = 0}^{q} \sum_{h=1}^{\floor{\sqrt{N}}} \int f \cdot \overline{f\circ T^{h}} \cdot \overline{f \circ T^n} \cdot f\circ T^{n+h} \, d\mu \right| \right)^{1/2}
\end{align*}
after pulling out $\Re$ and bounding it by the absolute value. Consider excising the values of $q$ between $0$ and $\floor{\sqrt{N}}-1$ from the sum, which we can bound trivially to create a remainder term of the order $1/N$. Pulling this out by subadditivity, it may join the other remainder term:
\begin{align*}
&\leq \frac{6 +4\sqrt{2}}{\floor{\sqrt{N}}} + 4\sqrt{2} \left(  \frac{1}{N} \sum_{q=\floor{\sqrt{N}}}^{N-1} \left|\frac{1}{N \cdot \floor{\sqrt{N}}}  \sum_{n = 0}^{q} \sum_{h=1}^{\floor{\sqrt{N}}} \int f \cdot \overline{f\circ T^{h}} \cdot \overline{f \circ T^n} \cdot f\circ T^{n+h} \, d\mu \right| \right)^{1/2}\, .
\end{align*}
Now, consider shifting the indices of the $h$ sum down by $n$, which causes the indices of the summands to increase by $n$:
\begin{align*}
&= \frac{6 +4\sqrt{2}}{\floor{\sqrt{N}}} + 4\sqrt{2} \left(  \frac{1}{N} \sum_{q=\floor{\sqrt{N}}}^{N-1} \left|\frac{1}{N \cdot \floor{\sqrt{N}}}  \sum_{n = 0}^{q} \sum_{h=1-n}^{\floor{\sqrt{N}}-n} \int f \cdot \overline{f\circ T^{h+n}} \cdot \overline{f \circ T^n} \cdot f\circ T^{h+2n} \, d\mu \right| \right)^{1/2} \, .
\end{align*}

We wish to exchange the $n$ and $h$ sums. Note that the bounds on the $n$ and $h$ sum form a parallelogram in the $n,h$ plane. Since $q \geq \floor{\sqrt{N}}$, we may split this parallelogram into three regions: two triangles for $h$ between $1-q$ and $\floor{\sqrt{N}}-q$ and for $h$ between $1$ and $\floor{\sqrt{N}}$, and another parallelogram for $h$ between $\sqrt{N}-q$ and $0$. For each such region, we can interchange the sums to get
$$\sum_{n=0}^q \sum_{h=1-n}^{\floor{\sqrt{N}}-n} = \left[\sum_{h = 1-q}^{\floor{\sqrt{N}}-q} \sum_{n=1-h}^q \right] + \left[\sum_{h = \floor{\sqrt{N}}-q+1 }^{0} \sum_{n=1-h}^{\floor{\sqrt{N}}-h}\right] + \left[\sum_{h = 1}^{\floor{\sqrt{N}}} \sum_{n=0}^{\floor{\sqrt{N}}-h} \right]$$
written without the summands. We see that the third term has a total of $\leq \floor{\sqrt{N}}^2$ summands. If we bound those terms trivially, we get a remainder term of the size $1/\floor{\sqrt{N}}$. Likewise, the number of summands in the first term is
$$\sum_{h = 1-q}^{\floor{\sqrt{N}}-q} q-1+h +1 = \sum_{h = 1}^{\floor{\sqrt{N}}}h$$
which again is on the order of $\floor{\sqrt{N}}^2$. Hence, both of these pieces can be cut off and absorbed into the remainder term, and we continue to see
\begin{align*}
&\leq \frac{6 +12\sqrt{2}}{\floor{\sqrt{N}}} + 4\sqrt{2} \left(  \frac{1}{N} \sum_{q=\floor{\sqrt{N}}}^{N-1} \left|\frac{1}{N \cdot \floor{\sqrt{N}}}  \sum_{h = \floor{\sqrt{N}}-q+1}^{0} \sum_{n=1-h}^{\floor{\sqrt{N}}-h} \int f \cdot \overline{f\circ T^{h+n}} \cdot \overline{f \circ T^n} \cdot f\circ T^{h+2n} \, d\mu \right| \right)^{1/2} \\
&= \frac{6 +12\sqrt{2}}{\floor{\sqrt{N}}} + 4\sqrt{2} \left(  \frac{1}{N} \sum_{q=\floor{\sqrt{N}}}^{N-1} \left|\frac{1}{N \cdot \floor{\sqrt{N}}}  \sum_{h = 0}^{q-\floor{\sqrt{N}}-1}\sum_{n=1}^{\floor{\sqrt{N}}} \int f\circ T^{-h} \cdot \overline{f\circ T^{n-h}} \cdot \overline{f \circ T^{n}} \cdot f\circ T^{2n} \, d\mu \right| \right)^{1/2}
\end{align*}
after adjusting the indices, translating the measure by $T^{h}$, and replacing $h$ with $-h$.
Now, we finally pull the absolute value inside the integral to make the uniform double recurrence averages appear:

\begin{align*}
&\leq \frac{6 +12\sqrt{2}}{\floor{\sqrt{N}}} + 4\sqrt{2} \left(  \frac{1}{N} \sum_{q=\floor{\sqrt{N}}}^{N-1} \frac{1}{N}\sum_{h = 0}^{q-\floor{\sqrt{N}}-1} \int |f\circ T^{-h}| \cdot \left|\frac{1}{\floor{\sqrt{N}}}  \sum_{n=1}^{\floor{\sqrt{N}}} (\overline{f\circ T^{-h} \cdot f}) \circ T^n \cdot f\circ T^{2n} \right| \, d\mu \right)^{1/2} \\
&\leq \frac{6 +12\sqrt{2}}{\floor{\sqrt{N}}} + 4\sqrt{2} \left(  \frac{1}{N} \sum_{q=\floor{\sqrt{N}}}^{N-1} \frac{1}{N}\sum_{h = 0}^{q-\floor{\sqrt{N}}-1} \sup_{\Vert g \Vert_\infty \leq 1}\int\left|\frac{1}{\floor{\sqrt{N}}}  \sum_{n=1}^{\floor{\sqrt{N}}} g\circ T^n \cdot f\circ T^{2n} \right| \, d\mu \right)^{1/2} \\
&\leq \frac{6 +12\sqrt{2}}{\floor{\sqrt{N}}} + 4\sqrt{2} \left(  \frac{1}{N} \sum_{q=\floor{\sqrt{N}}}^{N-1} \frac{1}{N}\sum_{h = 0}^{q-\floor{\sqrt{N}}-1} M^1_{\floor{\sqrt{N}}}(f) \right)^{1/2}\, .
\end{align*}
With the $h$ dependence lost, both the averages over $h$ and $q$ vanish, and we may take a cube root of both sides to achieve the desired result.
\end{proof}

\subsection{Sublinearity of WW averages}\label{S:sublinear}
As the first order WW averages are sublinear, it follows that they form a subspace of $L^2(\mathcal{K})^\perp$, as noted by Assani \cite{assani}. However, the higher-order WW averages introduced in \cite{assani2024higherorderwienerwintnersystems} are submultilinear. Hence, controlling the average $W_N^k(f+g)$ by factoring it out requires control also of the ``off-diagonal averages". In \cite{assani2024higherorderwienerwintnersystems}, this approach is referred to as ``multilinearity concerns", and such control is possible in the examples given.

However, we note that the uniform multiple recurrence averages $M_N^k$ as defined previously are sublinear. Since we have established between the forwards and reverse Bourgain bounds that polynomial decay rates can pass between $M_N^k$ and $W_N^k$, it follows that we can also transfer sublinearity, albeit with a worse powers and remainders:

\begin{theorem}\label{T:WW linearity- basic estimate}
Let $k\in \N$. Then there exists $N_k$ and $C''_k$ such that for all invertible dynamical systems $(X, \mathcal{F}, \mu, T)$ and for all $f_1, f_2\in L^\infty(\mu)$ both bounded by 1, we have
\begin{align}
W_N^k(f_1 + f_2) &\leq C_k'' \left( \frac{1}{N^{1/(3\cdot 2^{k+3})}} + \left[W_{\floor{N^{1/4}}}^k (f_1)\right]^{1/(3 \cdot 2^k)} + \left[W_{\floor{N^{1/4}}}^k (f_2)\right]^{1/(3 \cdot 2^k)}\right)
\end{align}
for all $N \geq N_k$. Hence, we have
$$W_N^k(f_1 + f_2) \precsim_k W_N^k(f_1) + W_N^k(f_2) \, .$$
\end{theorem}

For example, the $k=2$ case yields
$$W_N^2(f_1 + f_2) \leq C''_2\left( \frac{1}{N^{1/96}} + W_{\floor{N^{1/4}}} (f_1)^{1/12} + W_{\floor{N^{1/4}}} (f_2)^{1/12} \right)$$
for an absolute constant $C_2''$ and sufficiently large $N$, with no dependence on the functions or system.

\begin{proof}
We apply the reverse Bourgain bound, use sublinearity, and apply the Bourgain bound:
\begin{align*}
&W_N^k(f_1 + f_2) \\ &\leq C'_k\left(\frac{1}{N^{1/24}} + \left[M_{\floor{N^{1/4}}}^k(f_1 + f_2)\right]^{1/6} \right) \\
&\leq C'_k\left(\frac{1}{N^{1/24}} + \left[M_{\floor{N^{1/4}}}^k(f_1) + M_{\floor{N^{1/4}}}^k(f_1)\right]^{1/6} \right) \\
&\leq C'_k\left(\frac{1}{N^{1/24}} + \left[ C_{k} \left( \frac{1}{\floor{N^{1/4}}^{1/2^k}} + \left[W_{\floor{N^{1/4}}}^k (f_1)\right]^{1/2^{k-1}} \right) + C_{k} \left( \frac{1}{\floor{N^{1/4}}^{1/2^k}} + \left[W_{\floor{N^{1/4}}}^k (f_2)\right]^{1/2^{k-1}} \right) \right]^{1/6} \right) \, .
\end{align*}
Consolidating remainder terms gives the desired bound for a larger constant.
\end{proof}

As previously remarked, the relation $\precsim$ is not generally additive. But in the case of polynomial decay, it does follow that if $f, g \precsim 0$, then $f + g \precsim 0$. Hence, in the context of WW functions we immediately get the following:

\begin{corollary}
Let $(X, \mathcal{F}, \mu, T)$ be an invertible dynamical system, and $f_1, f_2 \in L^\infty(\mu)$. If $f_1$ and $f_2$ are both $k$-th order WW functions of power type, then $f_1 + f_2$ is a $k$-th order WW function of power type; that is, if $W_N^k (f_1)\in \text{Poly}(\N)$ and $W_N^k (f_2)\in \text{Poly}(\N)$, then $W_N^k (f_1+f_2)\in \text{Poly}(\N)$ and the collection of WW functions of power type
$$\left\{f\in L^\infty(\mu) : W_N^k(f) \in \text{Poly}(\N)\right\} = \left\{f\in L^\infty(\mu) : W_N^k(f) \approx 0\right\}$$
forms a vector space.

Specifically, if  $W_N^k (f_1)\in \text{Poly}(\N, \alpha)$ and $W_N^k (f_2)\in \text{Poly}(\N, \beta)$, then $W_N^k (f_1+f_2)\in \text{Poly}\left(\N, \frac{\min\{\alpha, \beta\}}{3\cdot 2^{k+2}}\right)$.
\end{corollary}

From the application of sublinearity in the proof of Theorem \ref{T:WW linearity- basic estimate}, we could also get the above results for the sum of any amount of functions $f_1, f_2, \dots, f_m$ rather than just two. Hence, if $E$ is a collection of $k$-th order WW functions of power type $\alpha > 0$, then any $f\in \text{span}(E)$ is a $k$-th order WW function of power type $\alpha / 3 \cdot 2^{k+2}$. In order to show that a given $(X, \mathcal{F}, \mu, T)$ is a $k$-th order WW system of some power type $\alpha >0$, it suffices to find a set of $k$-th order WW functions of power type $\beta$ whose span is dense in $L^2(\mathcal{Z}_k)^\perp$.

\subsection{Off-diagonal averages}\label{S:Off-diag}

Linearity results for higher order WW averages were approached in \cite{assani2024higherorderwienerwintnersystems} through ``multilinearity concerns"; for a collection $\mathcal{E}\subset L^\infty(\mu)$ of $k$-th order WW functions of power type $\alpha$, we can show that any $f\in \text{span}(\mathcal{E})$ is a $k$-th order WW functions of power type $\alpha$ if we could establish polynomial decay on the \textit{off-diagonal terms} of the form
\begin{align}\label{F:off diag av}
\frac{1}{\floor{\sqrt{N}}^{k-1}} \sum_{h \in [\floor{\sqrt{N}}]^{k-1}} \left\Vert\sup_t \left| \frac{1}{N} \sum_{n=1}^N e^{2\pi i n t} \left[ \prod_{\eta \in V_{k-1}} c^{|\eta|} e_\eta \circ T^{h \cdot \eta}\right] \circ T^n\right| \right\Vert_2^{2/3}
\end{align}
for every possible collection $\{e_{\eta} : \eta\in V_{k-1}\}\subset \mathcal{E}$. This comes from simply factoring everything out using multi-sublinearity, in which $W_N^k(f)$ is bounded by a finite collection of terms of the form (\ref{F:off diag av}). 

While theoretically straightforward and easy to check in the specific examples, like K automorphisms or skew products \cite{assani2024higherorderwienerwintnersystems}, this approach is less helpful for more generic cases, such as stability results. For the stability under a product of a $k$-th order WW system $(Y, \mathcal{G}, \nu, S)$ with a K automorphism $(X, \mathcal{F}, \mu, T)$, extra steps must be taken to ensure a dense set of $k$-th order WW functions of power type; while it is easy to show that there is a spanning set, we do not immediately know that the WW property is transferred because we do not know anything about the system $(Y, \mathcal{G}, \nu, S)$.

However, we observe in the proof of the reverse Bourgain bound on double recurrence (Theorem \ref{T:reverse BB, simplest case}), most terms are bounded away to make the uniform double recurrence average $M_{\floor{\sqrt{N}}}^1(f)$ appear. In higher order cases, we can let these functions be arbitrary, and we see that the off-diagonal terms can be controlled by uniform multiple recurrence averages. Since these can then be controlled by diagonal average $W_N^k$, we can string these estimates together to control all off-diagonal WW averages by $W_N^k$:
\begin{theorem}\label{T:WW control of off-diag}
Let $k\in \N$. Then there exists a constant $C_k''$ and $N_k$ such that for all invertible dynamical systems $(X, \mathcal{F}, \mu, T)$ and  for all collections $g_\eta \in L^\infty(\mu)$ indexed by $\eta \in V_{k-1}$ and satisfying $\max_{\eta\in V_{k-1}}\Vert g_\eta \Vert_\infty \leq 1$, and for any $\zeta\in V_{k-1}$ we have
\begin{align}
\begin{split}
&\frac{1}{\floor{\sqrt{N}}^{k-1}} \sum_{h \in [\floor{\sqrt{N}}]^{k-1}} \left\Vert\sup_t \left| \frac{1}{N} \sum_{n=1}^N e^{2\pi i n t} \left[ \prod_{\eta \in V_{k-1}} c^{|\eta|} g_\eta \circ T^{h \cdot \eta}\right] \circ T^n\right| \right\Vert_2^{2/3} \\ &\leq C''_k \left(\frac{1}{N^{1/(3\cdot 2^{k+3})}} + \left[W_{\floor{N^{1/4}}}^k(g_\zeta)\right]^{1/(3\cdot 2^{k})} \right)
\end{split}
\end{align}
for $N \geq N_k$.
\end{theorem}
\begin{remark}
Based on dependence, note that we can take a min over $\zeta\in V_{k-1}$ on the right-hand side and pass it through to the WW average. Hence, the above can be written as
\begin{align*}
\frac{1}{\floor{\sqrt{N}}^{k-1}} \sum_{h \in [\floor{\sqrt{N}}]^{k-1}} \left\Vert\sup_t \left| \frac{1}{N} \sum_{n=1}^N e^{2\pi i n t} \left[ \prod_{\eta \in V_{k-1}} c^{|\eta|} g_\eta \circ T^{h \cdot \eta}\right] \circ T^n\right| \right\Vert_2^{2/3} \precsim_k \min_{\eta \in V_{k-1}} W_N^k(g_\eta) \, .
\end{align*}
Since the minimum may also be bounded by the product over all $W_N^k(g_\eta)$, each raised to the appropriate power, this may also be used to bound the above terms.
\end{remark}

In \cite[Theorem 6.3]{assani2024higherorderwienerwintnersystems}, it is shown that all $k$-th order off-diagonal averages for $f$ (above averages where some $g_\eta = f$) converge to zero if $f\in L^2(\mathcal{Z}_{k}^\perp)$. This theorem shows that they do so at a uniform rate, which is controlled by the WW average itself.

To build up to this theorem, we first apply the exact same argument of the reverse Bourgain bound (Theorem \ref{T:reverse BB, simplest case}) on a collection of functions $g_\eta$ for $\eta\in V_{k-1}$, rather than only $f$. This is done in detail in the appendix:
\begin{lemma}\label{L:RBB for off-diag}
Let $k\in N$. Then there exists a constant $C_k'$ such that for all invertible dynamical systems $(X, \mathcal{F}, \mu, T)$ and for all collections $g_\eta \in L^\infty(\mu)$ indexed by $\eta \in V_{k-1}$ and satisfying $\max_{\eta\in V_{k-1}}\Vert g_\eta \Vert_\infty \leq 1$, we have
\begin{align}
\begin{split}
&\frac{1}{\floor{\sqrt{N}}^{k-1}} \sum_{h \in [\floor{\sqrt{N}}]^{k-1}} \left\Vert\sup_t \left| \frac{1}{N} \sum_{n=1}^N e^{2\pi i n t} \left[ \prod_{\eta \in V_{k-1}} c^{|\eta|} g_\eta \circ T^{h \cdot \eta}\right] \circ T^n\right| \right\Vert_2^{2/3} \\ &\leq C'_k \left(\frac{1}{N^{1/24}} + \left[M_{\floor{N^{1/4}}}^k(g_\textbf{1})\right]^{1/6} \right)
\end{split}
\end{align}
where $\textbf{1}\in V_{k-1}$ denotes the element with $1$'s in every component.
\end{lemma}

Since this holds for any collection $\{g_\eta : \eta\in V_{k-1}\}$, we may bound the off-diagonal term by any $g_\zeta$ for $\zeta\in V_{k-1}$ if we can permute the product over $\eta\in V_{k-1}$ to put $\zeta$ in the $\textbf{1}$ slot. We cannot reorder the terms $g_\eta$ arbitrarily, but we may do so enough that the changes created are lost in the application of the previous lemma. 

\begin{lemma}\label{L:permutation for off-diag}
Let $k\in N$. Then there exists a constant $C_k'$ such that for all invertible dynamical systems $(X, \mathcal{F}, \mu, T)$ and for all collections $g_\eta \in L^\infty(\mu)$ indexed by $\eta \in V_{k-1}$ and satisfying $\max_{\eta\in V_{k-1}}\Vert g_\eta \Vert_\infty \leq 1$, and for any $\zeta\in V_{k-1}$, we have
\begin{align}
\begin{split}
&\frac{1}{\floor{\sqrt{N}}^{k-1}} \sum_{h \in [\floor{\sqrt{N}}]^{k-1}} \left\Vert\sup_t \left| \frac{1}{N} \sum_{n=1}^N e^{2\pi i n t} \left[ \prod_{\eta \in V_{k-1}} c^{|\eta|} g_\eta \circ T^{h \cdot \eta}\right] \circ T^n\right| \right\Vert_2^{2/3} \\ &\leq C'_k \left(\frac{1}{N^{1/24}} + \left[M_{\floor{N^{1/4}}}^k(g_\zeta)\right]^{1/6} \right)
\end{split}
\end{align}
\end{lemma}

\begin{proof} Consider the off-diagonal average:
\begin{align*}
\frac{1}{\floor{\sqrt{N}}^{k-1}} \sum_{h \in [\floor{\sqrt{N}}]^{k-1}} \left\Vert\sup_t \left| \frac{1}{N} \sum_{n=1}^N e^{2\pi i n t} \left[ \prod_{\eta \in V_{k-1}} c^{|\eta|} g_\eta \circ T^{h \cdot \eta}\right] \circ T^n\right| \right\Vert_2^{2/3} \, .
\end{align*}
We wish to re-index in the following way
$$\sum_{h_i = 1}^{\floor{\sqrt{N}}} a_{h_i} = \sum_{h_i = 1}^{\floor{\sqrt{N}}} a_{\floor{\sqrt{N}} - h_i + 1}$$
for each variable $h_i$ with $\zeta_i = 0$.
If we define for any two $\alpha, \beta\in V_{k-1}$, the intersection $\alpha \cap \beta \in V_{k-1}$ as
$$(\alpha \cap \beta)_i = \begin{cases}
1 & \alpha_i = \beta_i = 1 \\ 0 &\text{otherwise}
\end{cases}$$
then the reindexing described above takes the expression $\eta\cdot h$ to
$$[\eta \cap \zeta - \eta \cap (\mathbf{1} - \zeta)] \cdot h + |\eta \cap (\mathbf{1} - \zeta)|(\floor{\sqrt{N}} + 1)\, .$$
Shifting the measure by $T^{(\mathbf{1} - \zeta) \cdot h}$ puts the product in the above term over
$$\prod_{\eta\in V_{k-1}} c^{|\eta|} \left(g_\eta \circ T^{|\eta \cap (\mathbf{1} - \zeta)|(\floor{\sqrt{N}} + 1)}\right) \circ T^{[\eta \cap \zeta - \eta \cap (\mathbf{1} - \zeta) + \mathbf{1} - \zeta] \cdot h }\, .$$
The expression $\eta \cap \zeta - \eta \cap (\mathbf{1} - \zeta) + \mathbf{1} - \zeta$ is exactly the permutation function on $V_{k-1}$ which flips the contents of each index $i$ with $\zeta_i = 0$. We observe that the $\eta = \zeta$ term above exactly corresponds to the term
$$c^{|\zeta|} g_\zeta \circ T^{\mathbf{1} \cdot h}\, . $$
Hence, we have placed $g_\zeta$ into the $\mathbf{1}$ slot. After potentially correcting for the complex conjugates, we can apply the reverse Bourgain bound from the previous lemma \ref{L:RBB for off-diag} to get the desired estimate.
\end{proof}
Chaining this together with the Bourgain bound (just as done in the proof of \ref{T:WW linearity- basic estimate}), gives us Theorem \ref{T:WW control of off-diag}:

\begin{proof}[Proof of Theorem \ref{T:WW control of off-diag}]
Beginning with the previous lemma, we apply the Bourgain bound:
\begin{align*}
&\frac{1}{\floor{\sqrt{N}}^{k-1}} \sum_{h \in [\floor{\sqrt{N}}]^{k-1}} \left\Vert\sup_t \left| \frac{1}{N} \sum_{n=1}^N e^{2\pi i n t} \left[ \prod_{\eta \in V_{k-1}} c^{|\eta|} g_\eta \circ T^{h \cdot \eta}\right] \circ T^n\right| \right\Vert_2^{2/3} \\ &\leq C'_k \left(\frac{1}{N^{1/24}} + \left[M_{\floor{N^{1/4}}}^k(g_\zeta)\right]^{1/6} \right) \\
&\leq C'_k \left(\frac{1}{N^{1/24}} + \left[C_k\left( \frac{1}{\floor{N^{1/4}}^{1/2^k}} + \left[W_{\floor{N^{1/4}}}^k (g_\zeta)\right]^{1/2^{k-1}} \right)\right]^{1/6} \right) \, .
\end{align*}
Consolidating remainder terms gives the desired result.
\end{proof}

\subsection{Conditional expectation and factors}\label{S:conditioning}

Due to the multilinear nature of the higher order WW averages, their behavior compared to conditional expectation is less clear to analyze. With the previous results in tow, we obtain the following:

\begin{theorem}\label{T:conditional exp}
Let $k\in \N$. Then there exists constants $C''_k$ and $N_k$ such that for all invertible dynamical systems $(X, \mathcal{F}, \mu, T)$, for all functions $f\in L^\infty(\mu)$ bounded by 1, and for all $T$-invariant sub $\sigma$-algebras $\mathcal{B} \subset \mathcal{F}$, we have
\begin{align*}
W_N^k\left(\E(f |\mathcal{B})\right) \leq C''_k \left(\frac{1}{N^{1/(3\cdot 2^{k+1})}} + \left[W_{\floor{N^{1/4}}}^k(f)\right]^{1/(3\cdot 2^{k})} \right)
\end{align*}
for all $N\geq N_k$. Hence, we have
$$W_N^k\left(\E(f |\mathcal{B})\right) \precsim_k W_N^k(f)\, .$$
\end{theorem}

\begin{proof}
Recall that conditional expectation $\mathbb{E}(\cdot|\mathcal{B}):L^2(\mu) \to L^2(\mu)$ is a bounded and positive linear operator which satisfies $|\mathbb{E}(f|\mathcal{B})| \leq \mathbb{E}(|f| \, |\mathcal{B})$. By positivity, it follows that if $f_t(x)$ are parameterized by $t$, we have $\sup_t |\mathbb{E}(f_t|\mathcal{B})| \leq \mathbb{E}(\sup_t|f_t| \, |\mathcal{B})$. Since $\mathcal{B}$ is $T$-invariant, it also follows that $\mathbb{E}(\cdot|\mathcal{B})$ commutes with the action of $T$ on $L^2$. 

Using this and other properties of conditional expectation, note for any $f\in L^2(\mu)$ and $g\in L^2(\mathcal{B})$ we have
\begin{align}\label{E:conditional strip}
\begin{split}
\left\Vert \sup_t \left| \frac{1}{N} \sum_{n=1}^N e^{2\pi i n t} \mathbb{E}(f|\mathcal{B}) \circ T^n \cdot g \circ T^n\right|\right\Vert_2 &=\left\Vert \sup_t \left| \frac{1}{N} \sum_{n=1}^N e^{2\pi i n t}\mathbb{E}(f \circ T^n|\mathcal{B}) \cdot g \circ T^n\right|\right\Vert_2 \\
&=\left\Vert \sup_t \left| \frac{1}{N} \sum_{n=1}^N e^{2\pi i n t}\mathbb{E}(f \circ T^n \cdot g \circ T^n|\mathcal{B})\right|\right\Vert_2 \\
&=\left\Vert \sup_t \left|\mathbb{E}\left( \frac{1}{N} \sum_{n=1}^N e^{2\pi i n t}f \circ T^n \cdot g \circ T^n\Big|\mathcal{B}\right)\right|\right\Vert_2 \\
&\leq \left\Vert \mathbb{E}\left( \sup_t \left| \frac{1}{N} \sum_{n=1}^N e^{2\pi i n t}f \circ T^n \cdot g \circ T^n\right|\Big|\mathcal{B}\right)\right\Vert_2 \\
&\leq \left\Vert \sup_t \left| \frac{1}{N}\sum_{n=1}^N  e^{2\pi i n t}f \circ T^n \cdot g \circ T^n\right|\right\Vert_2 \, .
\end{split}
\end{align}
In our $k$-th order WW average of $\E(f|\mathcal{B})$, we split up the product as follows:
$$\left[\prod_{\eta\in V_{k-1}} c^{|\eta|} \mathbb{E}(f|\mathcal{B}) \circ T^{\eta\cdot h}\right]\circ T^n = \E(f|\mathcal{B})\cdot T^n \cdot \left[\prod_{\eta\in V_{k-1} - \{\mathbf{0}\}} c^{|\eta|} \mathbb{E}(f|\mathcal{B}) \circ T^{\eta\cdot h}\right]\circ T^n \, .$$
On these functions, we may apply the estimate \ref{E:conditional strip} to remove a conditional expectation. This yields an off-diagonal average, which we can apply the theorem \ref{T:WW control of off-diag}:

\begin{align*}
&W_N^k\left( \E(f | \mathcal{B})\right) \\
&=\frac{1}{\lfloor\sqrt{N}\rfloor^{k-1}} \sum_{h \in [\lfloor\sqrt{N}\rfloor]^{k-1}} \left\Vert \sup_t \left| \frac{1}{N} \sum_{n=1}^N e^{2\pi i n t} \left[ \prod_{\eta \in V_{k-1}} c^{|\eta|} \mathbb{E}(f|\mathcal{B}) \circ T^{\eta\cdot h}\right]\circ T^n\right| \right\Vert_2^{2/3} \\
&\leq \frac{1}{\lfloor\sqrt{N}\rfloor^{k-1}} \sum_{h \in [\lfloor\sqrt{N}\rfloor]^{k-1}} \left\Vert \sup_t \left| \frac{1}{N} \sum_{n=1}^N e^{2\pi i n t} f\circ T^n \cdot \left[ \prod_{\eta \in V_{k-1} - \{\textbf{0}\}} c^{|\eta|} \mathbb{E}(f|\mathcal{B}) \circ T^{\eta\cdot h}\right]\circ T^n\right| \right\Vert_2^{2/3}\\
&\leq C''_k \left(\frac{1}{N^{1/(3\cdot 2^{k+1})}} + \left[W_{\floor{N^{1/4}}}^k(f)\right]^{1/(3\cdot 2^{k})} \right)
\end{align*}
for $N\geq N_{k}$.
\end{proof}

In the case of polynomial decay rates, we see that conditional expectations of WW functions of power type will be WW functions of power type. Moreover, we see that dense sets of WW function may project down, and this transfer will also apply to WW systems:
\begin{theorem}
Let $(X, \mathcal{F}, \mu, T)$ be a $k$-th order WW system of power type $\alpha$, and let $\mathcal{B}$ be a $T$-invariant $\sigma$-subalgebra of $\mathcal{F}$. Then $(X, \mathcal{B}, \mu, T)$ is a $k$-th order WW system of power type $\alpha / (3\cdot 2^{k+2})$.
\end{theorem}
\begin{proof}
By the bound from Theorem \ref{T:conditional exp}, we immediately see that
$$W_N^k(f) \in \text{Poly}(\N, \alpha) \quad \implies \quad W^k_N(\E(f|\mathcal{B}))\in \text{Poly}\left(\N, \frac{\alpha}{3 \cdot 2^{k+2}}\right)$$

In the case where $(X, \mathcal{F}, \mu, T)$ is a WW system, let $g \in L^2(\mathcal{Z}^\mathcal{B}_{k-1})^\perp$. As $g$ is $\mathcal{B}$-measurable, we note that its Host-Kra-Gowers seminorms are the same over $(X, \mathcal{B}, \mu, T)$ and $(X, \mathcal{F}, \mu, T)$, as the integrals will agree over any $\mathcal{B}$-measurable function. Hence, $g\in L^2(\mathcal{Z}^\mathcal{F}_{k-1})^\perp$. If $f_m\in L^2(\mathcal{Z}^\mathcal{F}_{k-1})^\perp$ is a sequence of $k$-th order WW functions of power type $\alpha$ converging to $g$, then $\mathbb{E}(f_m|\mathcal{B})$ is a sequence of $k$-th order WW functions of power type $\alpha / (3\cdot 2^{k+2})$ converging to $\mathbb{E}(g|\mathcal{B}) = g$. As before, since each $\mathbb{E}(f_m|\mathcal{B})$ is $\mathcal{B}$-measurable, it follows that these functions are also in $L^2(\mathcal{Z}^\mathcal{B}_{k-1})^\perp$. 
\end{proof}

\subsection{Weak WW averages and product stability}\label{S:WWW}

In \cite{assani04}, Assani defines \textit{weak WW averages}, which we denote by the following:
\begin{align*}
    w_N^1(f) &:= \sup_t \left\Vert \frac{1}{N} \sum_{n=1}^N e^{2\pi i n t} f \circ T^n \right\Vert_2^{2/3} \, .
\end{align*}
We denote these as ``first-order weak WW averages", and extend them to higher orders by obeying the same inductive formula \ref{F:inductive construction of WW averages} for the previous ``strong" WW averages:
\begin{align}\label{F:inductive construction of weak WW averages}
w^{k}_N(f) = \frac{1}{\floor{\sqrt{N}}}\sum_{h=1}^{\floor{\sqrt{N}}} w^{k-1}_N \left( f \cdot \overline{f \circ T^h} \right) \, .
\end{align}
By the same reasoning, we obtain the closed formula
\begin{align*}
w^{k}_N(f) &= \frac{1}{\floor{\sqrt{N}}^{k-1}}\sum_{h\in [\floor{\sqrt{N}}]^{k-1}} w^{1}_N \left( \prod_{\eta\in V_{k-1}} c^{|\eta|} f\circ T^{\eta \cdot h} \right) \\
&=\frac{1}{\floor{\sqrt{N}}^{k-1}}\sum_{h\in [\floor{\sqrt{N}}]^{k-1}} \sup_t \left\Vert \frac{1}{N} \sum_{n=1}^N e^{2\pi i n t}\left[\prod_{\eta\in V_{k-1}} c^{|\eta|} f\circ T^{\eta \cdot h} \right]\circ T^n \right\Vert_2^{2/3} \, .
\end{align*}

Likewise, we can define \textit{$k$-th order weak WW functions of power type} to be $f\in L^\infty(\mu)$ with $w_N^k \in \text{Poly}(\N)$, and \textit{$k$-th order weak WW functions of power type $\alpha$} to be $f\in L^\infty(\mu)$ with $w_N^k \in \text{Poly}(\N, \alpha)$.

As pulling the supremum out of the norm decreases the term, it follows that the weak WW averages $w_N^k$ are weaker than the ``strong" WW averages $W_N^k$ in the following sense:
$$w_N^k(f) \leq W_N^k(f)$$
for any $f\in L^\infty(\mu)$ and $N\in\N$. Hence, for any order, strong WW functions of power type will be weak WW function of the same power type. We observe that a partial converse is possible, at the cost of order:

\begin{theorem}\label{T:converse bound for weak WW averages}
There exists a constant $C$ such that for all dynamical systems $(X, \mathcal{F}, \mu, T)$ and functions $f\in L^\infty(\mu)$ bounded by 1, we have
$$W_N^k(f) \leq C\left(\frac{1}{N^{1/6}} + \left[w_N^{k+1}(f)\right]^{1/8} \right)$$
for all $N$ and $k$. Using our asymptotic notation, we have
$$w_N^k(f) \precsim W_N^{k+1}(f)\, .$$
\end{theorem}
Hence, the first order strong WW averages can be bounded by second order weak WW averages, and so on. 

\begin{proof} We use the following fact, shown by Assani in \cite{assani04}, that for any $f\in L^2(\mu)$ bounded by 1, we have 
\begin{align}\label{E:weak WW estimate from book}
\sup_{\Vert g\Vert_2 \leq 1} \frac{1}{N} \sum_{n=0}^{N-1} |\hat\sigma_{f, g} |^{2} \leq \sup_t \left \Vert \frac{1}{N} \sum_{n=1}^N e^{2\pi i n t} f \circ T^n \right\Vert_2 \, .
\end{align}
While Assani establishes the bound in which the $n$ on the left ranges from $1$ to $N$, the above holds by the same argument.

For any $f\in L^\infty(\mu)$, we apply the Van der Corput inequality pointwise with $H = \floor{\sqrt{N}}$ and apply Hölder's inequality:
\begin{align*}
W_N^1(f)^3 &\leq \frac{2}{\floor{\sqrt{N}}} + \frac{4}{\floor{\sqrt{N}}}\sum_{h=1}^{\floor{\sqrt{N}}} \left( \int \left| \frac{1}{N} \sum_{n=1}^N f\circ T^n \cdot \overline{f \circ T^{n+h}} \right|^2 \, d\mu \right)^{1/2} \, .
\end{align*}
Inside, we apply the lemma \ref{vdc-lem}:
\begin{align*}
&\leq \frac{2}{\floor{\sqrt{N}}} + \frac{4}{\floor{\sqrt{N}}}\sum_{h=1}^{\floor{\sqrt{N}}} \left( \frac{2}{N} \sum_{n=0}^{N-1} \left( \frac{N-n}{N}\right) \Re \left[ \int f \cdot \overline{f \circ T^{h}} \cdot \overline{f \circ T^n} \cdot f \circ T^{n+h} \, d\mu \right] \right)^{1/2} \, .
\end{align*}
Bounding $\Re$ by the absolute value and $(N-n)/N$ by 1, we obtain spectral measure coefficients on which we can apply (\ref{E:weak WW estimate from book}):
\begin{align*}
&\leq \frac{2}{\floor{\sqrt{N}}} + \frac{4\sqrt{2}}{\floor{\sqrt{N}}}\sum_{h=1}^{\floor{\sqrt{N}}} \left( \frac{1}{N} \sum_{n=0}^{N-1} |\hat \sigma_{f \cdot \overline{f\circ T^h}} (n)| \right)^{1/2} \\
&\leq \frac{2}{\floor{\sqrt{N}}} + \frac{4\sqrt{2}}{\floor{\sqrt{N}}}\sum_{h=1}^{\floor{\sqrt{N}}} \left( \frac{1}{N} \sum_{n=0}^{N-1} |\hat \sigma_{f \cdot \overline{f\circ T^h}} (n)|^2 \right)^{1/4} \\
&\leq \frac{2}{\floor{\sqrt{N}}} + \frac{4\sqrt{2}}{\floor{\sqrt{N}}}\sum_{h=1}^{\floor{\sqrt{N}}} \left( \sup_t \left\Vert \sum_{n=1}^N e^{2\pi i n t} f\circ T^n \cdot \overline{f\circ T^{n + h}}\right\Vert_2 \right)^{1/4} \\
&\leq \frac{2}{\floor{\sqrt{N}}} + 4\sqrt{2} \left( w_{N}^2(f) \right)^{3/8}
\end{align*}
after one more application of Hölder's inequality to reintroduce the 2/3 power. Taking the cube root of both sides of the inequality, we see that
\begin{align}\label{E:weak bound k=1 case}
W_N^1(f) \leq \frac{2^{1/3}}{N^{1/6}} + 2^{5/6} \left[w_N^2(f)\right]^{1/8} \, .
\end{align}

To lift this to higher orders, we write $W_N^k$ in terms of $W_N^1$ and apply (\ref{E:weak bound k=1 case}):
\begin{align*}
W_N^k(f) &=  \frac{1}{\floor{\sqrt{N}}^{k-1}}\sum_{h\in [\floor{\sqrt{N}}]^{k-1}} W^{1}_N \left( \prod_{\eta\in V_{k-1}} c^{|\eta|} f\circ T^{\eta \cdot h} \right) \\
&\leq \frac{1}{\floor{\sqrt{N}}^{k-1}}\sum_{h\in [\floor{\sqrt{N}}]^{k-1}} \left(\frac{2^{1/3}}{N^{1/6}} + 2^{5/6} \left[ w^{2}_N \left( \prod_{\eta\in V_{k-1}} c^{|\eta|} f\circ T^{\eta \cdot h} \right)\right]^{1/8} \right) \\
&\leq \frac{2^{1/3}}{N^{1/6}} + 2^{5/6} \left(\frac{1}{\floor{\sqrt{N}}^{k-1}}\sum_{h\in [\floor{\sqrt{N}}]^{k-1}} w^{2}_N \left( \prod_{\eta\in V_{k-1}} c^{|\eta|} f\circ T^{\eta \cdot h} \right) \right)^{1/8} \, .
\end{align*}
By the same inductive argument, adding $k-1$ layers of the cube increases the weak WW norm from $w^2_N(f)$ to $w^{k+1}_N(f)$, and after consolidating for constants, we get the desired bound.
\end{proof}

Between these two bounds, it follows that any results about weak WW averages can be transfer to strong WW averages, at the cost of order. For one such example, one results about products of weak WW averages, shown by Assani \cite{assani04}, immediately transfers to the higher order case:

\begin{theorem}\label{T:Product bound for weak WW averages}
Let $(X, \mathcal{F}, \mu, T)$ and $(Y, \mathcal{G}, \nu, S)$ be dynamical systems. For any $k\in \N$ and $f\in L^\infty(\mu)$ and $g\in L^\infty(\nu)$, both bounded by 1, we have
$$ww_N^k(f \otimes g) \leq \min \{ww_N^k (f), \, ww_N^{k} (g)\}$$
for all $N$, where the averages on the left are taken in the product system $X \times Y$, while the averages on the right are taken in $X$ and $Y$, as appropriate.
\end{theorem}
We note, in line with the remark \ref{R:is ergodicity needed}, that this product system $X \times Y$ may not be ergodic, in which it may not make sense to consider WW systems. However, this bound, along with the following corollary \ref{C:product bound for strong WW averages} and Theorem \ref{T:product of WW functions}, both hold regardless of ergodicity. Even if $X \times Y$ fails to be ergodic, it is still well-posed to consider WW averages and WW functions; we just may no longer be able to analyze their limiting behavior with the Gower-Host-Kra seminorms given by the construction (\ref{F:HKZ seminorm construction}).

\begin{proof}
In \cite{assani04}, the case $k=1$ is shown. For the higher order case, we expand $w_N^k$ in terms of $w_N^1$ and use the $k=1$ case:
\begin{align*}
w_N^k(f \otimes g) &= \frac{1}{\floor{\sqrt{N}}^{k-1}}\sum_{h\in [\floor{\sqrt{N}}]^{k-1}} w^{1}_N \left( \prod_{\eta\in V_{k-1}} c^{|\eta|} (f \otimes g)\circ (T\times S)^{\eta \cdot h}  \right) \\
&= \frac{1}{\floor{\sqrt{N}}^{k-1}}\sum_{h\in [\floor{\sqrt{N}}]^{k-1}} w^{1}_N \left( \left(\prod_{\eta\in V_{k-1}} c^{|\eta|} f \circ T^{\eta \cdot h} \right)\otimes \left(\prod_{\eta\in V_{k-1}} c^{|\eta|}g \circ S^{\eta \cdot h}\right) \right) \\
&\leq \frac{1}{\floor{\sqrt{N}}^{k-1}}\sum_{h\in [\floor{\sqrt{N}}]^{k-1}} w^{1}_N \left( \prod_{\eta\in V_{k-1}} c^{|\eta|} f \circ T^{\eta \cdot h}  \right) \\
&= w_N^k(f) \, .
\end{align*}
By the same reasoning, $w_N^k(f \otimes g)$ is less than $w_N^k(g)$, establishing the desired estimate.
\end{proof}

Using the bound from Theorem \ref{T:converse bound for weak WW averages}, we can transfer the previous theorem to strong WW averages at the cost of order, and bound back by the strong WW averages:

\begin{corollary}\label{C:product bound for strong WW averages}
Let $k\in \N$. There exists a constant $C$ such that for $k\in \N$ and all dynamical systems $(X, \mathcal{F}, \mu, T)$ and $(Y, \mathcal{G}, \nu, S)$ and functions $f\in L^\infty(\mu)$ and $g\in L^\infty(\nu)$, both bounded by 1, we have
\begin{align*}
W_N^k(f \otimes g) \leq C\left(\frac{1}{N^{1/6}} + \left[\min \left\{W_N^{k+1}(f), W_N^{k+1}(g)\right\}\right]^{1/8} \right)
\end{align*}
or
$$W_N^k(f \otimes g) \precsim \min\{W_N^{k+1}(f), W_N^{k+1}(g)\} \, .$$
\end{corollary}

Hence, the decay rate of the $k+1$-th order WW averages for $f$ transfers to $f\otimes g$ for any $g\in L^\infty(\nu)$. In the case of WW functions, we get the following:
\begin{theorem}\label{T:product of WW functions}
Let $(X, \mathcal{F},  \mu, T)$ be a dynamical system. If $f$ is a $k+1$-th order WW function of power type $\alpha$, then for any dynamical system $(Y,\mathcal{G}, \nu, S)$ and $g\in L^\infty(\nu)$ the function $f\otimes g$ is a $k$-th order WW function of power type $\alpha/8$ in the product system $X\times Y$.
\end{theorem}

With the sublinearity-like property of the WW averages established in Theorem \ref{T:WW linearity- basic estimate}, we have what we need to conclude the following:

\begin{theorem}
Let $(X, \mathcal{F}, \mu, T)$ be an invertible, weakly mixing $k+1$-th order WW system of power type $\alpha$, and let $(Y, \nu, \mathcal{G}, S)$ be an invertible $k$-th order WW system of power type $\beta$. Then the product $(X\times Y, \mathcal{F} \otimes \mathcal{G}, \mu \times \nu, T \times S)$ is a $k$-th order WW system of power type $\frac{\min\{\alpha/8, \beta\}}{3\cdot 2^{k+5}}$.
\end{theorem}

\begin{proof}
As $X$ is weakly mixing, we note that its Host-Kra-Ziegler factor $\mathcal{Z}_{k-1}^X$ is trivial. Knowing this, we observe the following decomposition for any $f\in L^\infty(\mu)$ and $g\in L^\infty(\mu)$:
\begin{align*}
f \otimes g &= \int f \, d\mu \otimes g + \left(f - \int f\, d\mu\right) \otimes g \\
&= \int f \, d\mu \otimes \mathbb{E}(g | \mathcal{Z}_{k-1}^Y) + \int f \, d\mu \otimes \left(g - \mathbb{E}(g | \mathcal{Z}_{k-1}^Y) \right)+ \left(f - \int f\, d\mu\right) \otimes g\, .
\end{align*}
The first piece corresponds to the Host-Kra-Ziegler factor of $X \times Y$. Since $\int f \, d\mu$ is constant, there is a dense set of $k$-th order WW functions of power type $\beta$ in the second piece, as $Y$ is a system of power type $\beta$. By the previous theorem, there is a dense set of $k$-th order WW functions of power type $\alpha/8$ in the last piece, as $X$ is a system of power type $\alpha$.

Hence, we have found a set of $k$-th order WW functions of power type $\min\{\alpha/8, \beta\}$ whose span is $L^2(\mathcal{Z}_{k-1}^{X \times Y})^\perp$. By the sublinearity from Theorem \ref{T:WW linearity- basic estimate}, it follows that the span of these functions is a dense subset of $k$-th order WW functions of power type $\frac{\min\{\alpha/8, \beta\}}{3\cdot 2^{k+5}}$, establishing the claim.
\end{proof}

\begin{remark}
As seen in \cite{assani2024higherorderwienerwintnersystems}, this stability also preserves the weaker condition that $Y$ satisfies pointwise convergence of $k$-th order multiple recurrence averages. Specifically, if $X$ is a weakly mixing $k+1$-th order WW system of power type, and $Y$ satisfies pointwise convergence of $k$-th order multiple recurrence averages, then $X\times Y$ satisfies pointwise convergence of $k$-th order multiple recurrence averages.
\end{remark}

\subsection{Alternative constructions of WW averages}\label{S:Alt const}

Recall that the family of higher-order WW averages $W_N^k$ was constructed in \cite{assani2024higherorderwienerwintnersystems} to satisfy
\begin{align*}
W^1_N(f) &= \left\Vert \sup_{t}\left| \frac{1}{N}\sum_{n=1}^N e^{2\pi i n t} f\circ T^n \right| \right\Vert_2^{2/3}
\end{align*}
and 
\begin{align*}
W^{k}_N(f) = \frac{1}{\floor{\sqrt{N}}}\sum_{h=1}^{\floor{\sqrt{N}}} W^{k-1}_N \left( f \cdot \overline{f \circ T^{h}} \right)
\end{align*}
for $k\geq 2$. Here, the choice of $h\in [\floor{\sqrt{N}}]$ was done to balance remainder terms of the form $1/H$ and $H/N$ in establishing the Bourgain bound (\ref{L:BB}). In the context of WW functions, we needed this remainder term to decay polynomially, so that the Bourgain bound can transfer polynomial decay. Hence, any other choice of $H = \floor{N^\delta}$ for $0 < \delta < 1$ would achieve this end, and could be used to construct WW averages that satisfy the theory of \cite{assani2024higherorderwienerwintnersystems}.

In this section, we wish to analyze the WW averages formed by different bounds on the indexing variable $h$. In more generality, we let $r_k: \N \to \N$ play the role of $\floor{\sqrt{N}}$, and we can define \textit{$k$-th order WW averages for $r_1, \dots, r_{k-1}$} as $\widetilde{W}_N^k$ which satisfy
\begin{align*}
\widetilde{W}^1_N(f) &= W_N^1(f) = \left\Vert \sup_{t}\left| \frac{1}{N}\sum_{n=1}^N e^{2\pi i n t} f\circ T^n \right| \right\Vert_2^{2/3}
\end{align*}
and 
\begin{align*}
\widetilde{W}^{k}_N(f) = \frac{1}{r_{k-1}(N)}\sum_{h_{k-1}=1}^{r_{k-1}(N)} \widetilde W^{k-1}_N \left( f \cdot \overline{f \circ T^{h_{k-1}}} \right)
\end{align*}
for all $k\geq 2$. 

From the Bourgain bound (\ref{E:BB for M, W}) and reverse Bourgain bound (\ref{E:RBB for W, M}), we have already established that $M_N^k(f) \precsim_k W_N^k(f)$ and $W_N^k(f) \precsim_k M_N^k(f)$, respectively, in which $W_N^k(f) \approx_k M_N^k(f)$. Here, we establish corresponding Bourgain bounds and Reverse Bourgain bounds for $\widetilde{W}_N^k$. In the case that each $r_1, \dots, r_{k-1}$ grow like polynomials, it will follow that $M_N^k(f) \precsim_{r, k} \widetilde W_N^k(f)$ and  $\widetilde W_N^k(f) \precsim_{r, k} M_N^k(f)$ hold, and we have $\widetilde W_N^k(f) \approx_{r, k} M_N^k(f)$. By transitivity, we will be able to conclude that any such $\widetilde{W}_N^k(f)$ will be equivalent to the classical WW average $W_N^k(f)$, and we have the following:

\begin{theorem}\label{T:Equivalence of WW averages}
Let $k\in N$, and let $\widetilde W_{N}^k$ be the WW averages constructed from the functions $r_1, \dots, r_{k-1}$. If there exists some $\alpha, \beta > 0$ such that
$$N^\alpha \leq r_m(N) \leq N^\beta$$
holds for each $m = 1, \dots, k-1$ and sufficiently large $N$, Then the WW functions of power type with respect to $\widetilde W_N^k$ are exactly equal to the classical $k$-th order WW functions of power type; ie, we have
$$\left\{ f\in L^\infty(\mu) : \widetilde W_N^k(f) \approx 0 \right\} = \left\{ f\in L^\infty(\mu) : W_N^k(f) \approx 0 \right\} \, .$$
Specifically, we see that
\begin{align}\label{F:alt WW comparisons, details}
\begin{split}
&W_N^k(f) \in \text{Poly}(\N, \gamma) \quad \implies \quad \widetilde W^k_N(f)\in \text{Poly}\left(\N, \frac{\alpha\gamma}{3 \cdot 2^{k+2}}\right) \\
&\widetilde W_N^k(f) \in \text{Poly}(\N, \gamma) \quad \implies \quad W^k_N(f)\in \text{Poly}\left(\N, \frac{\min\{\alpha,\gamma, 2/3\}}{3 \cdot 2^{k+3} \cdot \lceil \beta \rceil}\right) \, .
\end{split}
\end{align}
\end{theorem}
\begin{remark}
We remark that the constants $\alpha$ and $\beta$ need not be bounded above by 1. Hence, the WW averages constructed with each $h \in [N]$ will also be equivalent in rate to the classical WW average. 
\end{remark}

To show that $\widetilde{W}_N^k(f)$ satisfies a reverse Bourgain bound, we revisit the proof from \S \ref{s:RBB}. 
In the proof of Theorem \ref{T:reverse BB, simplest case} where $k=1$, when applying the Van der Corput inequality we chose $H=\floor{\sqrt{N}}$ to mirror the classical WW averages. However, we note that the proof could be carried through without choosing $H$, yielding a remainder term of the order $\left( 1/H + H/N\right)^{1/3}$. Specifically, we get the following:

\begin{lemma}\label{T:RBB for alt WW averages}
Let $k\in \N$. Then there exists a constant $C_k'$ where for any $f\in L^\infty(\mu)$ bounded by 1 and any $\widetilde W_N^k$ constructed from $r_1(N), \dots, r_{k-1}(N)$, we have 
\begin{align*}
\widetilde W_N^k(f) &\leq C_k'\left( \frac{1}{(R_k(N))^{1/6}} + \left[ M_{R_k(N)}^k (f)\right]^{1/6}\right) 
\end{align*}
where $R_k(N) = \floor{\min\{r_1(N), \dots, r_{k-1}(N), N\}^{1/2}}$
\end{lemma}
\begin{remark}
Just as easily, we get the same bound for off-diagonal averages of $\widetilde W_N^k$, as seen in Theorem \ref{T:WW control of off-diag}.
\end{remark}

Lemma \ref{T:RBB for alt WW averages} is shown in the appendix. With it, we can bound $\widetilde W_{N}^k$ by the classical WW average $W_N^k$:

\begin{theorem}\label{T:bound tilde by classical}
Let $k\in \N$. Then there exists a constant $C_k''$ and $N_k$ such that for all invertible dynamical systems $(X, \mathcal{F},\mu, T)$, all WW averages $\widetilde W_N^k$ constructed from $r_1, \dots, r_{k-1}$, and for all $f\in L^\infty(\mu)$ bounded by 1, we have
\begin{align}\label{E:bound tilde by classical}
\widetilde W_N^k(f) &\leq C_k'' \left( \frac{1}{(R_k(N))^{1/(3\cdot 2^{k+1})}} + \left[W_{R_k(N)}^k (f)\right]^{1/(3\cdot 2^{k})} \right)
\end{align}
for all $N$ with $R_k(N) = \floor{\min\{r_1(N), \dots, r_{k-1}(N), N\}^{1/2}} \geq N_k$.

Specifically, if every $r_m(N) \geq N^\alpha$ for some $0< \alpha \leq 1$, it follows that
$$\widetilde{W}_N^k(f) \precsim_{r, k} W_N^k(f)\, .$$
\end{theorem}
\begin{proof}
Chaining together the reverse Bourgain bound on $\widetilde W_N^k$ from Theorem \ref{T:RBB for alt WW averages} with the Bourgain bound (\ref{E:BB for M, W}) on $W_N^k$ , we obtain
\begin{align*}
\widetilde W_N^k(f) &\leq C_k'\left( \frac{1}{(R_k(N))^{1/6}} + \left[ M_{R_k(N)}^k (f)\right]^{1/6}\right) \\
&\leq C_k'\left( \frac{1}{(R_k(N))^{1/6}} + \left[  C_{k} \left( \frac{1}{(R_k(N))^{1/2^k}} + \left[W_{R_k(N)}^k (f)\right]^{1/2^{k-1}} \right)\right]^{1/6}\right)
\end{align*}
for $R_k(N) = \floor{\min\{r_1(N), \dots, r_{k-1}(N), N\}^{1/2}} \geq N_k$. Consolidating remainder terms gives the desired estimate.

In the case that each $r_m(N) \geq N^\alpha$, then $R_k(N)$ meets the requirements of the function $\phi$, and the special asymptotic notation may be used.
\end{proof}
Recall that the limiting behavior of the classical WW averages $W_N^k$ is bound by the Host-Kra-Gowers seminorms, as given by (\ref{E:bound WW by HKZ seminorm}). Hence, any $\widetilde W_N^k$ will inherit this limiting behavior by taking the limsup of the previous estimate (\ref{E:bound tilde by classical}), so long as the remainder term vanishes:

\begin{corollary}
Let $k\in \N$. Then there exists a constant $C_k'''$ such that for all invertible dynamical systems $(X, \mathcal{F}, \mu, T)$, all WW averages $\widetilde W_N^k$ constructed from $r_1, \dots, r_{k-1}$ such
$$\lim_{N} r_m(N) = \infty$$
for each $1\leq m \leq k-1$, and for all $f\in L^\infty(\mu)$ bounded by 1, we have
\begin{align*}
\limsup_N \widetilde W_N^k(f) \leq C_k''' \VERT f \VERT_{k+1}^{1/(9 \cdot 2^{k-1})} \, .
\end{align*}
\end{corollary}
\begin{proof}
Since for any fixed $k$ there are finitely many functions $r_m(N)$, knowing that all $r_m$ diverge to infinity is enough to show that the remainder term
$$\frac{1}{\min\{r_1(N), \dots, r_{k-1}(N), N\}^{1/(3 \cdot 2^{k+2})}}$$
converges to 0.
\end{proof}

As noted previously, if each $r_m$ grows like a polynomial, the estimate (\ref{E:bound tilde by classical}) transfers polynomial decay from $W_N^k$ to $\widetilde{W}_N^k$. To transfer in the other direction, it would suffice to establish a Bourgain bound for $\widetilde W_N^k$, which we could chain together with the reverse Bourgain bound (\ref{T:reverse BB, simplest case}). Towards this, we establish the following technical result:

\begin{lemma}[Bourgain bound for alternative WW averages]\label{L:BB for alt averages}
Let $k\in \N$ and $\beta \in \N$ There exists a constant $C_{k, \beta}$ and $N_{k, \beta}$ such that for all invertible dynamical systems $(X, \mathcal{F}, \mu, T)$ and $k$-th order WW averages constructed from $r_1, \dots, r_{k-1}$ satisfying $r_m(N) \leq N^\beta$ for each $m$, and for any $f\in L^\infty(\mu)$, we have
\begin{align*}
M_N^k(f) &\leq C_{k, \beta}\left(\sum_{m=1}^{k-1}\left[\frac{1}{r_m(\floor{N^{1/\beta}})} + \frac{r_m(\floor{N^{1/\beta}})}{N} \right]^{1/2^{k-1}} + \frac{1}{N^{1/(3\cdot \beta \cdot 2^{k-2})}}+ \left[\widetilde W_{\floor{N^{1/\beta}}}^k(f)\right]^{1/2^{k-1}} \right)
\end{align*}
for all $N\geq N_{k, \beta}$.
\end{lemma}

\begin{proof}
We take all relevant ideas from \cite{assani2024higherorderwienerwintnersystems}. Consider first for any \textit{integer} $\beta\geq 2$, we have
\begin{align}\label{E:shrinking}
W_N^1(g) &\leq  W_{\floor{N^{1/\beta}}}^1(g) + \left(\frac{\beta}{N^{1/\beta}}\right)^{2/3} \, .
\end{align}
This is shown in \cite[Theorem 6.3, $3a \implies 4a$]{assani2024higherorderwienerwintnersystems} for the case $\beta=2$, and holds for greater $\beta$ by the same argument.

As each $r_m(N) \leq N^\beta$, it follows that $r_m(\floor{N^{1/\beta}}) \leq N$ for all $N$. Hence, we can follow the same argument of the Bourgain bound in \cite{assani2024higherorderwienerwintnersystems}, picking $H = r_m(\floor{N^{1/\beta}})$ instead of $\floor{\sqrt{N}}$ in each inductive application of the Van der Corput inequality. This will form remainder terms of the form $(1/H + H/N)$, which do not balance out, but give us the following estimate:
\begin{align*}
M_N^k(f) &\leq C_{k} \sum_{m=1}^{k-1}\left[\frac{1}{r_m(\floor{N^{1/\beta}})} + \frac{r_m(\floor{N^{1/\beta}})}{N} \right]^{1/2^{k-1}} \\
&+ C_k\left(\frac{1}{\prod_{m=1}^{k-1} r_m(\floor{N^{1/\beta}})} \sum_{h_1 = 1}^{r_1(\floor{N^{1/\beta}})} \dots \sum_{h_{k-1} = 1}^{r_{k-1}(\floor{N^{1/\beta}})} W_N^1\left( \prod_{\eta\in V_{k-1}} c^{|\eta|} f \circ T^{\eta \cdot h}\right)\right)^{1/2^{k-1}}
\end{align*}
for $N\geq N_k$. Applying the above estimate (\ref{E:shrinking}), we have constructed exactly the average $\widetilde W_{\floor{N^{1/\beta}}}^k(f)$ and another remainder term of the order $1/N^{1/(3 \cdot \beta \cdot 2^{k-2})}$. Consolidating the remainder terms gives the desired average.
\end{proof}
Chaining this together with the reverse Bourgain bound from Theorem \ref{T:reverse BB, simplest case} obtains control of $W_N^k(f)$ by $\widetilde W_N^k(f)$. However, the precise statement of this relationship is exceedingly technical. Since our current interest is to transfer polynomial decay from $\widetilde W_N^k(f)$ to $ W_N^k(f)$, we specifically analyze the case in which $\widetilde W_N^k(f) \precsim 0$ in order to prove Theorem \ref{T:Equivalence of WW averages}:
\begin{proof}
The first line of (\ref{F:alt WW comparisons, details}) follows from Theorem \ref{T:bound tilde by classical}. For the second line, suppose that $\widetilde W_N^k$ is constructed from $r_1, \dots, r_{k-1}$ each satisfying $N^\alpha \leq r_m(N) \leq N^\beta$, and that $f$ has $\widetilde W_N^k(f) \leq \frac{1}{N^\gamma}$ for $\gamma>0$. By applying the Bourgain bound for $\widetilde W_N^k$ using $\beta' = \lceil \beta\rceil + 1$, we have for sufficiently large $N$ that
\begin{align*}
&M_N^k(f) \\
&\leq C_{k, \beta}\left(\sum_{m=1}^{k-1}\left[\frac{1}{r_m(\floor{N^{1/(\lceil\beta\rceil + 1)}})} + \frac{r_m(\floor{N^{1/(\lceil\beta\rceil + 1)}})}{N} \right]^{1/2^{k-1}} + \frac{1}{N^{1/(3\cdot (\lceil\beta\rceil + 1) \cdot 2^{k-2})}}+ \left[\widetilde W_{\floor{N^{1/(\lceil\beta\rceil + 1)}}}^k(f)\right]^{1/2^{k-1}} \right) \\
&\leq C_{k, \beta}\left(\sum_{m=1}^{k-1}\left[\frac{1}{N^{\alpha/(\lceil\beta\rceil + 1)}} + \frac{N^{\beta / (\lceil\beta \rceil+1)}}{N} \right]^{1/2^{k-1}} + \frac{1}{N^{1/(3\cdot (\lceil\beta\rceil + 1) \cdot 2^{k-2})}}+ \left[\frac{1}{N^{\gamma/(\lceil\beta\rceil + 1)}}\right]^{1/2^{k-1}} \right)\, .
\end{align*}
After some severe consolidation, we see that this term is $O(N^{-\delta})$ where
$$\delta = \min\left\{ \frac{\alpha}{2^{k} \lceil\beta\rceil}, \frac{1}{3 \cdot 2^{k-1} \lceil\beta\rceil}, \frac{\gamma}{2^k \lceil\beta\rceil} \right\} \, .$$
As $M_N^k(f) = O(N^{-\delta})$, it follows by the reverse Bourgain bound that $W_N^{k}(f)$ is $O(N^{-\delta /24})$, completing the second line of (\ref{F:alt WW comparisons, details})
\end{proof}

\section{Expanded convergence results for WW functions}\label{S:Conv results}

\subsection{Uniform polynomial WW bound and polynomial return times}\label{S:Polynomial bounds}

Pointwise convergence for multilinear averages over a family $C$ of weights given by
$$\frac{1}{N}\sum_{n=1}^N c_n \prod_{j=1}^J f_j(T^{a_j n} x)$$
have been long studied as natural generalizations of the classical Wiener-Wintner theorem, which corresponds to the case where $C = \{ e^{2\pi i n t} : t\in \R\}$ and $J = 1$, and its own generalization in the Return times theorem, corresponding to $C = \{g(S^n y) : g\in L^\infty(\nu) \text{ for some system }(Y, \mathcal{G}, \nu, S)\}$ and $J=1$. The linear case $J=1$ is long studied and many such generalizations and variants have been shown. For $J=2$, the case of weights $\{ e^{2\pi i n t} : t\in \R\}$ was shown by Assani, Duncan, and Moore \cite{ADM2016}, and the case of weights with polynomial phase $\{e^{2\pi i P(n)}\}$ by Assani and Moore \cite{AssaniMoore}. In the case that the system $(X, \mathcal{F}, \mu, T)$ satisfies pointwise convergence over $J$-th order recurrence averages for some $J\in \N$, Zorin-Kranich \cite{ZK2015} shows a uniform convergence over nilsequence weights.

Assani, Folks, and Moore \cite{assani2024higherorderwienerwintnersystems} establish that $J$-th order WW averages control the supremum of the averages over weights $e^{2\pi i n t}$ in norm. This argument can be inductively lifted to polynomial phases:

\begin{theorem} \label{T:polynomial WW estimate}
Let $J \in \N$ and $k\in \N \cup \{0\}$, and $a_1, \dots, a_J\in \mathbb{Z}$ be distinct and nonzero. Then there exists a constants $C_{J, k, a}$ and $N_{J, a}$ such that for all invertible dynamical system $(X, \mathcal{F}, \mu, T)$ and $f_1, \dots, f_J\in L^\infty (\mu)$, all bounded by 1, we have
\begin{align}\label{E:poly bound}
\left\Vert \sup_{t_1, \dots, t_k} \left|\frac{1}{N}\sum_{n=1}^N e^{2\pi i p_{k, t}(n)} \prod_{j=1}^J f_j \circ T^{a_j n} \right| \right\Vert_2 &\leq C_{J, k, a}\left(\frac{1}{\floor{\sqrt{N}}^{1/2^{k+J-2}}} + \left[W_N^{k+J-1}(f_1)\right]^{1/2^{k+J-2}}\right)
\end{align}
for all $N> N_{J, a}$, where $p_{k, t}(n) = t_1 n + t_2n^2 + \dots t_kn^k$ is the polynomial with coefficients $t_i$.
\end{theorem}

\begin{remark}
From this estimate, by the usual summability argument we may immediately establish almost everywhere convergence of the averages
$$\sup_{t_1, \dots, t_k} \left|\frac{1}{N}\sum_{n=1}^N e^{2\pi i p_{k, t}(n)} \prod_{j=1}^J f_j \circ T^{a_j n} \right|$$
to zero for $f_1$ in the $L^2$ closure of the $k+J-1$-th order WW functions. We note that this convergence falls under the result of Zorin-Kranich \cite{ZK2015} under the condition of the pointwise convergence of multiple recurrence averages, which WW systems have already been shown to satisfy.
\end{remark}

\begin{remark}
We remark that the above bound (\ref{E:poly bound}) does not, in general, classify the uniform characteristic factor for the multilinear averages with polynomial weights. In the case $J=1$, the factor of uniform convergence of the averages
$$\frac{1}{N}\sum_{n=1}^N e^{2\pi i n p_{k, t}(n)}f(T^nx)$$
to zero over all polynomials $p_{k, t}(n)$ is the level $k$ quasi-eigenfunctions, as shown by Frantzikinakis \cite{FRANTZIKINAKIS_2006}, which is not necessarily equal to the $k$-th Host-Kra-Ziegler factor. 
\end{remark}

\begin{proof}
Consider induction on $k$, the degree of the polynomial. The base case $k=0$ is the Bourgain bound (\ref{E:BB for M, W}).

To induct, we apply the Van der Corput inequality pointwise with $N > |a_1|^2$ and $H = \floor{\frac{\sqrt{N}}{|a_1|}}$.
After taking the supremum over $t_1, \dots, t_k$ and the integral in $\mu$ and extending the $n$ sum to $N$, we have
\begin{align*}
&\left\Vert \sup_{t_1, \dots, t_k} \left|\frac{1}{N}\sum_{n=1}^N e^{2\pi i p_{k, t}(n)} \prod_{j=1}^J f_j \circ T^{a_j n} \right| \right\Vert_2^2 \\
&\leq \frac{6}{\floor{\frac{\sqrt{N}}{|a_1|}}} + \frac{4}{\floor{\frac{\sqrt{N}}{|a_1|}}}\sum_{h=1}^{\floor{\frac{\sqrt{N}}{|a_1|}}} \int \sup_{t_1, \dots, t_k} \left| \frac{1}{N}\sum_{n=1}^{N} e^{2\pi i (p_{k, t}(n) - p_{k, t}(n+h))} \prod_{j=1}^J f_j\circ T^{a_jn} \cdot \overline{f_j \circ T^{a_jn + a_jh}}\right|\, d\mu  \, .
\end{align*}

We observe that for any $h$, the polynomial $p_{k, t}(n) - p_{k, t}(n+h)$ is one degree lower in $n$. Moreover, any terms of $p_{k, t}(n) - p_{k, t}(n+h)$ which only depend on $h$ and not $n$ can be factored out of the $n$ sum and vanish. Hence, taking a supremum over $p_{k, t}(n) - p_{k, t}(n+h)$ is bounded by taking a supremum over all expression $p_{k-1, t}(n)$:
\begin{align*}
\leq \frac{6}{\floor{\frac{\sqrt{N}}{|a_1|}}} + \frac{4}{\floor{\frac{\sqrt{N}}{|a_1|}}}\sum_{h=1}^{\floor{\frac{\sqrt{N}}{|a_1|}}} \int \sup_{t_1, \dots, t_{k-1}} \left| \frac{1}{N}\sum_{n=1}^{N} e^{2\pi i p_{k-1, t}(n)} \prod_{j=1}^J [f_j \cdot \overline{f_j \circ T^{a_jh}}] \circ T^{a_jn} \right|\, d\mu \, .
\end{align*}

By Hölder's inequality and the inductive hypothesis, these may be bounded for $N > N_{J-1, a}$:
\begin{align*}
&\leq \frac{6}{\floor{\frac{\sqrt{N}}{|a_1|}}} + \frac{4}{\floor{\frac{\sqrt{N}}{|a_1|}}}\sum_{h=1}^{\floor{\frac{\sqrt{N}}{|a_1|}}} \left\Vert \sup_{t_1, \dots, t_{k-1}} \left| \frac{1}{N}\sum_{n=1}^{N} e^{2\pi i p_{k-1, t}(n)} \prod_{j=1}^J [f_j \cdot \overline{f_j \circ T^{a_jh}}] \circ T^{a_jn} \right| \right\Vert_2 \\
&\leq \frac{6|a_1|}{\floor{\sqrt{N}}} + \frac{4}{\floor{\frac{\sqrt{N}}{|a_1|}}}\sum_{h=1}^{\floor{\frac{\sqrt{N}}{|a_1|}}} C_{k-1} \left( \frac{1}{\floor{\sqrt{N}}^{1/2^{k+J-3}}} + \left[W_{N}^{k+J-2}\left( f_1\cdot \overline{f_1 \circ T^{a_1h}} \right)\right]^{1/2^{k+J-3}} \right) \\
&\leq \frac{6|a_1|+4C_{k-1}}{\floor{\sqrt{N}}^{1/2^{k+J-3}}} + 4C_{k-1}|a_1| \cdot \frac{1}{\floor{\sqrt{N}}}\sum_{h=1}^{\floor{\frac{\sqrt{N}}{|a_1|}}} \left[W_{N}^{k+J-2}\left( f_1\cdot \overline{f_1 \circ T^{a_1h}} \right) \right]^{1/2^{k+J-3}} \, .
\end{align*}
By extending the $h$ sum to $\floor{\sqrt{N}}$, we remove the scaling by $a_1$ inside. After moving the $h$ sum inside by Hölder's inequality, we have constructed $W_N^{k+J-1}$:
\begin{align*}
&\leq \frac{6|a_1|+4C_{k-1}}{\floor{\sqrt{N}}^{1/2^{k+J-3}}} + 4C_{k-1}|a_1| \cdot \frac{1}{\floor{\sqrt{N}}}\sum_{h=1}^{\floor{\sqrt{N}}} \left[W_{N}^{k+J-2}\left( f_1\cdot \overline{f_1 \circ T^{h}} \right) \right]^{1/2^{k+J-3}} \\
&\leq \frac{6|a_1|+4C_{k-1}}{\floor{\sqrt{N}}^{1/2^{k+J-3}}} + 4C_{k-1}|a_1| \left( \frac{1}{\floor{\sqrt{N}}}\sum_{h=1}^{\floor{\sqrt{N}}} W_{N}^{k+J-2}\left( f_1\cdot \overline{f_1 \circ T^{h}} \right) \right)^{1/2^{k+J-3}} \\
&= \frac{6|a_1|+4C_{k-1}}{\floor{\sqrt{N}}^{1/2^{k+J-3}}} + 4C_{k-1}|a_1| \left[ W_{N}^{k+J-1}\left( f_1\right) \right]^{1/2^{k+J-3}} \, .
\end{align*}
Taking a square root and consolidating constants gives the desired bound for $N > N_{J, a} := \max\{N_{J, a}, |a_1|^2\}$.
\end{proof}
By a standard application of the spectral theorem, we obtain the following return times theorem:
\begin{theorem}\label{T:polynomial return times}
Let $(X, \mathcal{F}, \mu, T)$ be an invertible dynamical system, and let $f_1 \in L^\infty(\mu)$ be an $L^2$ limit of $k+J-1$-th order WW functions. Then for any functions $f_2, \dots, f_J\in L^\infty(\mu)$ all bounded by 1, there exists a set $X_{f_1, \dots, f_J, k}$ of full measure such that for any $x\in X_{f_1, \dots, f_J, k}$ and $a_1, \dots, a_J$ all distinct and nonzero and any dynamical system $(Y, \mathcal{G},\nu, S)$ and $g\in L^\infty(\nu)$ and any integer valued polynomial $P$ of degree less than or equal to $k$, the averages
\begin{align*}
\frac{1}{N}\sum_{n=1}^N g(S^{P(n)} y) \prod_{j=1}^J f_{j}(T^{a_j n} x)
\end{align*}
converge to zero $\nu$-a.e.
\end{theorem}

\begin{proof}
Since there are countably many possibilities for $a_1, \dots, a_J$, it follows that we may fix them and intersect over a countable collection of full measure sets at the end.

Consider first the case in which $f_1$ is itself a $k+J-1$-th order WW function of power type $\alpha$. Using the previous Theorem \ref{T:polynomial WW estimate}, it follows for sufficiently large $N$ that we have
\begin{align*}
\left\Vert \sup_{t_1, \dots, t_k} \left|\frac{1}{N}\sum_{n=1}^N e^{2\pi i p_{k, t}(n)} \prod_{j=1}^J f_j \circ T^{a_j n} \right| \right\Vert_2 &\leq C_k\left(\frac{1}{\floor{\sqrt{N}}^{1/2^{k+J-2}}} + \left[\frac{C_{f_1}}{N^\alpha}\right]^{1/2^{k+J-2}}\right) \leq \frac{C'}{N^{\alpha / 2^{k + J - 2}}}
\end{align*}
after consolidating constants. Picking $\gamma \in \N$ to satisfy $\alpha \gamma > 2^{k + J}$, it follows that 
\begin{align*}
\int \sup_{t_1, \dots, t_k} \left|\frac{1}{\floor{N^\gamma}}\sum_{n=1}^{\floor{N^\gamma}} e^{2\pi i p_{k, t}(n)} \prod_{j=1}^J f_j (T^{a_j n} x) \right|^2 d\mu(x) &\leq \frac{C'}{\floor{N^\gamma}^{\alpha/2^{k+J-1}}} \leq \frac{C'}{N^2}
\end{align*}
for sufficiently large $N$, where the constant is allowed to change, but picks up no $N$ dependence. These terms are summable in $N$, and by the monotone convergence theorem, we have
$$\sum_{N=1}^\infty \sup_{t_1, \dots, t_k} \left|\frac{1}{\floor{N^\gamma}}\sum_{n=1}^{\floor{N^\gamma}} e^{2\pi i p_{k, t}(n)} \prod_{j=1}^J f_j (T^{a_j n} x) \right|^2 < \infty $$
for all $x$ in a set of full measure $X_{f_1, \dots, f_{J}, k}$.

For any $x\in X_{f_1, \dots, f_J, k}$, applying the spectral theorem for any $g\in L^\infty(\nu)$, we have
\begin{align*}
\left\Vert \frac{1}{\floor{N^\gamma}} \sum_{n=1}^{\floor{N^\gamma}} g \circ S^{P(n)} \cdot \prod_{j=1}^J f_j (T^{a_j n} x)\right\Vert_{L^2(\nu)}^2 &= \int_0^1 \left| \frac{1}{\floor{N^\gamma}} \sum_{n=1}^{\floor{N^\gamma}} e^{2\pi i P(n) t} \prod_{j=1}^J f_j (T^{a_j n} x) \right|^2 \, d\sigma_g(t)  \\
& \leq \Vert g \Vert_2^2 \cdot \sup_t \left| \frac{1}{\floor{N^\gamma}} \sum_{n=1}^{\floor{N^\gamma}} e^{2\pi i P(n) t} \prod_{j=1}^J f_j (T^{a_j n} x) \right|^2 \\
& \leq \Vert g \Vert_2^2 \cdot \sup_{t_1, \dots, t_{k}} \left| \frac{1}{\floor{N^\gamma}} \sum_{n=1}^{\floor{N^\gamma}} e^{2\pi i p_{k, t}(n)} \prod_{j=1}^J f_j (T^{a_j n} x) \right|^2
\end{align*}
and again these terms are summable. By the monotone convergence theorem in $\nu$, it follows that the integrand
\begin{align}\label{F:average along}
\left|\frac{1}{\floor{N^\gamma}} \sum_{n=1}^{\floor{N^\gamma}} g( S^{P(n)} y) \prod_{j=1}^J f_j (T^{a_j n} x) \right|^2
\end{align}
convergence to zero $\nu$-a.e. Since the summands are bounded, we may extend the convergence from the subsequence $\floor{N^\gamma}$ to $N$ by the standard argument. We outline this argument here, for later comparison: Denote the summands in $n$ of the average (\ref{F:average along}) as $a_n$, and observe that $|a_n| \leq 1$. Moreover, let us ignore the square root. Let $N\in \mathbb{N}$ be arbitrary, and pick $M\in \mathbb{N}$ to satisfy
$$M^\gamma \leq N \leq (M+1)^\gamma \, .$$
We observe that
\begin{align*}
\left| \frac{1}{N}\sum_{n=1}^N a_n\right| &\leq \left|\frac{1}{\floor{M^\gamma}}\sum_{n=1}^{\floor{M^\gamma}} a_n\right| + \frac{1}{M^\gamma} \sum_{n = \floor{M^\gamma}+1}^{\floor{(M + 1)^\gamma}} |a_n| \leq \left|\frac{1}{\floor{M^\gamma}}\sum_{n=1}^{\floor{M^\gamma}} a_n\right| + \frac{\floor{(M + 1)^\gamma}-\floor{M^\gamma}}{M^\gamma} 
\end{align*}
As $N\to \infty$, it follows that $M \to \infty$ also: the averages along $\floor{M^\gamma}$ goes to zero by assumption, and the remainder term is $O(M^{-1})$, and also goes to zero. Hence, we obtain convergence along the entire subsequence.

Now, suppose that $F_m$ are $j+K-1$-th order WW functions of power type converging to $f_1$. Without loss of generality, take $\Vert F_m - f_1 \Vert_2< \frac{1}{m}$ for all $m\in \N$. Using the maximal ergodic theorem (here referenced as lemma \ref{L:maxIneq}), we observe for each $m\in \N$ and any $g \in L^\infty(\mu)$ that
\begin{align*}
&\left\Vert \frac{1}{N}\sum_{n=1}^N g\circ S^{P(n)} \cdot \prod_{j=1}^J f_j\circ T^{a_j n} - \frac{1}{N}\sum_{n=1}^N g\circ S^{P(n)} \cdot F_m \circ T^{a_1n} \cdot \prod_{j=2}^J f_j\circ T^{a_j n} \right\Vert_2 \\
&\leq \left\Vert \frac{1}{N}\sum_{n=1}^N \left| f_1 - f_M \right| \circ T^{a_1 n}  \right\Vert_2 \leq 2\Vert f_1 - F_m \Vert_2 \leq \frac{2}{m^2}
\end{align*}
and these terms are summable. By the monotone convergence theorem, it follows that the integrand must be converging to zero almost everywhere as $m\to \infty$.

Hence, for any $x\in \cap_{m=1}^\infty X_{F_m, f_2, \dots, f_J, k}$, and for any particular value of $m$, we see
\begin{align*}
&\limsup_N \left| \frac{1}{N}\sum_{n=1}^N g\circ S^{P(n)} \cdot \prod_{j=1}^J f_j\circ T^{a_j n}\right| \\
&\leq \limsup_N \left| \frac{1}{N}\sum_{n=1}^N g\circ S^{P(n)} \cdot \prod_{j=1}^J f_j\circ T^{a_j n} - \frac{1}{N}\sum_{n=1}^N g\circ S^{P(n)} \cdot F_m \circ T^{a_1n} \cdot \prod_{j=2}^J f_j\circ T^{a_j n} \right| + 0 \\
&\leq \sup_N \left| \frac{1}{N}\sum_{n=1}^N g\circ S^{P(n)} \cdot \prod_{j=1}^J f_j\circ T^{a_j n} - \frac{1}{N}\sum_{n=1}^N g\circ S^{P(n)} \cdot F_m \circ T^{a_1n} \cdot \prod_{j=2}^J f_j\circ T^{a_j n} \right| 
\end{align*}
which converges to zero as we let $m\to \infty$.
\end{proof}

\subsection{Multilinear Ergodic Hilbert Transform}\label{S:Hilb}

The \textit{(two-sided) ergodic Hilbert transform} of a function $f$ is defined as the punctured sum (skipping $n=0$)
$$\sum_{n=-\infty}^\infty\mathop{}^{\mkern-20mu '} \frac{f(T^n x)}{n} \, .$$
This average was studied by Cotlar, who showed that it exists a.e. for any $f\in L^1(\mu)$ \cite{Cotlar}. For such averages, we can make the same multilinear and return times generalizations. For example, pointwise a.e. convergence of the bilinear ergodic Hilbert transform given by 
$$\sum_{n=-\infty}^\infty\mathop{}^{\mkern-20mu '} \frac{f_1(T^n x)f_2(T^{-n}x)}{n}$$
was established by Demeter \cite{Demeter}. A survey of return times theorems of the form 
$$\sum_{n=-\infty}^\infty\mathop{}^{\mkern-20mu '} \frac{f(T^n x)g(S^{n}y)}{n}$$
is given by Assani and Presser \cite{AssaniPresser+2014+19+58}.

In contrast, the \textit{one-sided ergodic Hilbert transform} is taken by restricting to positive $n$:
$$\sum_{n=1}^\infty  \frac{f(T^n x)}{n}$$
and behaves poorly in general. Classically, it is known that in any dynamical system on a non-atomic measure space, there exists an integrable function for which the above diverges in the $L^2$ norm, as shown by Halmos \cite{Halmos}. However, in certain cases and against certain weights, a.e. convergence of one-sided averages has been shown to converge. For example, certain ``twisted", or oscillatory averages of the form
$\sum_{n=1}^\infty \frac{e^{2\pi i p(n)}}{n} f(T^nx)$
have been shown to converge a.e. by Krause, Lacey, and Wierdl \cite{Krause2016OnCO} for sufficiently sparse functions $p$. In the linear case, a.e. convergence of the average 
$$\sum_{n=1}^\infty \frac{e^{2\pi i n t}}{n^\sigma} f(T^n x)$$
for some $0 < \sigma < 1$ was shown for first-order WW functions $f$ by Assani and Nicolaou \cite{AssaniNicolaou}, and convergence of the two sided analogue was shown by Assani to be equivalent to $f$ being a first-order WW function of power type \cite{assani04}. More connections between the decay rates and convergence of one-sided linear ergodic Hilbert transforms are detailed by Assani and Lin \cite{AssaniLin}.

Here, we establish some convergence results for Higher-order WW functions of power type. More specifically, we show that any of the previously established pointwise convergence for ergodic averages for WW functions that have been previously established in \cite{assani2024higherorderwienerwintnersystems} have Hilbert transform analogues. For example, from the bound in Theorem \ref{T:polynomial WW estimate}, we deduce the following:

\begin{theorem}\label{T:HilbertTrans}

Let $(X, \mathcal{F}, \mu, T)$ be an ergodic dynamical system, $k\geq0$, $J\geq 1$, and let $f_1\in L^\infty(\mu)$ be a $J + k - 1$-th order WW function of power type $\alpha > 0$. Then For any $f_2, \dots, f_{J} \in L^\infty(\mu)$,  there exists a set $X_{f_1, \dots, f_{J}, k}$ of full measure such that the series 
$$\sum_{n=1}^\infty \frac{e^{2\pi i \left(t_1 n + t_2 n^2 + \dots + t_k n^k\right)}}{n^\sigma} \, \prod_{j=1}^{J}f_j(T^{a_j n} x)$$
converges for all $x\in X_{f_1, \dots, f_{J}, k}$ and $a_1, \dots, a_{J}\in \mathbb{Z}$, all distinct and nonzero, and for all $t_1, \dots, t_k \in [0, 1]$ and $\sigma\in (1 - \alpha/2^{J+k-1}, 1]$.

Moreover, the limiting function is continuous in $t_1, \dots, t_k$ and $\sigma$.
\end{theorem}
We highlight some special cases. the case in which $k=0$ corresponds to the multilinear ergodic Hilbert transform:
$$\sum_{n=1}^\infty \frac{\prod_{j=1}^{J}f_j(T^{a_j n} x)}{n^\sigma} $$
while the case $k=1$ refers to the multilinear twisted ergodic Hilbert transform:
$$\sum_{n=1}^\infty \frac{e^{2\pi i n t}}{n^\sigma} \prod_{j=1}^{J}f_j(T^{a_j n} x)\, .$$
The case in which $J = 1$ corresponds to the ergodic Hilbert transform with polynomial phase:
$$\sum_{n=1}^\infty \frac{e^{2\pi i \left(t_1 n + t_2n^2 + \dots + t_k n^k\right)}}{n^\sigma} f_1(T^{a_j n} x)\, .$$

\begin{remark}
Since the exponents $a_j$ are allowed to be negative, it follows that that we can obtain the analogous two-sided results by applying this theorem to both sides individually.
\end{remark}

To simplify the proofs, we specify the exact criterion used:

\begin{lemma}\label{L:HilbertCriterion}
Let $(V, | \cdot |)$ be a Banach space. Let $\{a_j\}_{j=1}^\infty \subset V$ be a sequence and denote $A_N = \frac{1}{N}\sum_{n=1}^N a_n$. Let $\sigma \in (0, 1]$ be such that $\sum_{N=1}^\infty \frac{|A_N|}{N^\sigma} < \infty$ and $N^{1-\sigma}A_N$ is convergent. Then the sequence of partial sums $S_N := \sum_{n=1}^N \frac{a_n}{n^\sigma}$ is convergent.
\end{lemma}
We remark that the Kronecker lemma provides a partial converse, that if $S_N$ converges, then $N^{1-\sigma}A_N$ converges to zero. While this proof does not require $N^{1-\sigma}A_N$ to converge to zero specifically, it will follow from the conclusion that it must be converging to zero.

\begin{proof}
Recall by definition that in any metric space a sequence $\{v_j\} \subset V$ is Cauchy if and only if
$$\lim_{N} \sup_{M > N} |v_M - v_N| = 0 \, .$$
Let $N\in \N$ and $M > N$. Making use of the fact that
$$(x+1)^\sigma - x^\sigma \leq x^{\sigma - 1}$$
for any $x\geq1$ and $0 < \sigma \leq 1$, we observe that
\begin{align*}
 \left| S_{M} - S_{N}\right| &= \left|\sum_{n=N+1}^M \frac{nA_n - (n-1)A_{n-1}}{n^\sigma}\right| \\
&= \left|\sum_{n=N+1}^M \frac{nA_n}{n^\sigma} - \sum_{n=N}^{M-1} \frac{nA_n}{(n+1)^\sigma}\right| \\
&= \left|\sum_{n=N+1}^{M-1} nA_n\left(\frac{1}{n^\sigma} - \frac{1}{(n+1)^\sigma}\right) + \frac{MA_M}{M^\sigma} - \frac{NA_N} {(N+1)^\sigma} + \frac{NA_N}{N^\sigma} - \frac{NA_N}{N^\sigma}\right| \\
&\leq \sum_{n=N}^{M-1} n|A_n|\left(\frac{(n+1)^\sigma - n^\sigma}{n^{2\sigma}}\right) + \left|\frac{MA_M}{M^\sigma} - \frac{NA_N}{N^\sigma}\right| \\
&\leq \sum_{n=N}^{M-1} n|A_n|\left(\frac{1}{n^{\sigma + 1}}\right) + \left|M^{1-\sigma}A_M - N^{1-\sigma}A_N\right| \\
&\leq \sum_{n=N}^{\infty} \frac{|A_n|}{n^\sigma} + \left|M^{1-\sigma}A_M - N^{1-\sigma}A_N\right|
\end{align*}
Hence,
$$\limsup_N \sup_{M > N}\left| S_{M} - S_{N}\right| \leq \limsup_N \sum_{n=N}^{\infty} \frac{|A_n|}{n^\sigma} + \limsup_N \sup_{M > N} \left|M^{1-\sigma}A_M - N^{1-\sigma}A_N\right| = 0\, .$$
and $S_N$ is Cauchy.
\end{proof}

\begin{proof}[Proof of \ref{T:HilbertTrans}]
Since there are only countably many choices of exponents $a_j$, it follows that we can fix the exponents, obtain a set of full measure, and intersect over all possibilities at the end.

Let $f_1$ satisfy $W^{k+J-1}_N(f_1) \leq CN^{-\alpha}$ for some $0<\alpha<1/2$, and let $\ep > 0$ be such that 
$$\sigma := 1- \alpha/2^{k+J-1} + \ep \leq 1 \, .$$ 
Using the polynomial WW bound (Theorem \ref{T:polynomial WW estimate}) we see that for any $f_2, \dots, f_{J}\in L^\infty(\mu)$, all uniformly bounded by 1, and sufficiently large $N$, we have
$$\left\Vert \sup_{t_1, \dots, t_k}\left|\frac{1}{N}\sum_{n=1}^N e^{2\pi i p_{k, t}(n)} \prod_{j=1}^{J} f_j \circ T^{a_j n}\right| \right \Vert_1 \leq C_{a, J} \left( \frac{1}{N^{1/2^{k+J-1}}} + \left[W_N^{k+J-1} (f_1)\right]^{1/2^{k+J-2}} \right) \leq \frac{C}{N^{\alpha/2^{k+J-2}}}$$
for a potentially larger value of $C$, where $p_{k, t}(n) = t_1n + \dots + t_kn^k$. By the monotone convergence theorem, we note that
\begin{align*}
&\int \sum_{N=1}^\infty \frac{\sup_{t_1, \dots, t_k}\left|\frac{1}{N}\sum_{n=1}^N e^{2\pi i p_{k, t}(n)} \prod_{j=1}^{J} f_j \circ T^{a_j n}\right|}{N^\sigma} \, d\mu \\&=  \sum_{N=1}^\infty \frac{\left\Vert\sup_{t_1, \dots, t_k}\left|\frac{1}{N}\sum_{n=1}^N e^{2\pi i p_{k, t}(n)} \prod_{j=1}^{J} f_j \circ T^{a_j n}\right| \right\Vert_1}{N^\sigma} 
\end{align*}
As we have just established, these summands are eventually bounded by 
$$\frac{C}{N^{\sigma + \alpha/2^{k+J - 2}}} = \frac{C}{N^{1 + \alpha/2^{k+J-1} + \ep}}$$
which is summable. Hence, the integrand  
$$\sum_{N=1}^\infty \frac{\sup_{t_1, \dots, t_k}\left|\frac{1}{N}\sum_{n=1}^N e^{2\pi i p_{k, t}(n)} \prod_{j=1}^{J} f_j \circ T^{a_j n}\right|}{N^\sigma} < \infty$$ for almost all $x\in X$. This meets the first condition of the lemma \ref{L:HilbertCriterion}.

Similarly, we observe by the polynomial WW estimate that for sufficiently large $N$, 
\begin{align*}
\int \sup_{t_1, \dots, t_k} \left| \frac{1}{N^\sigma} \sum_{n=1}^N e^{2\pi i p_{k, t}(n)} \prod_{j=1}^{J} f_j(T^{a_jn}x)\right| \, d\mu & \leq \frac{C}{N^{\sigma -1 + \alpha/2^{k+J-1}}} = \frac{C}{N^{\alpha/2^{k+J-1} + \ep}} \, .
\end{align*}
Hence, for $\gamma := 2^{k+J-1}/\alpha$ this integral is summable over the subsequence $\floor{N^\gamma}$. Hence, by the monotone convergence theorem the integral is finite almost everywhere, and the summands must be converging to zero:
\begin{align*}
\lim_N \sup_{t_1, \dots, t_k}\left|\frac{1}{\floor{N^\gamma}^\sigma} \sum_{n=1}^{\floor{N^\gamma}}e^{2\pi i p_{k, t}(n)} \prod_{j=1}^J f_j(T^{a_j n} x)\right| =0 \, .
\end{align*}
By our choice of $\gamma$ this convergence extends from the subsequence $N^\gamma$ to all of $N$ by the usual argument; from the comparison $M^\gamma < N \leq (M+1)^\gamma$, since all of the summands are bounded the remainder term has the form $M^{\gamma - 1 - \gamma \sigma} = M^{-(2^{k+J-1}/\alpha)\cdot \ep} $, which converges to zero as $M\to \infty$, and the second condition of lemma \ref{L:HilbertCriterion} is met.

Taking $X_{f_1, \dots, f_J, k}$ to be the set of full measure where both of these conditions hold, it follows that for any $x\in X_{f_1, \dots, f_J, k}$ that
\begin{align*}
\sum_{n=1}^\infty \frac{e^{2\pi i (t_1n + \dots + t_k n^k)} \prod_{j=1}^J f_j(T^{a_j n}x)}{n^\sigma}
\end{align*}
converges uniformly in $t_1, \dots, t_k$ for $\sigma$ chosen to be $1 - \alpha/2^{k+J-1} + \ep$. As remarked previously, this implies convergence for all $\sigma \in (1 - \alpha/2^{k+J-1} + \ep, 1]$. 

By the classical theory of Dirichlet series, for any such $x$ and $t_1, \dots, t_k$ where the above converges, the series is converging uniformly for all increased $\sigma$ \cite{Queffélec_Queffélec_2020}. Hence, taking $\ep = 1/m$ for sufficiently large $m\in \N$ and intersecting, we obtain a set of full measure where the above converges for all $\sigma \in (1 - \alpha/2^{k+J-1}, 1]$.

\end{proof}

By the same argument with another routine application of the spectral theorem, we obtain the following return times analogue:

\begin{theorem}\label{T:polynomial return times hilbert trans}
Let $(X, \mathcal{F}, \mu, T)$ be an ergodic dynamical system, $k\geq0$, $J\geq 1$, and let $f_1\in L^\infty(\mu)$ be a $J + k - 1$-th order WW function of power type $\alpha > 0$. Then For any $f_2, \dots, f_{J} \in L^\infty(\mu)$, there exists a set $X_{f_1, \dots, f_{J}, k}$ of full measure such that for any $x\in X_{f_1, \dots, f_J, k}$ and for any dynamical system $(Y, \nu, \mathcal{B}, S)$ and $g\in L^\infty(\nu)$ the averages
\begin{align*}
\sum_{n=1}^\infty \frac{g(S^{P(n)}y) \prod_{j=1}^J f_j(T^{a_j n}x)}{n^\sigma}
\end{align*}
converge $\nu$-a.e for all integer valuded polynomials $P$ of degree less than or equal to $k$ and $a_1, \dots, a_J\in \mathbb{Z}$, all distinct and nonzero, and $\sigma\in (1 - \alpha/2^{k+J-1}, 1]$.
\end{theorem}

\begin{proof}
We proceed in the same manner as Theorem \ref{T:HilbertTrans}.

Let $f_1$ satisfy $W^{k+J-1}_N(f_1) \leq CN^{-\alpha}$ for some $0<\alpha<1/2$, and again let $\ep > 0$ be such that 
$$\sigma := 1- \alpha/2^{k+J-1} + \ep \leq 1 \, .$$ 
As shown before by the polynomial WW estimate, we have 
$$\sum_{N=1}^\infty \frac{\sup_{t_1, \dots, t_k}\left|\frac{1}{N}\sum_{n=1}^N e^{2\pi i p_{k, t}(n)} \prod_{j=1}^{J} f_j \circ T^{a_j n}\right|}{N^\sigma} < \infty$$ for almost all $x\in X$. For any such $x$, we apply the monotone convergence theorem (in $\nu$) to see
\begin{align*}
&\int \sum_{N=1}^\infty \frac{\left| \frac{1}{N} \sum_{n=1}^N g \circ S^{P(n)} \cdot \prod_{j=1}^J f_j (T^{a_jn} x)\right|}{N^\sigma} \, d\nu \\
&\leq  \sum_{N=1}^\infty \frac{\left\Vert \frac{1}{N} \sum_{n=1}^N g \circ S^{P(n)} \cdot \prod_{j=1}^J f_j (T^{a_jn} x)\right\Vert_{L^2(\nu)}}{N^\sigma} \\
&= \sum_{N=1}^\infty \frac{\left\Vert \frac{1}{N} \sum_{n=1}^N e^{2\pi i P(n) t} \cdot \prod_{j=1}^J f_j (T^{a_jn} x)\right\Vert_{L^2(\sigma_g)}}{N^\sigma}  \\
&\leq \sum_{N=1}^\infty \frac{\Vert g\Vert_2 \cdot \sup_{t}\left| \frac{1}{N} \sum_{n=1}^N e^{2\pi i P(n) t} \cdot \prod_{j=1}^J f_j (T^{a_jn} x)\right|}{N^\sigma} \\
&\leq \sum_{N=1}^\infty \frac{\Vert g\Vert_2 \cdot \sup_{t_1, \dots, t_k}\left| \frac{1}{N} \sum_{n=1}^N e^{2\pi i p_{k, t}(n)} \cdot \prod_{j=1}^J f_j (T^{a_jn} x)\right|}{N^\sigma} <\infty \, .
\end{align*}
Hence, for $\mu$-a.e. $x\in X$, it must be the case that
$$\sum_{N=1}^\infty \frac{\left| \frac{1}{N} \sum_{n=1}^N g(S^{P(n)y}) \cdot \prod_{j=1}^J f_j (T^{a_jn} x)\right|}{N^\sigma} < \infty$$
for $\nu$-a.e. $y\in Y$. This meets the first condition of the lemma \ref{L:HilbertCriterion}.

As before, we also have for sufficiently large $N$ that
\begin{align*}
\int \sup_{t_1, \dots, t_k} \left| \frac{1}{N^\sigma} \sum_{n=1}^N e^{2\pi i p_{k, t}(n)} \prod_{j=1}^{J} f_j(T^{a_jn}x)\right| \, d\mu & \leq \frac{C}{N^{\alpha/2^{k+J-1} + \ep}}
\end{align*}
and for $\gamma := 2^{k+J-1}/\alpha$ this is summable over the subsequence $\floor{N^\gamma}$. Again, we compute by the monotone convergence theorem that
\begin{align*}
&\int \sum_{N=1}^\infty\left|\frac{1}{\floor{N^\gamma}^\sigma} \sum_{n=1}^{\floor{N^\gamma}} g\circ S^{P(n)} \cdot \prod_{j=1}^J f_j(T^{a_j n} x)\right| \, d\nu \\
&\leq \sum_{N=1}^\infty\left\Vert\frac{1}{\floor{N^\gamma}^\sigma} \sum_{n=1}^{\floor{N^\gamma}} g\circ S^{P(n)} \cdot \prod_{j=1}^J f_j(T^{a_j n} x)\right\Vert_{L^2(\nu)} \\
&\leq \sum_{N=1}^\infty\left\Vert\frac{1}{\floor{N^\gamma}^\sigma} \sum_{n=1}^{\floor{N^\gamma}} e^{2\pi i P(n)t} \prod_{j=1}^J f_j(T^{a_j n} x)\right\Vert_{L^2(\sigma_g)} \\
&\leq \sum_{N=1}^\infty \Vert g \Vert_2 \cdot \sup_t\left|\frac{1}{\floor{N^\gamma}^\sigma} \sum_{n=1}^{\floor{N^\gamma}} e^{2\pi i P(n)t} \prod_{j=1}^J f_j(T^{a_j n} x)\right| \\
&\leq \sum_{N=1}^\infty \Vert g \Vert_2 \cdot \sup_{t_1, \dots, t_k}\left|\frac{1}{\floor{N^\gamma}^\sigma} \sum_{n=1}^{\floor{N^\gamma}} e^{2\pi p_{k, t}(n)} \prod_{j=1}^J f_j(T^{a_j n} x)\right| < \infty
\end{align*}
in which the integrand must be finite $\nu$-a.e. and
\begin{align*}
\lim_N\left|\frac{1}{\floor{N^\gamma}^\sigma} \sum_{n=1}^{\floor{N^\gamma}} g(S^{P(n)}y) \cdot \prod_{j=1}^J f_j(T^{a_j n} x)\right| =0
\end{align*}
holds. For the same reasons as in the proof of Theorem \ref{T:HilbertTrans}, this convergence extends from the subsequence $N^\gamma$ to all of $N$ by the usual argument, and the second condition of lemma \ref{L:HilbertCriterion} is met.

Taking $X_{f_1, \dots, f_J}$ to be the set of full measure where both of these conditions hold, it follows that for any $x\in X_{f_1, \dots, f_J}$ that
\begin{align*}
\sum_{n=1}^\infty \frac{g(S^{P(n)}y) \prod_{j=1}^J f_j(T^{a_j n}x)}{n^\sigma}
\end{align*}
converges $\nu$-a.e. for $\sigma$ chosen to be $1 - \alpha/2^{k+J-1} + \ep$. Again, intersecting over sets for $\ep = 1/m$ gives convergence for all $\sigma \in (1 - \alpha/2^{k+J-1}, 1]$.
\end{proof}

In \cite{assani2024higherorderwienerwintnersystems}, the authors establish a pointwise multilinear return times theorem of the following form
\begin{align*}
\frac{1}{N} \sum_{n=1}^N \prod_{k=1}^K g_{k} (S^{b_k n} y) \prod_{j=1}^J f_{j} (T^{a_j n} x)
\end{align*}
for $f_1$ in the $L^2$ span of the $J+K-1$-th order WW functions of power type. By similar argumentation as the previous theorem, we may also obtain the following:

\begin{theorem}\label{T:mult hilb trans}
Let $(X, \mathcal{F}, \mu, T)$ be an ergodic dynamical system, let $J, K\in \N$, and let $f_1\in L^\infty(\mu)$ be a $J + k - 1$-th order WW function of power type $\alpha > 0$. Then for any $f_2, \dots, f_{J} \in L^\infty(\mu)$, there exists a set $X_{f_1, \dots, f_{J}, K}$ of full measure such that for any $x\in X_{f_1, \dots, f_J, K}$ and for any dynamical system $(Y, \mathcal{G},\nu, S)$ and $g_1, \dots, g_K\in L^\infty(\nu)$ the averages
$$\sum_{n=1}^\infty \frac{\prod_{k=1}^K g_{k} (S^{b_k n} y) \prod_{j=1}^J f_{j} (T^{a_j n} x)}{n^\sigma}$$
converge $\nu$-a.e for all collections $a_1, \dots, a_J\in \mathbb{Z}$ and $b_1, \dots, b_K$, all distinct and nonzero, and $\sigma\in (1 - \alpha/2^{J+K-1}, 1]$.
\end{theorem}

For clarity, we decompose the proof of Theorem \ref{T:mult hilb trans} into two lemmas, from which the conclusion will follow by the same argument as from Theorem \ref{T:polynomial return times hilbert trans}. The first lemma is an intermediary bound shown by the authors in \cite{assani2024higherorderwienerwintnersystems}.
Since some minor adjustments are needed, we present the proof here for the sake of completion. But we remark that all relevant ideas are taken from \cite{assani2024higherorderwienerwintnersystems}.

Let $J, K\in \N$ and let $(X, \mathcal{F},\mu,  T)$ be a dynamical system. For any $N$ and functions $f_1, \dots f_J\in L^\infty(\mu)$ and collection $a_1, \dots, a_J$ of distinct nonzero integers, we denote the following:

\begin{align*}
&F^{J, K, a}_N(f)(x) \\
&=  \frac{1}{\floor{\sqrt{N}}^{1/2^{K-1}}}+ \left(\frac{1}{\floor{\frac{\sqrt{N}}{|a_1|}}^{K-1}} \sum_{h \in \left[\floor{\frac{\sqrt{N}}{|a_1|}}\right]^{K-1}} \sup_{t} \left| \frac{1}{N} \sum_{n=1}^N e^{2\pi i n t} \prod_{j=1}^J \left[\prod_{\eta\in V_{K-1}} c^{|\eta|} f_j \circ T^{a_j (h \cdot \eta)}\right](T^{a_j n} x)\right| \right)^{1/2^{K-1}} \, .
\end{align*}

Concerning this average, we determine the following:

\begin{lemma}\cite[lemma 8.2]{assani2024higherorderwienerwintnersystems}\label{L:Intermediate 1}
Let $J, K\in \N$ and $a_1, \dots, a_J$ be a collection distinct nonzero integers. There exists constants $C_{J, K, a}$ such that for any dynamical systems $(X, \mathcal{F}, \mu, T)$ and $(Y, \mathcal{G}, \nu, S)$ and functions $f_1, \dots f_J\in L^\infty(\mu)$ and $g_1, \dots, g_K\in L^\infty(\nu)$ all uniformly bounded by 1, and for all collections $b_1, \dots, b_K$ of distinct, nonzero integers, and for any $x\in X$, we have
\begin{align*}
\left\Vert \frac{1}{N} \sum_{n=1}^N  \prod_{k=1}^K g_k \circ S^{b_k n} \cdot \prod_{j=1}^J f_j(T^{a_jn}x)\right\Vert_{L^2(\nu)} &\leq C_{J, K, a} \cdot F^{J, K, a}_N(f)(x)
\end{align*}
for all $N > |a_1|^2$ .
\end{lemma}

\begin{proof}
Consider induction on $K$. For the base case $K=1$, we apply the spectral theorem:
\begin{align*}
\left\Vert\frac{1}{N} \sum_{n=1}^N g_1 \circ S^{b_1 n} \cdot \prod_{j=1}^J f_j(T^{a_jn}x)\right\Vert_{L^2(\nu)}^{2} &=  \int \left| \frac{1}{N} \sum_{n=1}^N e^{2\pi i b_1 n t}  \prod_{j=1}^J f_j(T^{a_jn}x)\right|^{2} d\sigma_{g_1}(t) \\
&\leq \Vert g_1 \Vert_{L^2(\nu)}^2 \cdot \sup_{t} \left|  \frac{1}{N} \sum_{n=1}^N e^{2\pi i b_1 n t}  \prod_{j=1}^J f_j(T^{a_jn}x)\right|^2 \\
&\leq \sup_{t} \left|  \frac{1}{N} \sum_{n=1}^N e^{2\pi i n t}  \prod_{j=1}^J f_j(T^{a_jn}x)\right|^2
\end{align*}
as $g_1$ is bounded by assumption, and $b_1$ is lost in the supremum over $t$. Taking a square root gives the desired estimate.

To induct, consider for a larger $K$ applying the Van der Corput inequality pointwise in $y$:
\begin{align*}
&\left|\frac{1}{N} \sum_{n=1}^N  \prod_{k=1}^K g_k(S^{b_k n}y) \cdot \prod_{j=1}^J f_j(T^{a_jn}x)\right|^2 \\
&\leq \frac{2}{H} + \frac{2(N+H)}{N^2H^2}\sum_{h'=1}^H (H-h') \Re \left( \sum_{n=1}^{N-h'} \prod_{k=1}^K g_k(S^{b_k n}y) \overline{g_k(S^{b_kn + b_kh'}y)} \prod_{j=1}^J f_j(T^{a_jn}x)\overline{f_j(T^{a_jn + a_jh'}x)}\right) \, .
\end{align*}
Integrating over $\nu$, we pass the integral inside the $n$ sum, and translate the measure by $-b_1n$. Hence, the $k=1$ multiplicand no longer depends on $n$. Factoring this term and the integral out, and bounding $\Re$ by the absolute value yields
\begin{align*}
&\left\Vert \frac{1}{N} \sum_{n=1}^N  \prod_{k=1}^K g_k \circ S^{b_k n} \cdot \prod_{j=1}^J f_j(T^{a_jn}x)\right\Vert_{L^2(\nu)}^2\\
&\leq \frac{2}{H} + \frac{2(N+H)}{N^2H^2}\sum_{h'=1}^H (H-h') \int \left| \sum_{n=1}^{N-h'} \prod_{k=2}^K [g_k \cdot \overline {g_k \circ S^{b_kh'}}] \circ S^{(b_k - b_1)n} \prod_{j=1}^J f_j(T^{a_jn}x)\overline{f_j(T^{a_jn + a_jh'}x)}\right| d\nu \, .
\end{align*}
With $\Re$ gone, we may clean up terms. As usual, we extend the $n$ sum from $N-h'$ to $N$ at the cost of a $1/H$-order term. Finally, bounding the 1-norm by the 2-norm the inductive hypothesis appear:
\begin{align*}
&\leq \frac{6}{H} + \frac{4}{H}\sum_{h'=1}^H \left\Vert \frac{1}{N}\sum_{n=1}^{N} \prod_{k=2}^K [g_k \cdot \overline {g_k \circ S^{b_kh'}}] \circ S^{(b_k - b_1)n} \prod_{j=1}^J f_j(T^{a_jn}x)\overline{f_j(T^{a_jn + a_jh'}x)}\right\Vert_{L^2(\nu)} \\
&\leq \frac{6}{H} + \frac{4}{H} \sum_{h'=1}^H C_{J, K-1, a}\Bigg( \frac{1}{\floor{\sqrt{N}}^{1/2^{K-2}}} \\
& + \left[\frac{1}{\floor{\frac{\sqrt{N}}{|a_1|}}^{K-2}} \sum_{h\in \left[\floor{\frac{\sqrt{N}}{|a_1|}}\right]^{K-2}} \sup_{t} \left| \frac{1}{N} \sum_{n=1}^N e^{2\pi i n t} \prod_{j=1}^J \left[ \prod_{\eta \in V_{K-2}} c^{|\eta|} (f_j \cdot \overline{f_j \circ T^{a_j h'}})\circ T^{a_j(h \cdot \eta)} \right](T^{a_jn} x)\right|\right]^{1/2^{K-2}}\Bigg) \\
&\leq \frac{6}{H} + \frac{4 C_{J, K-1, a}}{\floor{\sqrt{N}}^{1/2^{K-2}}} + 4 C_{J, K-1, a} \\
&  \cdot\left(\frac{1}{H} \sum_{h'=1}^H \frac{1}{\floor{\frac{\sqrt{N}}{|a_1|}}^{K-2}} \sum_{h\in \left[\floor{\frac{\sqrt{N}}{|a_1|}}\right]^{K-2}} \sup_{t} \left| \frac{1}{N} \sum_{n=1}^N e^{2\pi i n t} \prod_{j=1}^J \left[ \prod_{\eta \in V_{K-2}} c^{|\eta|} (f_j \cdot \overline{f_j \circ T^{a_j h'}})\circ T^{a_j(h \cdot \eta)} \right](T^{a_jn} x)\right|\right)^{1/2^{K-2}}
\end{align*}
after using Hölder's inequality to pull the $h'$ sum inside the power. Picking $H = \floor{\frac{\sqrt{N}}{|a_1|}}$, the remainder terms consolidate into a term of order $1/\floor{\sqrt{N}}^{1/2^{K-2}}$. The sum over $h'$ creates another layer to the cube (which is scaled by $a_j$, and we may simplify relabeling $h$ to be $(h, h')\in [\floor{\frac{\sqrt{N}}{|a_1|}}]^{K-1}$:
\begin{align*}
&\leq \frac{6|a_1| + 4 C_{J, K-1, a}}{\floor{\sqrt{N}}^{1/2^{K-2}}} + 4 C_{J, K-1, a} \\
&  \cdot\left(\frac{1}{\floor{\frac{\sqrt{N}}{|a_1|}}^{K-1}} \sum_{h\in \left[\floor{\frac{\sqrt{N}}{|a_1|}}\right]^{K-1}} \sup_{t} \left| \frac{1}{N} \sum_{n=1}^N e^{2\pi i n t} \prod_{j=1}^J \left[ \prod_{\eta \in V_{K-1}} c^{|\eta|} f_j\circ T^{a_j(h \cdot \eta)} \right](T^{a_jn} x)\right|\right)^{1/2^{K-2}} \, .
\end{align*}
For $N > |a_1|^2$, Consolidating constants and taking a square root of both sides gives the desired bound.
\end{proof}

We also establish the following:
\begin{lemma}\label{T:intermediate 2}
Let $J, K\in \N$ and $a_1, \dots, a_J$ be a collection distinct nonzero integers. There exists constants $C'_{J, K, a}$ such that for any dynamical system $(X, \mathcal{F}, \mu, T)$ and functions $f_1, \dots f_J\in L^\infty(\mu)$, all uniformly bounded by 1, and for any $x\in X$, we have
$$\int F_N^{J, K, a}(f)(x) \, d\mu(x) \leq C'_{J, K, a} \left( \frac{1}{\floor{\sqrt{N}}^{1/2^{J+K-2}}} + \left[ W_N^{J+K-1} (f_1)\right]^{1/2^{J+K-2}}\right)$$
\end{lemma}

\begin{proof}

We use Hölder's inequality to pull the integral to the inside, and bound it by the 2-norm:

\begin{align*}
&\int F_N^{J, K, a, b}(f, g)(x) \, d\mu(x)
\leq \frac{1}{\floor{\sqrt{N}}^{1/2^{K-1}}}\\ &+\left(\frac{1}{\floor{\frac{\sqrt{N}}{|a_1|}}^{K-1}} \sum_{h \in \left[\floor{\frac{\sqrt{N}}{|a_1|}}\right]^{K-1}} \left\Vert \sup_{t} \left| \frac{1}{N} \sum_{n=1}^N e^{2\pi i n t} \prod_{j=1}^J \left[\prod_{\eta\in V_{K-1}} c^{|\eta|} f_j \circ T^{a_j (h \cdot \eta)}\right]\circ T^{a_j n}\right| \right\Vert_2 \right)^{1/2^{K-1}} \, .
\end{align*}
We see that under integration, weighted multiple recurrence averages have appeared. Since these are covered by theorem \ref{T:polynomial WW estimate} with $k=1$, this makes the $J$ order WW averages appear:

\begin{align*}
&\leq \frac{1}{\floor{\sqrt{N}}^{1/2^{K-1}}}\\ &+\left(\frac{1}{\floor{\frac{\sqrt{N}}{|a_1|}}^{K-1}} \sum_{h \in \left[\floor{\frac{\sqrt{N}}{|a_1|}}\right]^{K-1}} C_{J, 1, a}\left( \frac{1}{\floor{\sqrt{N}}^{1/2^{J-1}}} + \left[ W_N^{J}\left( \prod_{\eta\in V_{K-1}} c^{|\eta|} f_1 \circ T^{a_1 (h \cdot \eta)}\right)\right]^{1/2^{J-1}}\right) \right)^{1/2^{K-1}} \, .
\end{align*}
Using subadditivity, we pull out and consolidate remainder terms, and by Hölder's inequality we pull the $h$ sum inside. We get rid of the $a_1$ scaling by extending the $h$ sum to $h\in \left[ \floor{\sqrt{N}}\right]^{K-1}$, and we have added $K-1$ more layers to the cube:
\begin{align*}
&\leq \frac{1+ C_{J, 1, a}^{1/2^{K-1}}}{\floor{\sqrt{N}}^{1/2^{J+K-2}}}+ C_{J, 1, a}^{1/2^{K-1}}\left( \frac{1}{\floor{\frac{\sqrt{N}}{|a_1|}}^{K-1}} \sum_{h \in \left[\floor{\frac{\sqrt{N}}{|a_1|}}\right]^{K-1}} W_N^{J}\left( \prod_{\eta\in V_{K-1}} c^{|\eta|} f_1 \circ T^{a_1 (h \cdot \eta)}\right) \right)^{1/2^{J+K-2}} \\
&\leq \frac{1+ C_{J, 1, a}^{1/2^{K-1}}}{\floor{\sqrt{N}}^{1/2^{J+K-2}}}+ C_{J, 1, a}^{1/2^{K-1}}|a_1|^{1/2^{J+K-2}}\left( \frac{1}{\floor{\sqrt{N}}^{K-1}} \sum_{h \in \left[\floor{\sqrt{N}}\right]^{K-1}} W_N^{J}\left( \prod_{\eta\in V_{K-1}} c^{|\eta|} f_1 \circ T^{h \cdot \eta}\right)  \right)^{1/2^{J+K-2}} \\
&= \frac{1+ C_{J, 1, a}^{1/2^{K-1}}}{\floor{\sqrt{N}}^{1/2^{J+K-2}}}+ C_{J, 1, a}^{1/2^{K-1}}|a_1|^{1/2^{J+K-2}}\left(  W_N^{J+K-1}\left( f_1 \right)  \right)^{1/2^{J+K-2}} \, .
\end{align*}
Consolidating constants gives the desired bound.

\end{proof}

\begin{proof}[Proof of Theorem \ref{T:mult hilb trans}]
With the previous lemmas established, we may proceed exactly as in Theorem \ref{T:polynomial return times hilbert trans}.

Let $f_1$ satisfy $W_N^{K+J-1}(f_1) \leq CN^{-\alpha}$ where $0 < \alpha < 1/2$, and pick $\ep > 0$ be such that 
$$\sigma := 1- \alpha/2^{k+J-1} + \ep \leq 1 \, .$$ 
Using lemma \ref{T:intermediate 2}, we observe that 
\begin{align}\label{E:bound of F by polynomial}
\int F_N^{J, K, a}(f)(x) \, d\mu(x) &\leq C'_{J, K, a} \left( \frac{1}{\floor{\sqrt{N}}^{1/2^{J+K-2}}} + \left[ \frac{C}{N^\alpha}\right]^{1/2^{J+K-2}}\right) \leq \frac{C'}{N^{\alpha/2^{J+K-2}}}
\end{align}
for sufficiently large $N$ and a potentially larger constant $C'$ with the same dependencies. Hence, by the monotone convergence theorem, we observe that
\begin{align*}
\int \sum_{N=1}^\infty \frac{F_N^{J, K, a}(f)(x)}{N^\sigma} \, d\mu(x) &\leq  \sum_{N=1}^\infty \frac{\int F_N^{J, K, a}(f)(x) \, d\mu(x)}{N^\sigma} \, . 
\end{align*}
Since these terms are eventually bounded by $C/N^{1 + \alpha/2^{J+K-1} + \ep}$, this sum convergence. Hence, the integrand is finite almost everywhere, and  there exists a set $X_{f_1, \dots, f_J, K}$ of full measure such that for all $x\in X_{f_1, \dots, f_J, K}$, the terms
$$\frac{F_{N}^{J, K, a}(f)(x)}{N^\sigma}$$
are summable. For all dynamical systems $(Y,  \mathcal{G},\nu, S)$ and $g_1, \dots, g_K\in L^\infty(\nu)$, all bounded by 1, using lemma \ref{L:Intermediate 1} and the monotone convergence theorem in $\nu$, it follows that for any such $x\in X_{f_1, \dots, f_J, K}$ we have
\begin{align*}
&\int \sum_{N=1}^\infty \frac{\left|\frac{1}{N} \sum_{n=1}^N  \prod_{k=1}^K g_k \circ S^{b_k n} \cdot \prod_{j=1}^J f_j(T^{a_jn}x) \right|}{N^\sigma} \\&\leq \sum_{N=1}^\infty \frac{\left\Vert\frac{1}{N} \sum_{n=1}^N  \prod_{k=1}^K g_k \circ S^{b_k n} \cdot \prod_{j=1}^J f_j(T^{a_jn}x) \right\Vert_{L^2(\nu)}}{N^\sigma} 
\\ &\leq \sum_{N=1}^\infty  \frac{C_{J, K, a} \cdot F_N^{J, K, a, b} (f, g)(x)}{N^\sigma} < \infty
\end{align*}
and the summands are finite $\nu$-a.e. This meets the first critera of lemma \ref{L:HilbertCriterion}.

To establish the second Hilbert transform convergence criteria, we see from the estimate (\ref{E:bound of F by polynomial}) that for sufficiently large $N$, we have
\begin{align*}
 \int  N^{1-\sigma} F_N^{J, K, a}(f) (x) \, d\mu(x) \leq \frac{C'}{N^{\alpha/2^{J+K-1} + \ep}}
\end{align*}
and is summable over the subsequence $\floor{N^\gamma}$ for $\gamma := 2^{J+K-1}/\alpha$, and hence the sum is finite for almost all $x$. Again, by the monotone convergence theorem in $\nu$ and lemma \ref{L:Intermediate 1}, we have
\begin{align*}
&\int \sum_{N=1}^\infty \left|\frac{1}{\floor{N^\gamma}^\sigma} \sum_{n=1}^{\floor{N^\gamma}} \prod_{k=1}^K g_k\circ S^{b_k n}\cdot \prod_{j=1}^J f_j(T^{a_j n}x)\right| \, d\nu \\
&\leq \sum_{N=1}^\infty   \floor{N^\gamma}^{1-\sigma}\int \left|\frac{1}{\floor{N^\gamma}} \sum_{n=1}^{\floor{N^\gamma}} \prod_{k=1}^K g_k\circ S^{b_k n}\cdot \prod_{j=1}^J f_j(T^{a_j n}x)\right| \, d\nu \\
&\leq  \sum_{N=1}^\infty  \floor{N^\gamma}^{1-\sigma} \cdot C_{J, K, a} \cdot F_{\floor{N^\gamma}}^{J, K, a} (f)(x) <\infty \, .
\end{align*}
Hence, for any dynamical systems $(Y, \nu, \mathcal{G}, S)$ and $g_1, \dots, g_K\in L^\infty(\nu)$, all bounded by 1, the above integrand is finite $\nu$ a.e., and the summands 
\begin{align*}
\frac{1}{\floor{N^\gamma}^\sigma} \sum_{n=1}^{\floor{N^\gamma}} \prod_{k=1}^K g_k( S^{b_k n}y)\cdot \prod_{j=1}^J f_j(T^{a_j n}x)
\end{align*}
are converging to zero for $\nu$-a.e. $y\in Y$. By our choice of $\gamma$, we may extend from convergence along $\floor{N^\gamma}$ to convergence along $N$ by the usual argument, as see in the proof of Theorem \ref{T:HilbertTrans}. Hence, the second criteria of lemma \ref{L:HilbertCriterion} is met $\nu$-a.e.

Taking $X_{f_1, \dots, f_J, K}$ to be the set where both of these parts hold, it follows by the lemma \ref{L:HilbertCriterion} that for any $x$ in this set, the sums
\begin{align*}
\sum_{n=1}^\infty \frac{\prod_{k=1}^K g_{k} (S^{b_k n} y) \prod_{j=1}^J f_{j} (T^{a_j n} x)}{n^\sigma}
\end{align*}
are converging $\nu$-a.e. for any dynamical system $(Y, \mathcal{G}, \nu, S)$ and $g_1, \dots, g_K\in L^\infty(\nu)$, all bounded by 1, and the theorem is established after taking $\ep = 1/m$ for sufficiently large $m$ and intersecting.
\end{proof}

\appendix
\section{Proof of Generalized reverse Bourgain bound}\label{A:general case}

\begin{theorem}\label{T:general case}
Let $k\in \N$. There exists a constant $C_k'$ such that for all $N\in \N$ and $H_1, \dots, H_k \in \N$ with $H_k \leq N$, we have
\begin{align*}
&\frac{1}{\prod_{m=1}^{k-1} H_m} \sum_{h_1 = 1}^{H_1} \dots \sum_{h_{k-1}=1}^{H_{k-1}} \left\Vert\sup_t \left| \frac{1}{N} \sum_{n=1}^N e^{2\pi i n t} \left[ \prod_{\eta \in V_{k-1}} c^{|\eta|} g_\eta \circ T^{h \cdot \eta}\right] \circ T^n\right| \right\Vert_1^{2/3} \\ &\leq C_k' \left(\frac{1}{H_k^{1/3}} + \frac{H_{(1)}^{1/6}}{N^{1/6}} + \left[\frac{1}{H_{(2)}}\sum_{p=1}^{H_{(1)}} M_p^k (g_\mathbf{1}) \right]^{1/6} + \left(\frac{H_{(2)} - H_{(1)} + 1}{H_{(2)}}\right)^{1/6} M_{H_{(1)}}^k (g_\mathbf{1})^{1/6} \right)
\end{align*}
where $H_{(1)}$ and $H_{(2)}$ denote the two smallest (not necessarily distinct) values of $H_1, \dots, H_k$.
\end{theorem}

We begin with the proofs from the previous section that follow directly from this theorem:

\begin{proof}[Proof of Theorem \ref{T:reverse BB, simplest case}]
Take all $g_\eta = f$, take $H_1, \dots, H_{k-1} = \floor{\sqrt{N}}$ and $H_k = \floor{N^{1/4}}$, in which $H_{(1)} = \floor{N^{1/4}}$ and $H_{(2)} = \floor{\sqrt{N}}$. Applying the general case yields
\begin{align*}
W_N^k(f) &\leq C'_k\left( \frac{1}{N^{1/12}} + \frac{1}{N^{1/12}} + \left[\frac{1}{N^{1/2}} \sum_{p=1}^{\floor{N^{1/4}}} M_p^k(f)\right]^{1/6} + \left(\frac{\floor{\sqrt{N}} - \floor{N^{1/4}} + 1}{\floor{\sqrt{N}}}\right)^{1/6} M_{\floor{N^{1/4}}}^k (g_\mathbf{1})^{1/6} \right)
\end{align*}
for all $N$. Bounding the $p$ summands trivially, the average over $p$ decays at the order $1 / N^{1/4}$, yielding a remainder term $1/N^{1/24}$ that absorbs all others. On the last term, we bound the multiplier by 1 to get the desired estimate.
\end{proof}

\begin{proof}[Proof of Theorem \ref{L:RBB for off-diag}]
Take $H_1, \dots, H_{k-1} = \floor{\sqrt{N}}$ and $H_k = \floor{N^{1/4}}$, in which $H_{(1)} = \floor{N^{1/4}}$, and apply the same argument as the previous proof.
\end{proof}

\begin{proof}[Proof of Theorem \ref{T:RBB for alt WW averages}]
For Theorem \ref{A:general case}, take each $g_\eta = f$, and for $1\leq m \leq k-1$, take $H_{k-1} = r_{k-1}(N)$, and let $H_k = \floor{\min\{r_1(N), \dots, r_{k-1}(N), N\}^{1/2}}$, in which $H_{(1)} = \floor{\min\{r_1(N), \dots, r_{k-1}(N), N\}^{1/2}}$ and $H_{(2)} = \min\{r_1(N), \dots, r_{k-1}(N)\}$. The left-hand side becomes exactly $\widetilde W_N^k(f)$. 

For the terms on the right-hand side, the  $1/H_k^{1/3}$ and $H_{(1)}^{1/6}/N^{1/6}$ can both be bound above by $1/N^{1/12}$. Bounding every $p$ summand trivally, that term is again bounded by the ratio $H_{(1)}/H_{(2)}$, which in this case can be bounded by $1/\min\{r_1(N), \dots, r_{k-1}(N)\}^{1/2}$. Raising this to the 1/6 power, these two remainders can be combined and we get the desired estimate.
\end{proof}

The main difficulty in proving \ref{T:general case} over the $k=1$ case is exchanging the order of the following sums:
\begin{align*}
    \sum_{n=0}^q \sum_{h_1 = 1-n}^{H_1-n} \dots \sum_{h_k = 1-n}^{H_k-n} \, .
\end{align*}
To this end, we will present the required estimates as lemmas before proving this theorem. To build up to these, we consider the following notations.

Let $k \geq 1$ and $H_1, \dots, H_k \geq 1$, and $q\in \mathbb{Z}$ be such that $H_{(1)} < q \leq N-1$, where $H_{(1)}$ is the smallest value of $H_m$ for any $m$. For each $n$ (which we eventually take to be between $0$ and $q$), we define the $k$ dimensional rectangle
$$\Box_n = \{ h\in \Z^{k} : 1 - n \leq h_m \leq H_m - n \text{ for each }1 \leq m \leq k\}\, .$$
Hence, the desired sum can be written
\begin{align*}
    \sum_{n=0}^q \sum_{h\in \Box_n} \, .
\end{align*}
Now, we make the following observations about these rectangles. Since each dimension is an interval, the sets $\Box_n$ mostly behave like intervals.
For example, the intersection of any collection of $\Box_n$'s only depends on the smallest and largest index: if $a_1 \leq a_2 \leq \dots \leq a_l$ is any sequence of integers, then
\begin{align}\label{O:int of rectangles}
\Box_{a_1} \cap \Box_{a_2} \cap \dots \cap \Box_{a_l} = \Box_{a_1}  \cap \Box_{a_l}
\end{align}
Also similarly to intervals, the intersection of two rectangles is a smaller rectangle, and if two rectangles are too far apart, then they will not intersect at all. This distance is determined by the smallest width of the rectangles, which is $H_{(1)}$. Specifically, we have
$$\#( \Box_a \cap \Box_b ) = \begin{cases} \prod_{m=1}^k (H_m - (b-a)) & b-a < H_{(1)} \\ 0 & b - a \geq H_{(1)} \end{cases}$$
which we could also write as
\begin{align}\label{O:box size}
\#( \Box_a \cap \Box_b ) = \prod_{m=1}^k \max\{H_m - (b - a), 0\} \, .
\end{align}

Based on these observations, we see that for any $h\in \mathbb{Z}^k$, the collection of $i\in \mathbb{Z}$ such that $h \in \Box_i$ forms an interval, which we can denote the bounds of as $L_h$ and $U_h$, and we see that the width of this interval $U_n - L_h +1$ is bounded by $H_{(1)}$. If we define the set
$$\Gamma_q = \bigcup_{n=0}^q \Box_n$$
then we have the notation to exchange the desired sums:
\begin{align*}
    \sum_{n=0}^q \sum_{h\in \Box_n} = \sum_{h \in \Gamma_q} \sum_{n= L_h}^{U_h} \, .
\end{align*}
In the proof of Theorem \ref{T:general case}, we wish to group the terms $h$ from $\Gamma_q$ together by the width of the interval from $L_h$ to $U_h$, or by the quantity $U_h - L_h + 1$ which ranges from $1$ to $H_{(1)}$. Hence, we count the following:
\begin{lemma}\label{L:exact count for gamma}
For all terms as defined in the preceding paragraphs, we have for each $1 \leq p \leq H_{(1)}$ that
\begin{align}\label{F:exact count for gamma}
\begin{split}
&\#\{h\in \Gamma_q : U_h - L_h +1 = p\} \\ &= 2 \left[ \prod_{m=1}^k \max\{H_m - p + 1, 0\} - \prod_{m=1}^k \max\{H_m - p, 0\}\right] \\
&+ (q - p) \left[ \prod_{m=1}^k \max\{H_m - p+1, 0\} - 2\prod_{m=1}^k \max\{H_m - p,0 \} + \prod_{m=1}^k \max\{H_m - p - 1, 0\}\right]
\end{split}
\end{align}
\end{lemma}
We note that the above terms could be combined, but we leave it in the above for the purpose of cancellations which we later analyze.

\begin{proof}
Fix $1 \leq p \leq H_{(1)}$. In order for any $h \in \Gamma_q$ to lie in exactly $p$ rectangles, it must be the case that the bounds $L_h$ and $U_h$ are equal to $i$ and $i+p-1$ for some $0 \leq i \leq q - p + 1$. Hence, for each such $i$ we count the number of $h$ such that $L_h = i$ and $U_h = i + p-1$. By the interval properties we observed (\ref{O:int of rectangles}), this condition can be written as
$$h \in \Box_{i-1}^c \cap \Box_{i} \cap \Box_{i+p-1} \cap \Box_{i + p}^c$$
where the $c$ denotes set complements and at the endpoints, $\Box_{-1}$ and $\Box_{q +1}$ are taken to be empty. Hence, we need to count the size of this set for each $i$ and add them up.

To illustrate this method, we begin with $i=0$, in which we wish to count
$$\# (\Box_0 \cap \Box_{p-1} \cap \Box_{p}^c)\, . $$
We observe that this set can be written as a set difference, which we can count by $\#(A - B) = \#A - \#(A \cap B)$. Hence,
\begin{align*}
\# (\Box_0 \cap \Box_{p-1} \cap \Box_{p}^c) &= \# (\Box_0 \cap \Box_{p-1} - \Box_{p}) \\
&= \# (\Box_0 \cap \Box_{p-1}) - \# (\Box_0 \cap \Box_{p-1} \cap \Box_{p}) \\
&= \# (\Box_0 \cap \Box_{p-1}) - \# (\Box_0 \cap \Box_{p}) \\
&= \prod_{m=1}^k \max\{H_m - p + 1, 0\} - \prod_{m=1}^k \max\{H_m - p, 0\}
\end{align*}
using the previous observation (\ref{O:box size}) on the exact size of intersections. By the same argument, the other edge case of $i = q - p+1$ also satisfies
\begin{align*}
\# (\Box_{q-p}^c \cap \Box_{q-p+1} \cap \Box_{q}) &= \prod_{m=1}^k \max\{H_m - p + 1, 0\} - \prod_{m=1}^k \max\{H_m - p, 0\} \, .
\end{align*}
For values $1 \leq i \leq q - p$, we apply the same method, with inclusion-exclusion as appropriate to reformulate everything in terms of intersections, which we can simplify:
\begin{align*}
&\# (\Box_{i-1}^c \cap \Box_i \cap \Box_{i+p-1} \cap \Box_{i + p}^c) \\
&= \# (\Box_i \cap \Box_{i+p-1} - (\Box_{i-1} \cup \Box_{i + p})) \\
&=  \# (\Box_i \cap \Box_{i+p-1}) - \#((\Box_i \cap \Box_{i+p-1}) \cap (\Box_{i-1} \cup \Box_{i + p})) \\
&=  \# (\Box_i \cap \Box_{i+p-1}) - \#(( \Box_{i-1} \cap \Box_i \cap \Box_{i+p-1} ) \cup ( \Box_i \cap \Box_{i+p-1}\cap \Box_{i + p})) \\
&=  \# (\Box_i \cap \Box_{i+p-1}) - \#( \Box_{i-1} \cap \Box_i \cap \Box_{i+p-1} ) - \#(\Box_i \cap \Box_{i+p-1}\cap \Box_{i + p}) + \#(\Box_{i-1} \cap \Box_i \cap \Box_{i+p-1} \cap \Box_{i + p}) \\
&=  \# (\Box_i \cap \Box_{i+p-1}) - \#( \Box_{i-1} \cap \Box_{i+p-1} ) - \#(\Box_i \cap  \Box_{i + p}) + \#(\Box_{i-1} \cap \Box_{i + p}) \\
&= \prod_{m=1}^k \max\{H_m - p+1, 0\} - 2\prod_{m=1}^k \max\{H_m - p,0 \} + \prod_{m=1}^k \max\{H_m - p - 1, 0\} \, .
\end{align*}
We observe that we have lost $i$ dependence. Hence, adding these terms together the appropriate number of times gives us a bound on the desired term.
\end{proof}

To use this count to get meaningful estimates, we make some observations about the kinds of cancellation this formula (\ref{F:exact count for gamma}) may afford. Consider the polynomial in $x_1, \dots, x_k$ given by
$$\prod_{m=1}^k (x_m + 1)\, .$$
Factoring this term out, every possible monomial appears exactly once. Hence when we subtract off the product of all $x_m$'s, we are left with
\begin{align}\label{O:cancellation 1}
\prod_{m=1}^k (x_m + 1) - \prod_{m=1}^k x_m = \sum_{\eta \in V_{k} - \{ \mathbf{1}\}} \prod_{\eta_{i} = 1} x_{i}
\end{align}
using previous notation. By the same reasoning, expanding the expression
$$\prod_{m=1}^k (x_m - 1)$$
gives every possible monomial with parity. Hence, when we add them together, there is cancellation:
\begin{align}\label{O:cancellation 2}
\prod_{m=1}^k (x_m + 1) - 2\prod_{m=1}^k x_m + \prod_{m=1}^k (x_m - 1) &= 2\sum_{\overset{\eta \in V_k - \{\mathbf{1}\}}{k - |\eta| \text{ even}}} \prod_{\eta_{i} = 1} x_{i} \, .
\end{align}
These observations can clean up the formula (\ref{F:exact count for gamma}) from the previous lemma, but only once we are able to remove the $\max\{ \cdot, 0 \}$ operators. Hence, we will get different estimates on proportion depending on the size of $p$ and the number of terms that vanish. When no terms vanish, we get the following:
\begin{lemma}\label{L:est for small p}
Let $1 \leq p \leq H_{(1)} - 1$. Then 
\begin{align}\label{E:est for small p}
\frac{\#\{h\in \Gamma_q : U_h - L_h +1 = p\}}{N \cdot \prod_{m=1}^k H_m} \leq 2^{k+1}\left(\frac{1}{N \cdot H_{(1)}} + \frac{1}{{H_{(1)}} \cdot H_{(2)}}\right)
\end{align}
where $H_{(1)}$ and $H_{(2)}$ denote the two smallest (not necessarily distinct) values of $H_m$.
\end{lemma}
\begin{proof}
Consider first the case where $1 \leq p \leq H_{(1)} - 2$. Based on the range of $p$, all of the maximum operators from the formula (\ref{F:exact count for gamma}) vanish, and we apply our observations (\ref{O:cancellation 1}) and (\ref{O:cancellation 2}) concerning expanding these expressions:
\begin{align*}
&\#\{h\in \Gamma_q : U_h - L_h +1 = p\} \\ &= 2 \left[ \prod_{m=1}^k (H_m - p + 1) - \prod_{m=1}^k (H_m - p)\right] \\
&+ (q - p) \left[ \prod_{m=1}^k (H_m - p+1) - 2\prod_{m=1}^k (H_m - p)  + \prod_{m=1}^k (H_m - p - 1)\right] \\
&= 2 \left[ \sum_{\eta \in V_{k} - \{ \mathbf{1}\}} \prod_{\eta_{i} = 1} (H_i - p)\right] + 2(q - p) \left[\sum_{\overset{\eta \in V_k - \{\mathbf{1}\}}{k - |\eta| \text{ even}}} \prod_{\eta_{i} = 1} (H_i - p) \right]\, .
\end{align*}
After we divide through by $\prod_{m=1}^k H_m$ (and also $N$), we see that for every $\eta$ summand, the pieces $(H_i - p)/H_i$ for $\eta_i = 1$ can be bound away, and we will only be left with the $H_i$ for $\eta_i = 0$ terms in the denominator:
\begin{align*}
&\frac{\#\{h\in \Gamma_q : U_h - L_h +1 = p\}}{N \cdot \prod_{m=1}^k H_m} \leq \frac{2}{N} \sum_{\eta \in V_{k} - \{ \mathbf{1}\}} \frac{1}{\prod_{\eta_{i} = 0} H_i}  + 2\left(\frac{q - p}{N}\right) \sum_{\overset{\eta \in V_k - \{\mathbf{1}\}}{k - |\eta| \text{ even}}} \frac{1}{\prod_{\eta_{i} = 0} H_i }\, .
\end{align*}
In the first sum on the right hand side, we see that at least one $H_i$ terms appears in every summand, so we can bound them all above by $1/H_{(1)}$. In the second sum, we see that at least two $H_i$'s appear in every summand, and we can bound them all above by $1/({H_{(1)}} \cdot H_{(2)})$ for the two smallest possible values. In both cases, the total number of summands is less than $2^k$, so after bounding $q$ by $N$ and consolidating constants, we get the desired estimate.

Now consider the case $p = H_{(1)} - 1$. Here, the formula (\ref{F:exact count for gamma}) simplifies to
\begin{align*}
\#\{h\in \Gamma_q : U_h - L_h +1 = H_{(1)}-1\} &= 2 \left[ \prod_{m=1}^k (H_m - H_{(1)} + 2) - \prod_{m=1}^k (H_m - H_{(1)} + 1)\right] \\
&+ (q - p) \left[ \prod_{m=1}^k (H_m - H_{(1)} + 2) - 2\prod_{m=1}^k (H_m - H_{(1)}+1)\right]\, .
\end{align*}
Consider the terms on the second line. In the first product, the $m$ term which minimizes $H_m$ yields a 2 in the overall product, which we can factor out in tandem with the 2 attached to the second product. If we let the punctured product $\prod '$ correspond to deleting the $m$ term which minimizes $H_m$, this can be rewritten as
\begin{align*}
\#\{h\in \Gamma_q : U_h - L_h +1 = H_{(1)}-1\} &= 2 \left[ \prod_{m=1}^k (H_m - H_{(1)} + 2) - \prod_{m=1}^k (H_m - H_{(1)} + 1)\right] \\
&+ 2(q - p) \left[ \prod_{m=1}^k\mathop{}^{\mkern-5mu '} (H_m - H_{(1)} + 2) - \prod_{m=1}^k\mathop{}^{\mkern-5mu '} (H_m - H_{(1)}+1)\right] \, .
\end{align*}

After dividing by $N \cdot \prod_{m=1}^k H_m$, the first set of brackets can be dealt with the same way as before to yield a term of the order $1/(N \cdot H_{(1)})$. In the second set of brackets, we can expand by the same observations, only on $k-1$ terms rather than $k$. Relative to this part, the smallest value of $H_m$ is now $H_{(2)}$, and overall this piece yields a term of the order $1/({H_{(1)}} \cdot H_{(2)})$, and we get the same overall estimate.
\end{proof}

When $p$ is as big as possible, most terms from the formula (\ref{F:exact count for gamma}) vanish and we can get the following:
\begin{lemma}\label{L:est for big p}
Using the same notation as before, we have
\begin{align}\label{E:est for big p}
\frac{\#\{h\in \Gamma_q : U_h - L_h +1 = H_{(1)}\}}{N \cdot \prod_{m=1}^k H_m} \leq 2\left(\frac{H_{(2)} - H_{(1)} + 1}{H_{(1)}\cdot H_{(2)}}\right) \, .
\end{align}
\end{lemma}
We leave the $H_{(2)}$ term unbounded to note that if $H_{(1)}$ and $H_{(2)}$ are essentially equal, we can obtain a much better estimate.

\begin{proof}
When $p = H_{(1)}$, we see that the later terms in the formula (\ref{F:exact count for gamma}) are zero, in which
$$\#\{h\in \Gamma_q : U_h - L_h +1 = H_{(1)}\} = (q - H_{(1)} + 2)\prod_{m=1}^k (H_m - H_{(1)} + 1)\, .$$
Dividing through, we see that
\begin{align*}
\frac{\#\{h\in \Gamma_q : U_h - L_h +1 = H_{(1)}\}}{N \cdot \prod_{m=1}^k H_m} &= \left( \frac{q - H_{(1)} + 2}{N}\right) \prod_{m=1}^k  \left(\frac{H_m - H_{(1)} + 1}{H_m}\right) \\
&\leq 2 \left( \frac{1}{H_{(1)}} \right) \left( \frac{H_{(2)} - H_{(1)} + 1}{H_{(2)}}\right)
\end{align*}
bounding every other term away.
\end{proof}

With the estimates of lemmas \ref{L:est for small p} and \ref{L:est for big p}, we prove the general case of Theorem \ref{T:general case}.

\begin{proof}[Proof of Theorem \ref{T:general case}.] 

In this proof, we allow the constant $C'_k$ to change from line to line, while only ever picking up dependence on $k$. In fact, the only place in which $k$ dependence is picked up is from applying the estimate (\ref{E:est for big p}) under what is ultimately a $1/6$-th power. Hence, the final constant $C'_k$ presented here will look like $2^{k/6}$, times an absolute constant.

To begin, we ignore the $1/3$ power on the desired average, and add it back at the end by an application of Hölder's inequality. On the desired averages, we apply the Van der Corput inequality \ref{vdc-lem} pointwise for $1 \leq H_k \leq N$, bounding the remainder trivially and extending the $n$ sum to $N$ as usual, with indexing variable $h_k$:
\begin{align*}
&\frac{1}{\prod_{m=1}^{k-1} H_m} \sum_{h_1 = 1}^{H_1} \dots \sum_{h_{k-1}=1}^{H_{k-1}} \left\Vert\sup_t \left| \frac{1}{N} \sum_{n=1}^N e^{2\pi i n t} \left[ \prod_{\eta \in V_{k-1}} c^{|\eta|} g_\eta \circ T^{h \cdot \eta}\right] \circ T^n\right| \right\Vert_2^{2} \\
&\leq \frac{1}{\prod_{m=1}^{k-1} H_m} \sum_{h_1 = 1}^{H_1} \dots \sum_{h_{k-1}=1}^{H_{k-1}} \left( \frac{2}{H_k} + \frac{4H_k}{N} + \frac{4}{H_k} \sum_{h_k = 1}^{H_k} \int \left| \frac{1}{N} \sum_{n=1}^N \prod_{\eta \in V_{k-1}} (c^{|\eta|} g_\eta \circ T^{h\cdot \eta}) \circ T^n \cdot (\overline{g_\eta \circ T^{h\cdot \eta}}) \circ T^{n + h_k}\right| \, d\mu \right) \\
&\leq \frac{2}{H_k} + \frac{4H_k}{N} + 4\left(\frac{1}{\prod_{m=1}^{k} H_m} \sum_{h_1 = 1}^{H_1} \dots \sum_{h_k = 1}^{H_k} \int \left| \frac{1}{N} \sum_{n=1}^N \prod_{\eta \in V_{k-1}} (c^{|\eta|} g_\eta \circ T^{h\cdot \eta}) \circ T^n \cdot (\overline{g_\eta \circ T^{h\cdot \eta}}) \circ T^{n + h_k}\right|^2 \, d\mu  \right)^{1/2}
\end{align*}
after using Hölder's inequality to group all of the sums together. 

As in the $k=1$ case, we wish to use lemma \ref{L:vdc for systems} on the inner integral. We see that the variable $h_k$ has already added another layer to the cube, and using this lemma will add another layer in the variable called $n$. To consolidate notation, this will move us from $\eta\in V_{k-1}$ to $\eta\in V_{k+1}$. The new layers in $h_k$ and $n$ do not correspond to different functions $g_\eta$. Hence, we denote $\eta''$ to remove the last two components of $\eta$, so that $\eta'' \in V_{k-1}$ and the labeling of $g_{\eta''}$ makes sense. Let $(h, n)$ denote the $k+1$ tuple $(h_1, \dots, h_k, n)$. After applying this lemma, we can write:
\begin{align*}
&\leq \frac{C'_k}{H_k} + \frac{C'_kH_k}{N} + C'_k\left(\frac{1}{\prod_{m=1}^{k} H_m} \sum_{h_1 = 1}^{H_1} \dots \sum_{h_{k}=1}^{H_{k}} \frac{1}{N} \sum_{n=0}^{N-1} \left( \frac{N-n}{N} \right) \Re \left[ \int \prod_{\eta\in V_{k+1}} c^{|\eta|} g_{\eta''} \circ T^{(h, n) \cdot \eta} \, d\mu \right]  \right)^{1/2}
\end{align*}
where we have begun to use $C'_k$ to keep track of the constants.
Recall the same expansion from the specific proof: that for any sequence $a_n$, by interchanging sums we have
$$\frac{1}{N}\sum_{n=0}^{N-1} \left( \frac{N-n}{N} \right) a_n = \frac{1}{N}\sum_{q=0}^{N-1} \frac{1}{N}\sum_{n=0}^q a_n\, .$$
We use this to expand the $n$ sum as a sum over $n$ and $q$. After doing so, we pull the real component and $q$ sum out and bound by the absolute value:
\begin{align*}
&= \frac{C'_k}{H_k} + \frac{C'_kH_k}{N} + C'_k\left(\frac{1}{ \prod_{m=1}^{k} H_m} \sum_{h_1 = 1}^{H_1} \dots \sum_{h_{k}=1}^{H_{k}} \frac{1}{N} \sum_{q=0}^{N-1} \frac{1}{N} \sum_{n=0}^{q}  \Re \left[ \int \prod_{\eta\in V_{k+1}} c^{|\eta|} g_{\eta''} \circ T^{(h, n) \cdot \eta} \, d\mu \right]  \right)^{1/2} \\
&\leq  \frac{C'_k}{H_k} + \frac{C'_kH_k}{N} + C'_k \left(\frac{1}{N} \sum_{q=0}^{N-1}  \left|\frac{1}{N \cdot \prod_{m=1}^{k} H_m} \sum_{n=0}^{q} \sum_{h_1 = 1}^{H_1} \dots \sum_{h_{k}=1}^{H_{k}} \int \prod_{\eta\in V_{k+1}} c^{|\eta|} g_{\eta''} \circ T^{(h, n) \cdot \eta} \, d\mu \right|  \right)^{1/2} \, .
\end{align*}
Recalling that $H_{(1)}$ denotes smallest value of $H_m$, consider cutting out the terms $q=0$ to $q = H_{(1)}-1$, and we can bound away trivially as a remainder term of order $H_{(1)}/N$ (we can actually pull out a remainder of better order, but it will eventually be absorbed into a term of this form). When we pull it out of the $1/2$th power by subadditivity, it may join the other remainder terms:

\begin{align*}
&\leq  \frac{C'_k}{H_k} + \frac{C'_k H_{(1)}^{1/2}}{N^{1/2}} + C'_k \left(\frac{1}{N} \sum_{q=H_{(1)}}^{N-1}  \left|\frac{1}{N \cdot \prod_{m=1}^{k} H_m} \sum_{n=0}^{q} \sum_{h_1 = 1}^{H_1} \dots \sum_{h_{k}=1}^{H_{k}}  \int \prod_{\eta\in V_{k+1}} c^{|\eta|} g_{\eta''} \circ T^{(h, n) \cdot \eta} \, d\mu \right|  \right)^{1/2} \, .
\end{align*}

Now, we wish to shift the index of each $h_m$ sum down by $n$, which causes the indices of the summands to increase by $n$. Briefly, we analyze the effect of this on the integrand alone. We shift every $h_m$ up by $n$, and group terms by $|\eta|$:
\begin{align*}
\int \prod_{\eta \in V_{k+1}} c^{|\eta|} g_{\eta''} \circ T^{(h_1 + n, \dots, h_k + n, n) \cdot \eta} \, d\mu &= \int \prod_{j=0}^{k+1} \prod_{\overset{\eta \in V_{k+1}}{|\eta| = j}} c^{|\eta|} g_{\eta''} \circ T^{(h_1, \dots, h_k, 0) \cdot \eta + jn} \, d\mu \\
&=\int \prod_{j=0}^{k+1} \prod_{\overset{\eta \in V_{k+1}}{|\eta| = j}} c^{|\eta|} g_{\eta''} \circ T^{(h_1, \dots, h_k, 0) \cdot (\eta - \textbf{1}) + jn} \, d\mu
\end{align*}
after translating the measure by $T^{h \cdot (-\textbf{1})}$, where as before $\textbf{1}$ denotes the $\eta\in V_{k+1}$ with $1$'s in every slot. Denoting
$$G^h_j := \prod_{\overset{\eta \in V_{k+1}}{|\eta| = j}} c^{|\eta|} g_{\eta''} \circ T^{(h_1, \dots, h_k, 0) \cdot (\eta- \textbf{1}) }$$
we see that the above integrand can be written as
\begin{align*}
\int \prod_{j=1}^{k+1} G^h_j \circ T^{jn} \, d\mu
\end{align*}
and we observe that $G^h_{k+1} = c^{|k+1|}g_{\textbf{1}''}$. Without loss of generality, we can take $G^h_{k+1} = g_{\textbf{1}''}$. For the sake of notation, we can remove the double apostrophe to take $\mathbf{1}$ to be in $V_{k-1}$, and write $g_{\mathbf{1}}$.

Returning to our main estimate, after this index shift we obtain
\begin{align*}
&\leq  \frac{C'_k}{H_k} + \frac{ H_{(1)}^{1/2}}{N} + C'_k \left(\frac{1}{N} \sum_{q=H_{(1)}}^{N-1}  \left|\frac{1}{N \cdot \prod_{m=1}^{k} H_m} \sum_{n=0}^{q} \sum_{h_1 = 1-n}^{H_1-n} \dots \sum_{h_{k}=1-n}^{H_{k}-n} \int \prod_{j=0}^{k+1} G^h_j \circ T^{jn} \, d\mu \right|  \right)^{1/2} \, .
\end{align*}
In the language of the previous lemmas, we can exchange the $h$ and $n$ sums. We shift the indices of the $n$ sum down, and shift the measure by $T^{-(k+1)(L_h -1)}$ to keep the $j=k+1$ term equal to $g_{\textbf{1}} \circ T^{(k+1)n}$:
\begin{align*}
&=  \frac{C'_k}{H_k} + \frac{C'_k H^{1/2}_{(1)}}{N^{1/2}} + C'_k \left(\frac{1}{N} \sum_{q=H_{(1)}}^{N-1}  \left|\frac{1}{N \cdot \prod_{m=1}^{k} H_m} \sum_{h \in \Gamma_q} \sum_{n=L_h}^{U_h}  \int \prod_{j=0}^{k+1} G^h_j \circ T^{jn} \, d\mu \right|  \right)^{1/2} \\
&=  \frac{C'_k}{H_k} + \frac{C'_k H_{(1)}^{1/2}}{N^{1/2}} + C'_k \left(\frac{1}{N} \sum_{q=H_{(1)}}^{N-1}  \left|\frac{1}{N \cdot \prod_{m=1}^{k} H_m} \sum_{h \in \Gamma_q} \sum_{n=1}^{U_h - L_h+1}  \int \prod_{j=0}^{k+1} G^h_j  \circ T^{(j - k -1)(L_h-1)} \circ T^{jn} \, d\mu \right|  \right)^{1/2} \, .
\end{align*}
We group together terms from $h$ by $p = U_h - L_h+1 $, which ranges from 1 to $H_{(1)}$. After pulling the absolute value inside, we can bound by a supremum to make the uniform $k+1$ recurrence averages appear:
\begin{align*} 
&= \frac{C'_k}{H_k} + \frac{C'_k H_{(1)}^{1/2}}{N^{1/2}} +C'_k \left(\frac{1}{N} \sum_{q=H_{(1)}}^{N-1}  \left|\frac{1}{N \cdot \prod_{m=1}^{k} H_m} \sum_{p = 1}^{H_{(1)}} \sum_{\overset{h \in \Gamma_q}{U_h - L_h +1= p}} \sum_{n=1}^{p}  \int \prod_{j=0}^{k+1} G^h_j  \circ T^{(j - k -1)(L_h - 1)} \circ T^{jn} \, d\mu \right|  \right)^{1/2} \\
&\leq \frac{C'_k}{H_k} + \frac{C'_k H_{(1)}^{1/2}}{N^{1/2}} +C'_k \cdot \\ &\left(\frac{1}{N} \sum_{q=H_{(1)}}^{N-1} \frac{1}{N \cdot \prod_{m=1}^{k} H_m} \sum_{p = 1}^{H_{(1)}} \sum_{\overset{h \in \Gamma_q}{U_h - L_h +1= p}} \int \left| G^h_0 \circ T^{(-k-1)(L_h-1)}\right|\cdot \left| \sum_{n=1}^{p} \prod_{j=1}^{k+1} G^h_j  \circ T^{(j - k -1)(L_h - 1)} \circ T^{jn} \, d\mu \right|  \right)^{1/2} \\
&\leq \frac{C'_k}{H_k} + \frac{C'_k H_{(1)}^{1/2}}{N^{1/2}} \\ &+C'_k \left(\frac{1}{N} \sum_{q=H_{(1)}}^{N-1} \sum_{p = 1}^{H_{(1)}} \frac{p}{N \cdot \prod_{m=1}^{k} H_m} \sum_{\overset{h \in \Gamma_q}{U_h - L_h +1= p}} \, \sup_{\overset{h_j \in L^\infty(\mu)}{\max_{1 \leq j \leq k} \Vert h_j \Vert_\infty \leq 1}}\int \left|\frac{1}{p} \sum_{n=1}^{p} \prod_{j=1}^{k+1} h_{j} \circ T^{jn} \cdot g_{\mathbf{1}''} \circ T^{(k+1)n} \, d\mu \right|  \right)^{1/2} \\
&= \frac{C'_k}{H_k} + \frac{C'_k H_{(1)}^{1/2}}{N} + C'_k\left(\frac{1}{N} \sum_{q=H_{(1)}}^{N-1} \sum_{p = 1}^{H_{(1)}} \frac{p}{N \cdot \prod_{m=1}^{k} H_m} \sum_{\overset{h \in \Gamma_q}{U_h - L_h +1 = p}} M_p^k (g_{\mathbf{1}})  \right)^{1/2} \, .
\end{align*}
With $h$ dependence lost, we can apply our estimates from lemmas \ref{L:est for small p} and \ref{L:est for big p} on the amount of $h$ with $U_h - L_h + 1 = p$. Recalling that we had a different estimate for $p = H_{(1)}$, we pull this term out entirely, in which the $q$ sum collapses:
\begin{align*}
= \frac{C'_k}{H_k} + \frac{C'_k H_{(1)}^{1/2}}{N^{1/2}} & +C'_k \left(\frac{1}{N} \sum_{q=H_{(1)}}^{N-1} \sum_{p = 1}^{H_{(1)}-1} \frac{p \cdot \#\{h \in \Gamma_q : U_h - L-h + 1 = p\}}{N \cdot \prod_{m=1}^{k} H_m}  M_p^k (g_{\mathbf{1}}) \right)^{1/2} \\ &+ C'_k \left(\frac{H_{(1)} \cdot \#\{h \in \Gamma_q : U_h - L-h + 1 = H_{(1)}\}}{N \cdot \prod_{m=1}^{k} H_m}  M_{H_{(1)}}^k (g_{\mathbf{1}}) \right)^{1/2} \\
\leq \frac{C'_k}{H_k} + \frac{C'_kH_{(1)}^{1/2}}{N^{1/2}} &+ C'_k \left(\frac{1}{N} \sum_{q=H_{(1)}}^{N-1} \sum_{p = 1}^{H_{(1)}-1} \left( \frac{p}{N \cdot H_{(1)}} + \frac{p}{H_{(1)} \cdot H_{(2)}}\right) M_p^k (g_{\mathbf{1}}) \right)^{1/2} \\ &+ C'_k \left(H_{(1)}\left(\frac{H_{(2)} - H_{(1)} + 1}{H_{(1)}\cdot H_{(2)}}\right)  M_{H_{(1)}}^k (g_{\mathbf{1}}) \right)^{1/2} \, .
\end{align*}

Bounding $p$ by $H_{(1)}$, the $1/N$ term can be bound trivially and pulled out as a remainder term of the order $H_{(1)}^{1/2}/N^{1/2}$. With $q$ dependence lost in the first sum, we clean up:
\begin{align*}
\leq \frac{C'_k}{H_k} + \frac{C'_k H_{(1)}^{1/2}}{N^{1/2}}  & + C'_k \left(\frac{1}{H_{(2)}} \sum_{p = 1}^{H_{(1)}-1}  M_p^k (g_{\mathbf{1}}) \right)^{1/2} + C'_k \left(\left(\frac{H_{(2)} - H_{(1)} + 1}{H_{(2)}}\right)  M_{H_{(1)}}^k (g_{\mathbf{1}}) \right)^{1/2}
\end{align*}
and after adjusting for the $1/3$ power, we get the desired estimate.
\end{proof}

\printbibliography

\end{document}